\documentclass[11pt, leqno]{amsart} 

\usepackage{graphicx} 
\usepackage{amsmath,amsthm,amssymb} 
\usepackage{amsfonts} 
\usepackage{dsfont}
\usepackage{mathrsfs} 
\usepackage[T1]{fontenc} 
\usepackage[utf8]{inputenc} 

\usepackage[dvipsnames, svgnames]{xcolor} 
\usepackage[shortlabels]{enumitem} 
\usepackage[colorlinks,citecolor=citec,hypertexnames=false,urlcolor=urlc,breaklinks=true, pagebackref]{hyperref} 

\usepackage{accents} 

\usepackage{amsbsy} 
\usepackage{amscd} 
\usepackage{latexsym} 
\usepackage{txfonts} 

\usepackage{exscale} 
\usepackage{bbm} 
\usepackage{mathtools} 
\usepackage{bbding}

\usepackage{geometry} 

\makeatletter
\providecommand\@dotsep{4.5}
\makeatother

\usepackage{thmtools} 

\newif\ifsoldark
\newif\ifsollight
\newif\ifclassic
\newif\ifplain

\usepackage{etoolbox}
\appto\appendix{\addtocontents{toc}{\protect\setcounter{tocdepth}{1}}}

\renewcommand*\backref[1]{\ifx#1\relax \else (p. #1) \fi} 

\numberwithin{equation}{section} 

\theoremstyle{plain}
\newtheorem{theorem}[equation]{Theorem}
\newtheorem{lemma}[equation]{Lemma}
\newtheorem{corollary}[equation]{Corollary}
\newtheorem{proposition}[equation]{Proposition}

\theoremstyle{definition}
\newtheorem{definition}[equation]{Definition}

\theoremstyle{remark}
\newtheorem{remark}[equation]{Remark}

\newtheorem{claim*}{Claim}

\newcommand{\dist}{\operatorname{dist}}
\newcommand{\card}{\operatorname{card}}
 
\newcommand{\dv}{\operatorname{div}}
\newcommand{\Di}{\operatorname{D}}
\newcommand{\Reg}{\operatorname{R}}
\newcommand{\n}[1]{\mathscr{#1}}
\newcommand{\m}[1]{\mathcal{#1}}

\newcommand{\bb}[1]{\mathbb{#1}}

\newcommand{\medcup}{\textstyle\bigcup}

\newcommand{\RNum}[1]{\uppercase\expandafter{\romannumeral #1\relax}}

\newcommand{\Lip}{\operatorname{Lip}}

\DeclareMathOperator{\supp}{supp}
\DeclareMathOperator{\diam}{diam}

\DeclareMathOperator*{\esssup}{ess\,sup}  

\def\div{\mathop{\operatorname{div}}\nolimits}

\def\Xint#1{\mathchoice
	{\XXint\displaystyle\textstyle{#1}}%
	{\XXint\textstyle\scriptstyle{#1}}%
	{\XXint\scriptstyle\scriptscriptstyle{#1}}%
	{\XXint\scriptscriptstyle\scriptscriptstyle{#1}}%
	\!\int}
\def\XXint#1#2#3{{\setbox0=\hbox{$#1{#2#3}{\int}$}
		\vcenter{\hbox{$#2#3$}}\kern-.5\wd0}}

\def\dashint{\Xint-}

\def\YYint#1#2#3{{\setbox0=\hbox{$#1{#2#3}{\iint}$}
		\vcenter{\hbox{$#2#3$}}\kern-.51\wd0}}
 
\newcommand{\ep}{\varepsilon}
\newcommand{\ra}{\rightarrow}
\newcommand{\lra}{\longrightarrow}   
\newcommand{\loc}{\operatorname{loc}}

\newcommand{\sss}{{\rm {Stop}}}
\newcommand{\ttt}{{\rm {Top}}}
\newcommand{\Ins}{{\rm {Ins}}}
\newcommand{\Bdry}{{\rm {Bdry}}}

\newcommand{\tree}{{\rm Tree}}

\newcommand{\wh}[1]{{\widehat{#1}}}

\newcommand{\R}{\mathbb{R}}

\newcommand{\wt}{\widetilde}

\newcommand{\cF}{{\mathcal  F}}

\newcommand{\cB}{{\mathcal  B}}

\newcommand{\1}{{\mathds 1}}

\def\HH{\mathcal{H}}

\def\R{\mathbb{R}}
\def\pom{{\partial\Omega}}

\newcommand{\DD}{{\mathcal D}}

\newcommand{\rf}[1]{{(\ref{#1})}}

\def\ve{\varepsilon} 
\def\vphi{\varphi}
   
\def\cM{{\mathcal{M}}}

\def\avint{\Xint-}

\newcommand{\WW}{\m W}

\newcommand*{\dt}[1]{%
	\accentset{\mbox{\Large\bfseries .}}{#1}}

\allowdisplaybreaks

\setlist{nosep}

\colorlet{citec}{blue}
\colorlet{urlc}{blue}
\colorlet{toc}{blue}
\colorlet{hyperc}{blue}
\colorlet{bpcolor}{NavyBlue}
\colorlet{smcolor}{purple}
\colorlet{impcolor}{blue}
\colorlet{eqcolor}{black}
\colorlet{lmcolor}{black}
\colorlet{propcolor}{black}
\colorlet{thmcolor}{black}
\colorlet{defcolor}{black}
\colorlet{rmcolor}{black}
\colorlet{excolor}{black}

\begin{document}
	
\author[M. Mourgoglou]{Mihalis Mourgoglou}
\email[]{michail.mourgoglou@ehu.eus}
\address{Departamento de Matemáticas, Universidad del País Vasco, UPV/EHU, Barrio Sarriena S/N 48940 Leioa, Spain and, IKERBASQUE, Basque Foundation for Science, Bilbao, Spain}

\author[B. Poggi]{Bruno Poggi}
\email[]{brunopoggi@pitt.edu}
\address{Department of Mathematics, University of Pittsburgh, Pittsburgh, PA}

\author[X. Tolsa]{Xavier Tolsa}
\email[]{xavier.tolsa@uab.cat}
\address{ICREA, Barcelona, Departament de  Matemàtiques, Universitat Autònoma de Barcelona, and Centre de Recerca Matemàtica, Barcelona, Catalonia}

\title[$L^p$-solvability of the Poisson-Dirichlet problem]{Solvability of the Poisson-Dirichlet problem with interior data  in $L^{p'}$-Carleson spaces and its applications to the $L^{p}$-regularity problem}
\date{\today}

\thanks{M.M. was supported  by IKERBASQUE and partially supported by the grant PID2020-118986GB-I00 of the Ministerio de Econom\'ia y Competitividad (Spain), and by  IT-1615-22 (Basque Government). B.P. and X.T. are supported by the European Research Council (ERC) under the European Union's Horizon 2020 research and innovation programme (grant agreement 101018680), and partially supported by the grant 2021-SGR-00071 (Catalonia). X.T. is also partially supported by MICINN (Spain) under  grant PID2020-114167GB-I00 and the María de Maeztu Program for units of excellence (Spain) (CEX2020-001084-M). This material is based upon work funded by the  Deutsche Forschungsgemeinschaft (DFG, German Research Foundation) under Germany's Excellence Strategy – EXC-2047/1 – 390685813, while the authors were in residence at the Hausdorff Research Institute in Spring 2022 during the program ``Interactions between geometric measure theory, singular integrals, and PDEs''.
}
 
\begin{abstract} We prove that the $L^{p'}$-solvability of the homogeneous Dirichlet problem for an elliptic operator $L=-\dv A\nabla$ with    real and  merely bounded coefficients  is equivalent to the $L^{p'}$-solvability  of  the   Poisson Dirichlet problem $Lw=H-\dv F$,   which is defined in terms of an $L^{p'}$ estimate on the non-tangential maximal function,  assuming that $\dist(\cdot, \partial \Omega) H$ and $F$ lie in certain $L^{p'}$-Carleson-type spaces, and that the domain $\Omega\subset\bb R^{n+1}$, $n\geq2$, satisfies the corkscrew condition and has $n$-Ahlfors regular boundary. In  turn, we use this result to show that, in a bounded domain with uniformly $n$-rectifiable boundary that satisfies the corkscrew condition, $L^{p'}$-solvability of the homogeneous Dirichlet problem  for an operator $L=-\dv A\nabla$  satisfying the Dahlberg-Kenig-Pipher condition (of arbitrarily large constant) implies solvability of  the $L^p$-regularity problem for the adjoint operator $L^*=-\dv A^T \nabla$, where $1/p+1/p'=1$ and $A^T$ is the transpose matrix of $A$.   This result for Dahlberg-Kenig-Pipher operators is new even if $\Omega$ is the unit ball, despite the fact that the $L^{p'}$-solvability of the Dirichlet problem for these operators in Lipschitz domains has been known since 2001.
	
Further novel applications include i)  new local estimates for the Green's function and its gradient in rough domains, ii)  a local $T1$-type theorem for the $L^{p}$-solvability of the ``Poisson-Regularity problem'', itself equivalent to the $L^{p'}$-solvability of the homogeneous Dirichlet problem, in terms of certain gradient estimates for local landscape functions, and iii) new $L^p$ estimates for the eigenfunctions (and their gradients) of symmetric operators $L$ on bounded rough domains. 
\end{abstract} 

\maketitle

\hypersetup{linkcolor=hyperc}

	\tableofcontents

\section{Introduction}\label{sec.intro}

The $L^p$-solvability of homogeneous boundary value problems for second-order linear  elliptic partial differential equations   with non-smooth coefficients  on rough domains  has been an area of intense and fruitful research in the last few decades, culminating in shocking equivalences between certain geometric properties of domains   in the Euclidean space  and $L^p$-solvability of the homogeneous Dirichlet problem   for the Laplacian. However, the   literature regarding the $L^p$-solvability of corresponding \emph{inhomogeneous problems} is far more scant.   Here, by $L^p$-solvability of a boundary value problem, we mean solvability of the problem together with a natural, scale-invariant bound on the $L^p$-norm of an appropriate non-tangential maximal function.  For instance, a natural question to consider is to find the quantitative conditions which should be put on $H$ and $F$ so that   the solution $w$ to the problem 
\begin{equation}\label{eq.poisson}
	\left\{\begin{array}{ll}-\dv A\nabla w=H-\dv F,\quad&\text{in }\Omega,\\w=g,\quad&\text{on }\partial\Omega,\end{array}\right.
\end{equation}
satisfies that  $\m N(w)\in L^p(\partial\Omega)$, where $\m N$ is the non-tangential maximal function, defined in Section \ref{sec.main}. Assuming that $\Omega$ has reasonable geometry, if $H\equiv0$, $F\equiv0$, and $A\equiv I$, then it is known that $g\in L^p(\partial\Omega)$ implies that $\m N(w)\in L^p(\partial\Omega)$, but the problem for non-trivial $H$ and $F$ is not well understood, although    energy estimates  for solutions of the Poisson problem  are well-known if the data   lie  in Besov spaces (see also Section \ref{sec.connections} for some related results in the literature).


In this paper, we prove that the $L^p$-solvability of the homogeneous Dirichlet problem (denoted $(\Di_p^L)$)
\begin{equation}\label{eq.dirichlet}
	\left\{\begin{array}{ll}-\dv A\nabla u=0,\quad&\text{in }\Omega,\\u=g,\quad&\text{on }\partial\Omega,\end{array}\right.
\end{equation}
is equivalent to the $L^p$-solvability of  the \emph{Poisson problem} (\ref{eq.poisson}), whenever $H$  and $F$ lie in certain Carleson spaces (see  Theorems \ref{thm.poisson} and  \ref{thm.poisson2}). Furthermore, we will see a non-trivial application of this result to the regularity problem for elliptic PDEs with rough coefficients,  which    was our initial motivation to study the Poisson problem  and shows that the question of the $L^p$-solvability of the Poisson problem is far from being an isolated curiosity.
  
We   use  the $L^p$-solvability of the Poisson problem   to  study the Dirichlet regularity  problem for Dahlberg-Kenig-Pipher operators, which are operators $L=-\dv A\nabla$ where  $\dist(\cdot,  \Omega^c) \, |\nabla A|^2$  satisfies a Carleson measure condition. More precisely,   we show that if $\Omega\subset\bb R^{n+1}, n\geq2$ is bounded and uniformly $n$-rectifiable, and if $A$ is a real, not necessarily symmetric,  DKP matrix in $\Omega$ (see Definition \ref{def.cond}), then $L^{p'}$-solvability of the Dirichlet problem (\ref{eq.dirichlet}) for the adjoint operator $L^*=-\dv A^T\nabla$ implies  $L^p$-solvability of the Dirichlet \emph{regularity problem} for the operator $L$ (see Theorem \ref{thm.regularity}). In particular,   if $\Omega$ is a bounded domain satisfying the corkscrew condition which has $n$-Ahlfors regular boundary  and interior big pieces of chord-arc domains, then for any DKP operator associated with $\Omega$, there exists  $p>1$ such that the regularity problem for $L$ is $L^p$-solvable (Corollary \ref{cor.reg}). Note that Corollary \ref{cor.reg} is new even when $\Omega$ is the unit ball,  and our result complements the study of the Dirichlet problem for DKP operators, initiated in 2001 \cite{kp3}.

Briefly, we mention several further novel applications of the $L^p$-solvability of the Poisson problem:
\begin{itemize}
	\item We present new local $L^p$ estimates on the n.t. maximal function of the Green's function and its gradient (see (\ref{eq.truent1}) and (\ref{eq.truent2})), which are equivalent to the solvability of $(\Di_{p'}^L)$ (Theorem \ref{thm.poisson2}(g),(h)). 
	\item We establish an equivalence between the solvability of $(\Di_{p'}^L)$ and a certain $L^p$  estimate on the non-tangential maximal function of the gradient of (local) \emph{landscape functions} of the operator $-\dv A^T\nabla$ (Theorem \ref{thm.poisson2}(f) and Remark \ref{rm.landscape}). This result may be interpreted as a local $T1$ theorem for the $L^p$-solvability of a problem we call the \emph{Poisson-regularity problem}, itself also equivalent to the solvability of $(\Di_{p'}^L)$ (Theorem \ref{thm.poisson2}(d),(e)).  
	\item  We furnish new estimates (Corollary \ref{cor.eigenfn}) for eigenfunctions and their gradients on bounded rough domains, in terms of the landscape function (also known as the \emph{torsion function}),  loosely related to the Hassell-Tao inequality.
\end{itemize}

\subsection{Definitions and main results}\label{sec.main}
Let us now state the main results more precisely, and in order to do so, we will need to   record  some definitions. We always consider domains $\Omega\subset\bb R^{n+1}$, $n\geq1$ satisfying the  corkscrew condition and with $n$-Ahlfors regular boundary (see Section \ref{sec.geom} for definitions of these classical geometric concepts), while $A=(A_{ij})$ is always a real, not necessarily symmetric $(n+1)\times(n+1)$ matrix of   merely  bounded measurable coefficients in $\Omega$ verifying the strong ellipticity conditions 
\begin{equation}\label{eq.elliptic}
	\lambda|\xi|^2\leq\sum_{i,j=1}^{n+1}A_{ij}(x)\xi_i\xi_j,\qquad \Vert A\Vert_{L^{\infty}(\Omega)}\leq\frac1{\lambda},\quad x\in\Omega,\quad\xi\in\bb R^{n+1},
\end{equation}  
for some $\lambda\in(0,1)$. We denote $L=-\dv A\nabla$, and $L^*=-\dv A^T\nabla$ where $A^T$ is the transpose of $A$. Moreover, we denote $\delta(x)=\delta_\Omega(x):=\dist(x,\Omega^c)$.

For $\alpha>0$,   $U\subsetneq\bb R^{n+1}$, and $\xi\in\partial U$, we define the \emph{cone with vertex $\xi$ and aperture} $\alpha>0$ by
\begin{equation}\label{eq.cone}\nonumber
\gamma_\alpha^U(\xi):=\{x\in U:|x-\xi|<(1+\alpha)\delta_U(x)\},
\end{equation}
and if $U=\Omega$, we  write $\gamma_\alpha=\gamma_\alpha^\Omega$. Define the \emph{non-tangential maximal function} of  $u\in L^{\infty}_{\loc}(\Omega)$ by 
\[
\m N_\alpha(u)(\xi):=\sup_{x\in\gamma_\alpha(\xi)}|u(x)|,\quad\text{for }\xi\in\partial\Omega.
\]
Following \cite{kp}, we introduce the \emph{modified non-tangential maximal function} $\widetilde{\m N}_{\alpha,\hat c, r}$ for a given aperture $\alpha>0$, a parameter $\hat c\in(0,1/2]$, and $r\geq1$: for any $u\in L^r_{\loc}(\Omega)$, we write
\begin{equation}\label{eq.mnt}\nonumber
\widetilde{\m N}_{\alpha,\hat c,r}(u)(\xi):=\sup_{x\in\gamma_\alpha(\xi)}\Big(\dashint_{B(x,\hat c\delta(x))}|u(y)|^r\,dm(y)\Big)^{1/r},\qquad\xi\in\partial\Omega.
\end{equation}
The $L^p$ norms of non-tangential maximal functions are equivalent under changes of   $\alpha$ or the averaging parameter $\hat c$ (see Lemma \ref{lm.ntchange}). For ease of notation, we will often write   $\widetilde{\m N}_{\alpha,r}=\widetilde{\m N}_{\alpha,\hat c,r}$ if we do not wish to specify $\hat c$. When we do not need to specify neither $\alpha$ nor $\hat c$,  we will write ${\m N}=\m N_\alpha$, $\gamma=\gamma_\alpha$, and $\widetilde{\m N}_r=\widetilde{\m N}_{\alpha,\hat c,r}$. 

\begin{definition}[The $L^p$ Dirichlet and regularity problems]\label{def.dirichlet} Let $p\in(1,\infty)$ and $p'$ its H\"older conjugate. We say that the (homogeneous) \emph{Dirichlet problem for the operator $L$ with $L^{p'}$ data in $\Omega$ is solvable} (write $(\Di_{p'}^L)$ is solvable in $\Omega$), if there exists $C\geq1$ so that for each $g\in C_c(\partial\Omega)$, the solution $u$ to the continuous Dirichlet problem (\ref{eq.dirichlet})  with boundary data $g$ satisfies the estimate 
\begin{equation}\label{eq.direst}
	\Vert\m N(u)\Vert_{L^{p'}(\partial\Omega,\sigma)}\leq C\,\Vert g\Vert_{L^{p'}(\partial\Omega,\sigma)},
\end{equation}
where $\sigma$ is the restriction of the $n$-dimensional Hausdorff measure to $\partial\Omega$. We call  $C$ in (\ref{eq.direst}) the \emph{$(\Di_{p'}^L)$ constant}.  
 
Let $\Omega$ be a bounded domain, and $p\in(1,\infty)$. We say that the (homogeneous) \emph{Dirichlet regularity problem}   (or just \textit{regularity problem})  for the operator $L$ with   $\dt W^{1,p}(\pom,\sigma)$  data  is solvable in $\Omega$ (write $(\Reg_{p}^L)$ is solvable in $\Omega$), if there exists $C\geq1$ so that for each $f\in\Lip(\partial\Omega)$, the solution $u$ to the continuous Dirichlet problem for $L$ (\ref{eq.dirichlet}) with boundary data $f$ satisfies the estimate
\begin{equation}\label{eq.regest}
	\Vert\wt{\m N}_2(\nabla u)\Vert_{L^{p}(\partial\Omega,\sigma)}\leq C\Vert f\Vert_{\dt{W}^{1,p}(\partial\Omega,\sigma)}.
\end{equation} 
Here, $\dt W^{1,p}(\partial\Omega)$ is the \emph{Haj\l{}asz-Sobolev space} on $\partial\Omega$; we defer  its definition to Section \ref{sec.sobolev}. We call the constant $C$ in (\ref{eq.regest}) the \emph{$(\Reg_p^L)$ constant}.
\end{definition}

From now on we  take $\sigma=\m H^n|_{\partial\Omega}$ to be the underlying measure on $\partial\Omega$. The space $\dt W^{1,p}(\partial\Omega)$ was identified in \cite{mt22} as the correct space to consider the $L^p$ regularity problem in rough domains beyond the  Lipschitz setting; indeed, if $\Omega$ is Lipschitz, then $(\Reg_p^L)$ has usually \cite{kp, dpr17} been defined with the estimate $\Vert\wt{\m N}_2(\nabla u)\Vert_{L^p(\partial\Omega)}\leq C\Vert\nabla_tf\Vert_{L^p(\partial\Omega)}$ taking the place of (\ref{eq.regest}), where $\nabla_t$ denotes the tangential derivative. In two-sided chord arc domains, one has that $\Vert f\Vert_{\dt W^{1,p}(\partial\Omega)} \approx\Vert\nabla_t f\Vert_{L^p(\partial\Omega)}$; for more information on the geometric conditions that guarantee this bound, see \cite{mt22} and \cite{tt22}.

\subsubsection{$L^p$-solvability of the Poisson-Dirichlet problem}\label{sec.main1}

Let $q\geq1$ and $\hat c\in(0,1/2]$.   Define the \emph{$q$-Carleson functional} of a   function $H:\Omega\ra\bb R, H\in L_{\loc}^{q}(\Omega)$ by
\begin{equation}\label{eq.carlesonfn}
	\n C_{\hat c,q}(H)(\xi):=\sup_{r>0}\frac1{r^n}\int_{B(\xi,r)\cap\Omega}\Big(\dashint_{B(x,\hat c\delta(x))}|H|^{q}\Big)^{1/q}\,dm(x),\qquad \xi\in\partial\Omega.
\end{equation} 
The $L^p$ norms of the Carleson functionals $\n C_{\hat c,q}$ defined in (\ref{eq.carlesonfn}) are equivalent under a change of the averaging parameter $\hat c$ (see Lemma \ref{lm.changeavg}), and thus we    write $\n C_q=\n C_{\hat c,q}$ if we do not need to specify $\hat c$. 

Let us comment on the link between the Carleson functionals $\n C_q$ and the modified non tangential maximal functions $\wt{\m N}_r$. For any $q\geq1$ and $p>1$, let ${\bf C}_{q,p}$ be the Banach space of  functions $H\in L^q_{\loc}(\Omega)$ such that $\n C_q(H)\in L^p(\partial\Omega)$, with norm
\[
\Vert H\Vert_{{\bf C}_{q,p}}:=\Vert\n C_q(H)\Vert_{L^p(\partial\Omega)}.
\]
For any $r\in[1,\infty]$,  and $p>1$, let ${\bf N}_{r,p}$ be the Banach space of functions $u\in L^r_{\loc}(\Omega)$ such that $\widetilde{\m N}_r(u)\in L^p(\partial\Omega)$ (we identify $\widetilde{\m N}_{\infty}=\m N$), with norm
\[
\Vert u\Vert_{{\bf N}_{r,p}}:=\Vert\widetilde{\m N}_r(u)\Vert_{L^p(\partial\Omega)}.
\]
Let $p,q\in(1,\infty)$ and $p',q'$ the corresponding H\"older conjugates. Then (see Proposition \ref{prop.duality})  ${\bf N}_{q,p}=({\bf C}_{q',p'})^*$, and we have the estimates
\begin{equation}\label{eq.dual2}
\Big|\int_\Omega Hu\,dm\Big|\lesssim\Vert\n C_{q'}(H)\Vert_{L^{p'}(\partial\Omega)}\Vert\widetilde{\m N}_q(u)\Vert_{L^p(\partial\Omega)},\qquad\text{for each }H\in{\bf C}_{q',p'}, u\in{\bf N}_{q,p},
\end{equation}
\[
\Vert\widetilde{\m N}_q(u)\Vert_{L^p(\partial\Omega)}\lesssim\sup_{H:\Vert\n C_{q'}(H)\Vert_{L^{p'}(\partial\Omega)}=1}\Big|\int_{\Omega}Hu\,dm\Big|,\qquad\text{for each }u\in{\bf N}_{q,p}.
\]

The Carleson   functionals  as defined in (\ref{eq.carlesonfn}) were introduced by Hyt\"onen and Ros\'en in \cite{hr13} when $\Omega$ is the half-space; they were motivated  by the question of identifying the dual space of ${\bf N}_{2,p}$ which has applications for the solvability of the Neumann problem \cite{kp, aa11, ar12, barton21}.   Hyt\"onen and Ros\'en obtained dyadic versions of the above duality  \cite[Section 2]{hr13}, and in the proof of Proposition \ref{prop.duality} we will see that their dyadic results translate to our setting of domains satisfying the corkscrew condition and with $n$-Ahlfors regular boundary.  Moreover, observe that   (\ref{eq.dual2}) can be understood as a generalization  of Carleson's theorem (see \cite[Theorem 2]{carleson62} or \cite{cms}). Finally, the spaces ${\bf N}_{q,p}$ and ${\bf C}_{q',p'}$ are closely related to the tent spaces of Coifman, Meyer and Stein \cite{cms}, and reduce to them for certain choices of the parameters.   Indeed, if we define the area integral operators by
\begin{equation}\label{eq.areaint}
\m A_{q}(H)(\xi):=\Big(\int_{\gamma(\xi)}|H(x)|^q\,\frac{dm(x)}{\delta(x)^{n+1}}\Big)^{\frac1q},\qquad	\wt{\m A}_{q}(H)(\xi):=\Big(\int_{\gamma(\xi)}\Big(\dashint_{B(x,\delta(x)/2)}|H|^2\,dm\Big)^{\frac{q}2}\,\frac{dm(x)}{\delta(x)^{n+1}}\Big)^{\frac1q},
\end{equation}
for $q>0$ and $\xi\in\partial\Omega$, then the proof of \cite[Theorem 3]{cms}  implies that, under the assumption that $\Omega$ is a Corkscrew domain with $n$-Ahlfors regular boundary such that either $\Omega$ is bounded or $\partial\Omega$ is unbounded,
\begin{equation}\label{eq.tentandcar}
\Vert\n C_2(H)\Vert_{L^p(\partial\Omega)}\approx\Vert\wt{\m A}_1(\delta_{\Omega}H)\Vert_{L^p(\partial\Omega)},\qquad\text{for all }p>1.
\end{equation}
Hence, if
\[
T^p_q:=\big\{H\in L^q_{\loc}(\Omega):\m A_q(H)\in L^p(\partial\Omega)\big\},\qquad \wt T^p_q:=\big\{H\in L^2_{\loc}(\Omega):\wt{\m A}_q(H)\in L^p(\partial\Omega)\big\},\quad 0<q\leq p<\infty,
\]
then for any $p>1$, a given function $H\in L^2_{\loc}(\Omega)$ lies in ${\bf C}_{2,p}$ if and only if $\delta_\Omega H\in\wt T^p_1$.

We are ready to state the first main result. Write $2^*=\frac{2(n+1)}{n-1}$ and $2_*=(2^*)'=\frac{2(n+1)}{n+3}$.
  
\begin{theorem}\label{thm.poisson} Let $\Omega\subsetneq\bb R^{n+1}$, $n\geq2$ be a domain satisfying the corkscrew condition and with $n$-Ahlfors regular boundary, such that either $\Omega$ is bounded or $\partial\Omega$ is unbounded. Let $p>1$, $p'$ its H\"older conjugate, and $L=-\dv A\nabla$. Assume that $(\Di_{p'}^L)$ is solvable in $\Omega$.   Let $g\in C_c(\partial\Omega)$, and $H,F\in L^{\infty}_c(\Omega)$. Then there exists a unique weak solution $w\in C(\overline{\Omega})\cap W^{1,2}_{\loc}(\Omega)$ to the problem (\ref{eq.poisson})  satisfying that 
\begin{equation}\label{eq.ntmaxn}
\Vert\widetilde{\m N}_{2^*}(w)\Vert_{L^{p'}(\partial\Omega)}\leq C\Big[\Vert\n C_{2_*}(\delta_\Omega H)\Vert_{L^{p'}(\partial\Omega)}+\Vert\n C_2(F)\Vert_{L^{p'}(\partial\Omega)}+\Vert g\Vert_{L^{p'}(\partial\Omega)}\Big].
\end{equation}
Here, $C$ depends only on $n$, $\lambda$, $p'$, the corkscrew and  $n$-Ahlfors regularity constants,   and the $(\Di_{p'}^L)$ constant. 
\end{theorem}

\begin{remark}\label{rm.degnm} Note that, by the classical De Giorgi-Nash-Moser theory for elliptic equations with non-zero right-hand side, the use of the exponent $2^*$ for the modified non-tangential maximal function in the left-hand side of (\ref{eq.ntmaxn}) is sharp under our quantitative assumptions on $H$ and $F$. However, as will be clear by the method of proof in Section \ref{sec.poisson}, if $q\in[1,\frac{n+1}n)$ and $r\in[1,\frac{n+1}{n-1})$, then we may also show that
	\begin{equation}\label{eq.ntmaxtrue}
		\Vert\m N(w)\Vert_{L^{p'}(\partial\Omega)}\leq C\Big[\Vert\n C_{r'}(\delta_\Omega H)\Vert_{L^{p'}(\partial\Omega)}+\Vert\n C_{q'}(F)\Vert_{L^{p'}(\partial\Omega)}+\Vert g\Vert_{L^{p'}(\partial\Omega)}\Big].
	\end{equation}
\end{remark} 

\begin{remark} The assumption $H,F\in L^{\infty}_c(\Omega)$ in Theorem \ref{thm.poisson} can be generalized to $F\in{\bf C}_{2,p'}$ and $H$ such that $\delta_\Omega H\in{\bf C}_{2_*,p'}$; see Theorem \ref{thm.poisson3}.
\end{remark}

Theorem \ref{thm.poisson} is new even when $L=-\Delta$ or $\Omega$ is a ball. This is the first time that   an estimate such as (\ref{eq.ntmaxn}) or (\ref{eq.ntmaxtrue}) has appeared explicitly in the literature for the solution to a  Poisson-Dirichlet problem on a domain $\Omega\subsetneq\bb R^{n+1}$.  On the other hand, certain related estimates may be found in the literature \cite{hmm, barton21}, and Sobolev and Besov space estimates for solutions to the Poisson problem with right-hand side in Sobolev-Besov spaces are well-studied and have been considered many times; we will look into the connections of Theorem \ref{thm.poisson} with these results a little deeper in Section \ref{sec.connections}.

The  proof of Theorem \ref{thm.poisson} is presented in Section \ref{sec.poisson1}.  The argument contains several novel components; let us give a very brief sketch of the ideas for the case that $g\equiv0$, $H\equiv0$. First,  we may use  Green's representation formula   to see that we must find  appropriate bounds for  $\sup_{x\in\gamma_\alpha(\xi)}\int_\Omega\nabla_yG(x,y)F(y)\,dm(y)$, for any $\xi\in\partial\Omega$. To achieve the desired bounds, we split $\Omega$ into four   regions tailored to the pole $x$ and a non-tangential cone $\gamma(\xi)$. Namely, the first region is a small Whitney ball around $x$, the second region is a small Carleson tent over $\xi$, the third region is the rest of a large cone $\gamma_\beta(\xi)$, and   the last region  corresponds to a far-away, tangential (with respect to $\xi$) portion of $\Omega$. 
	
Each of the regions described above is treated separately, using harmonic analysis techniques:  for the first, we use the   properties of the solution operators $L^{-1}_\Omega$ and $L^{-1}_\Omega\dv$ (\ref{eq.op}); for the second, we use a dyadic Carleson embedding result and the weak-$RH_p$ property of the Poisson kernel; for the third, we relate the bound to an area integral estimate and use (\ref{eq.tentandcar}); and finally, we break up the fourth region into annular regions on which we apply the same ideas as for the second region, but since these annular regions are tangential and far-away from $\xi$ and $x$, we use the boundary H\"older  inequality and a global weak-$RH_p$ inequality for the Poisson kernel to furnish summable decay. The bounds corresponding to the second and fourth regions are the most difficult, since it is here where a delicate Carleson embedding is used; indeed, this is one of the main technical innovations of our paper. 

It is natural to wonder how sharp Theorem \ref{thm.poisson} is. This question is solved by our second main result, which gives several new characterizations of the solvability of $(\Di_{p'}^L)$, a new estimate  for the gradient of the Green's function, and a certain local $T1$ theorem, under very mild geometric assumptions of the domain. First, we need a few more definitions.

\begin{definition}[Landscape functions]\label{def.landscape} Given a Whitney cube $I$ in a Whitney decomposition of $\Omega$ (see Section \ref{sec.whitney} for definitions), we say that the unique weak solution $u_I\in Y_0^{1,2}(\Omega)$ to the equation $Lu_I=\1_I$ is the \emph{local landscape function}   subordinate to $I$. If $\Omega$ is bounded, then the unique weak solution $u\in Y_0^{1,2}(\Omega)$ to the equation $Lu=\1_{\Omega}$ is the \emph{landscape function} of $L$.
\end{definition} 

Up to a rescaling, the landscape function is known in the literature of shape optimization as the \emph{torsion function} of $L$ (see, for instance, \cite{bc94}). The landscape function (introduced with this name by Filoche and Mayboroda in \cite{fm}) has recently been a subject of heavy interest when $L$ is the Schr\"odinger operator $L=-\dv A\nabla+V$, $V\geq0$, due to applications to mathematical physics and the phenomenon of localization of eigenfunctions. For related literature, see \cite{fm, adfjm, pog21}.

\begin{definition}[$(\operatorname{PD}_{p'}^L)$ and $(\operatorname{PR}_p^L)$]
For any $p\in(1,\infty)$, we say that $(\operatorname{PD}_{p'}^L)$ is solvable in $\Omega$   if there exists  $C>0$ so that for each $H,F\in L^{\infty}_c(\Omega)$, the unique weak solution $w\in Y_0^{1,2}(\Omega)$ (see Section \ref{sec.pde} for the definition of this space) to the equation $Lw=H-\dv F$ satisfies the estimate
\begin{equation}\label{eq.pd1}
\Vert\wt{\m N}_{2^*}(w)\Vert_{L^{p'}(\partial\Omega)}\leq C\big[\Vert\n C_{2_*}(\delta_\Omega H)\Vert_{L^{p'}(\partial\Omega)}+\Vert\n C_2(F)\Vert_{L^{p'}(\partial\Omega)}\big].
\end{equation}
We say that $(\operatorname{PD}_{p'}^L)$ is solvable in $\Omega$ for $H\equiv0$ if the estimate (\ref{eq.pd1}) holds for $H\equiv0$ and arbitrary $F\in L^{\infty}_c(\Omega)$.

We say that $(\operatorname{PR}_p^L)$ is solvable in $\Omega$ if   there exists  $C>0$ so that for each $H, F\in L^{\infty}_c(\Omega)$, the unique weak solution $v\in Y_0^{1,2}(\Omega)$ to the equation $Lv=H-\dv F$ satisfies the estimate
\begin{equation}\label{eq.estreg}
\|\wt{\m N}_{2}(\nabla v)\|_{L^{p}(\partial\Omega)}\leq C\big[\|\n C_{2_*}(H)\|_{L^{p}(\partial\Omega)}+\|\n C_2(F/\delta_\Omega)\|_{L^{p}(\partial\Omega)}\big].
\end{equation}
We say that $(\operatorname{PR}_p^L)$ is solvable \emph{for local landscape functions in $\Omega$} if there exist $C>0$, $C'\geq2$ so that for each Whitney cube $I$ in a Whitney decomposition of $\Omega$, the local landscape function $u_I$ satisfies the estimate
\begin{equation}\label{eq.estreg2}
	\Vert\wt{\m N}_2^{C'\ell(I)}(\nabla u_I)\Vert_{L^p(C'B_Q,\sigma)}\leq C\ell(I)^{1+\frac np},
\end{equation}
where $Q$ is a boundary cube of $I$ (see (\ref{eq.truncated}) for the definition of the truncated non-tangential maximal function $\wt{\m N}^{s}$).
\end{definition}

 \begin{remark}\label{rm.more} If $(\operatorname{PR}_p^{L^*})$ is solvable in $\Omega$, then an easy variant of \cite[Theorem 7.2]{mt22}  allows for solvability with general data in the space ${\bf C}_{2_*,p}$. More precisely, for any $H\in{\bf C}_{2_*,p}$, there exists a weak solution $v\in W^{1,2}_{\loc}(\Omega)$ to the equation $Lv=H$ satisfying  (\ref{eq.estreg}) and the non-tangential limit
 	\[
 	\lim_{\gamma(\xi)\ni x\ra\xi}\Big(\dashint_{B(x,\delta(x)/2)}|v|^{2^*}\,dm\Big)^{\frac1{2^*}}=0,\qquad\text{for }\sigma\text{-a.e.\ }\xi\in\partial\Omega.
 	\]
 	Similarly, if $(\Reg_p^L)$ is solvable in $\Omega$, then by essentially the same argument of \cite[Theorem 7.2]{mt22}, we can extend the solvability to general data in the space $W^{1,p}(\partial\Omega)$: for any $f\in W^{1,p}(\partial\Omega)$, there exists a weak solution $u\in W^{1,2}_{\loc}(\Omega)$ to the equation $Lu=0$ satisfying   (\ref{eq.regest}) and that $u\ra f$ non-tangentially $\sigma$-a.e.\ on $\partial\Omega$.
 \end{remark}

\begin{theorem}[Characterizations of $(\Di_{p'}^L)$]\label{thm.poisson2} Let $\Omega\subsetneq\bb R^{n+1}$, $n\geq2$ be a domain satisfying the corkscrew condition and with $n$-Ahlfors regular boundary, such that either $\Omega$ is bounded, or $\partial\Omega$ is unbounded. Let $p\in(1,\infty)$, $p'$ its H\"older conjugate, and $L=-\dv A\nabla$. The following are equivalent.
\begin{enumerate}[(a)]
\item $(\Di_{p'}^L)$ is solvable in $\Omega$.
\item $(\operatorname{PD}_{p'}^L)$ is solvable in $\Omega$.
\item $(\operatorname{PD}_{p'}^L)$ is solvable in $\Omega$ for $H\equiv0$.
\item $(\operatorname{PR}_p^{L^*})$ is solvable in $\Omega$.
\item $(\operatorname{PR}_p^{L^*})$ is solvable in $\Omega$ for $F\equiv0$.
\item $(\operatorname{PR}_p^{L^*})$ is solvable for local landscape functions in $\Omega$.
\item There exist $C>0$, $C'\geq2$ so that for some $q\in(1,\frac{n+1}n)$ and for each $x\in\Omega$, we have the estimate
\begin{equation}\label{eq.truent1}
\Vert\wt{\m N}_q^{C'\delta(x)}\big(\nabla_2 G_L(x,\cdot)\big)\Vert_{L^p(B(x,C'\delta(x)),\sigma)}\leq C\delta(x)^{-n/p'},
\end{equation}
where $G_{L}$ is the Green's function for the operator $L$ (see Definition \ref{def.green}). 
\item There exist $C>0$, $C'\geq2$ so that for some $q\in(1,\frac{n+1}{n-1})$ and for each $x\in\Omega$, we have the estimate
\begin{equation}\label{eq.truent2}
	\Vert\wt{\m N}_q^{C'\delta(x)}\big(\frac{G_L(x,\cdot)}{\delta}\big)\Vert_{L^p(B(x,C'\delta(x)),\sigma)}\leq C\delta(x)^{-n/p'}.
\end{equation}
\end{enumerate}
\end{theorem}

Theorem \ref{thm.poisson2} is proved in Section \ref{sec.proofpoisson}. Note that (a)$\implies$(b) is directly implied by Theorem \ref{thm.poisson}, while (b)$\implies$(c), (d)$\implies$(e) and (e)$\implies$(f) follow by definition. The implications (c)$\implies$(d) and (e)$\implies$(c) follow from an application of the $\m N$-$\n C$ duality, while the direction (c)$\implies$(a) is proved by combining this duality with an argument in \cite[Section 9]{mt22}. The statements (f)$\implies$(a), (g)$\iff$(a) and (h)$\iff$(a) are obtained via similar arguments to the proof of (c)$\implies$(a). 

Let us discuss Theorem \ref{thm.poisson2}. It shows that one loses no information when switching from the study of the homogeneous Dirichlet problem with $L^{p'}$ data, to the study of the Poisson problem with inhomogeneous data $\div F$ with $F$ in the Carleson space ${\bf C}_{2,p'}$. Perhaps surprisingly, both of these problems are also equivalent to the ``Poisson regularity'' problem $(\operatorname{PR}_p^{L^*})$, even though the implication $(\Di_{p'}^L)\implies(\Reg_p^{L^*})$ remains an open problem for arbitrary elliptic operators $L$ on Lipschitz domains.   We now emphasize two observations, important enough to merit their own remarks.

\begin{remark} The statements (g) and (h) give  quantitative characterizations of $(\Di_{p'}^L)$ in terms of new local boundary estimates on the the Green's function and its gradient.  We have not been able to find either estimate (\ref{eq.truent1}) or (\ref{eq.truent2}) written explicitly in the literature even in significantly simpler geometric settings, or for $L=-\Delta$. On the other hand, both estimates are natural since they are closely related to the $RH_p$ property of the Poisson kernel.   Moreover, observe that a global version of the estimate (\ref{eq.truent1}) follows formally from applying that $(\operatorname{PR}_p^{L^*})$ is solvable in $\Omega$, since $L^*G_L(x,\cdot)=\delta_x$, and a straightforward computation shows that, if $\n C(\mu)(\xi):=\sup_{r>0}\frac1{r^n}\mu(B(\xi,r)\cap\overline{\Omega})$ for $\xi\in\partial\Omega$ and $\mu$ a measure on $\overline{\Omega}$, then $\Vert\n C(\delta_x)\Vert_{L^p(\sigma)}\approx\delta(x)^{-n/p'}$. This argument can be made rigorous by approximating $\delta_x$ by averages, along the lines of \cite{hk}.
	
There has recently been strong interest in studying how estimates for the Green's function are related to the geometry of domains \cite{azz19, dlm22, dm22, chm22, flm22}, and  in this connection, let us point out that, when coupled with the literature on free boundary results for $(\Di_{p'}^{-\Delta})$, Theorem \ref{thm.poisson2} implies that under its background assumptions, the estimates (\ref{eq.truent1}), (\ref{eq.truent2}) for $L=-\Delta$ characterize the IBPCAD condition (or uniform rectifiability plus the weak local John condition; see \cite{ahmmt}). 
\end{remark}

\begin{remark}\label{rm.landscape} Here, we give context to the statement (f) and why its equivalence to the statements (a) and (e) is interesting in its own right. As remarked before, if (e) holds, then (f) holds immediately, since
\begin{equation}\label{eq.computew}
\Vert\n C_2(\1_I)\Vert_{L^p(\sigma)}\approx\ell(I)^{1+\frac np},\qquad\text{for each Whitney box }I.
\end{equation}
So (f)$\implies$(e) tells us that in order to ascertain that $(\operatorname{PR}_p^{L^*})$ is solvable for arbitrary $H\in{\bf C}_{2,p}$ and $F\equiv0$, it is enough to obtain the estimate (\ref{eq.estreg2}) (which itself is a local version of the global estimate (\ref{eq.estreg})) for the countable collection $H=\1_I$, where $I$ runs over all Whitney boxes in $\Omega$. It is easy to interpret this result as a local $T1$ theorem, as follows: if $T=\nabla (L_\Omega^*)^{-1}$ (see (\ref{eq.op})), then the statement (e) is equivalent to asking that the map $T:{\bf C}_{2_*,p}\ra{\bf N}_{2,p}$ is bounded (by Remark \ref{rm.more}). But (f)$\implies$(e) gives that this map is bounded if the local estimate (\ref{eq.estreg2}) holds for $\{T\1_I\}_{I\in\m W}$, which is exactly a  local $T1$ theorem.  

The implications (f)$\implies$(a) and (f)$\implies$(d) join a family of results in the literature about the landscape function which say, loosely speaking, that several distinct properties of the operator $L$ can be a priori controlled by properties of the landscape function; see for instance \cite[Display 3]{fm}, \cite[Lemma 4.1, Theorem 4.4, Theorem 5.1]{adfjm} or \cite[Theorem 1.21, Theorem 1.25]{pog21}. Finally, we mention that the equivalence (f)$\iff$(a) gives the first robust connection between the theory of the landscape/torsion function (primarily developed for the study of questions in mathematical physics and shape optimization)  and the theory of solvability of the Dirichlet problem with singular data. 
\end{remark}

We now give an immediate application of Theorems \ref{thm.poisson} and \ref{thm.poisson2} toward  estimates for eigenfunctions and their gradients of elliptic operators $L=-\dv A\nabla$  on rough domains.

\begin{corollary}[Boundary estimates for eigenfunctions]\label{cor.eigenfn} Retain the setting and assumptions of Theorem \ref{thm.poisson2}, and moreover, assume that $\Omega$ is bounded, that $A$ is symmetric, and that $(\Di_{p'}^L)$ is solvable in $\Omega$. Let $\psi\in W_0^{1,2}(\Omega)$ satisfy $L\psi=E\psi$ in $\Omega$ for some eigenvalue $E>0$, and let $u$ be the landscape function of $L$ on $\Omega$. Then there exists $C>0$, depending only on $n$, $\lambda$, $p$, the corkscrew condition and $n$-Ahlfors regularity constants, and the $(\Di_{p'}^L)$ constant, such that
\begin{equation}\label{eq.efn1}
\Vert\wt{\m N}_{2^*}(\psi)\Vert_{L^{p'}(\partial\Omega)}\leq CE\Vert\n C_{2_*}(\dist(\cdot,\partial\Omega)\psi)\Vert_{L^{p'}(\partial\Omega)}\leq CE^2\Vert\psi\Vert_{L^{\infty}(\Omega)}\Vert\n C_{2_*}(\dist(\cdot,\partial\Omega)u)\Vert_{L^{p'}(\partial\Omega)},
\end{equation}
\begin{equation}\label{eq.efn2}
	\Vert\wt{\m N}_2(\nabla\psi)\Vert_{L^{p}(\partial\Omega)}\leq CE\Vert\n C_{2_*}(\psi)\Vert_{L^{p}(\partial\Omega)}\leq C E^2\Vert\psi\Vert_{L^{\infty}(\Omega)}\Vert\n C_{2_*}(u)\Vert_{L^{p}(\partial\Omega)}.
\end{equation}
In particular, if $\Omega$ is a bounded Lipschitz domain and $L=-\Delta$, then both estimates (\ref{eq.efn1}) and (\ref{eq.efn2}) hold with $p=p'=2$. 
\end{corollary}

\noindent\emph{Proof.} We prove (\ref{eq.efn2}).  Since $(\Di_{p'}^L)$ is solvable in $\Omega$, then by Theorem \ref{thm.poisson2} we have that $(\operatorname{PR}_{p}^L)$ is solvable in $\Omega$, and so we see that
\[
\Vert\wt{\m N}_2(\nabla\psi)\Vert_{L^{p}(\partial\Omega)}\lesssim\Vert\n C_{2_*}(E\psi)\Vert_{L^{p}(\partial\Omega)}\leq E\Vert\n C_{2_*}(\psi)\Vert_{L^{p}(\partial\Omega)}.
\]
Since (see \cite[Display 3]{fm}) $|\psi(x)|\leq E\Vert\psi\Vert_{L^{\infty}(\Omega)}u(x)$ for each $x\in\Omega$, the desired estimate (\ref{eq.efn2}) follows.  The estimate (\ref{eq.efn1}) follows by an analogous argument, using Theorem \ref{thm.poisson}.\hfill{$\square$}

The estimates (\ref{eq.efn1}) and (\ref{eq.efn2}) are new even for the eigenfunctions of the Laplacian;   although estimates  related to (\ref{eq.efn2}) controlling the normal derivative of Dirichlet eigenfunctions have been studied heavily, such as the Hassell-Tao inequality \cite{ht02}, when the domain is smooth. On the other hand, the usual techniques in the literature rely on Rellich identities, and thus do not extend to the optimally rough geometric setting that we consider here. Indeed, since we merely assume that $(\Di_{p'}^L)$ is solvable (in addition to the background assumptions that  the domain has Corkscrew points and its boundary is $n$-Ahlfors regular), the estimates (\ref{eq.efn1}) and (\ref{eq.efn2}) hold up to the IBPCAD domains when $L=-\Delta$, and for certain other specific $L$, the estimates hold \emph{even for domains with fractal boundaries} (see \cite{dm2} for the existence of such an operator in the complement of the 4-corner Cantor set).

  We emphasize that the constant $C$ in the estimates (\ref{eq.efn1}) and (\ref{eq.efn2}) does not depend on $|\Omega|$, nor on the eigenvalue $E$. There is a vast literature studying the geometry of eigenfunctions of $-\Delta$  and estimates of eigenfunctions near the boundary; for a few related results, see for instance, \cite{lp95, bp99, vdbb99, ht02, kls13, sx13} and the survey \cite{gn13}.

\subsubsection{The regularity problem for Dahlberg-Kenig-Pipher operators}\label{sec.dkp} We now pass to an important application of Theorem \ref{thm.poisson}. We will see that Theorem \ref{thm.poisson} helps us solve the regularity problem for certain elliptic operators whose gradient satisfies a Carleson measure condition.   Denote  
$${\rm osc}_{B(x,\delta(x)/2)}(A) = \sup_{y,z\in B(x,\delta(x)/2)}|A(y)-A(z)|.$$

\begin{definition}[DKP and DPR conditions]\label{def.cond} We say that a real matrix function $A$ in $\Omega\subset\bb R^{n+1}$, $n\geq1$, satisfies the $K$-\emph{Dahlberg-Kenig-Pipher condition} ($K$-DKP) in $\Omega$ if $A\in\Lip_{\loc}(\Omega)$ and 
\begin{equation}\label{eq.dkp}
d\mu(x):= {\1}_\Omega(x)\,\esssup\big\{\delta_\Omega(z)|\nabla A(z)|^2:z\in B(x,\delta_\Omega(x)/2)\big\}\,dm(x)
\end{equation}
is a $K$-Carleson measure with respect to $\HH^n|_\pom$. That is, for any ball $B$ centered in $\pom$, 
\[
\mu(B\cap\Omega) \leq K\HH^n(B\cap\pom).
\]

We say that  $A$ satisfies the \emph{$\tau$-Dindo\v{s}-Pipher-Rule condition} ($\tau$-DPR) in $\Omega$ if the measure
$$d\mu(x):= {\1}_\Omega(x)\,\delta_\Omega(x)^{-1}{\rm osc}_{B(x,\delta_\Omega(x)/2)}(A)^2\,dm(x)$$
is a $\tau$-Carleson measure with respect to $\HH^n|_\pom$.
\end{definition}

We will say that $A$ is a DKP (DPR) matrix and that $L=-\dv A\nabla$ is a DKP (DPR) operator if $A$ satisfies the $K$-DKP ($\tau$-DPR) condition for some $K>0$ ($\tau>0$). It is easy to see that if $A$ satisfies the $K$-DKP condition, then it also satisfies the $\tau$-DPR condition, with $\tau\lesssim K$. The DKP operators first arose in the literature as pullbacks of the Laplacian under a mapping of Dahlberg-Kenig-Stein \cite{dah4}. The well-posedness of the $L^p$-Dirichlet problem for  DKP operators in Lipschitz domains was shown in \cite{kp3}, and their result carries over to chord-arc domains, and more generally to domains with interior big pieces of chord-arc domains considered in \cite{ahmmt} (see Section \ref{sec.geom} for the definition), by the methods of \cite{jk3, dj} (see for instance the remarks following Theorem 1.1 in \cite{ahmmt}).

The question of establishing that $(\Reg_p^L)$ is solvable for some $p>1$ when $L$ is a DKP operator has remained open since 2001, even when $\Omega$ is the unit ball. The closest result in the published literature is from \cite{dpr17}, where the authors show that for any $p\in(1,\infty)$, $(\Reg_p^L)$ is solvable if $\Omega$ is a Lipschitz domain \emph{with  small enough Lipschitz constant $\ell$} and if $A$ satisfies the $\tau$-DPR condition \emph{with small enough $\tau>0$}, depending on $p$. 

The following theorem is our third main result, which establishes that $(\Di_{p'}^{L^*})\implies(\Reg_p^L)$ for DKP operators on bounded Corkscrew domains with uniformly $n$-rectifiable boundaries (see Section \ref{sec.geom} for the definition of uniform $n$-rectifiability). This solves the question posed in the previous paragraph.
 
\begin{theorem}[$(\Di_{p'}^{L^*})\implies(\Reg_p^L)$ for DKP operators]\label{thm.regularity} Let $\Omega\subset\bb R^{n+1}$, $n\geq2$, be a bounded domain satisfying the corkscrew condition and with uniformly $n$-rectifiable boundary. Let $p\in(1,\infty)$, $p'$ its H\"older conjugate, and $L=-\dv A\nabla$, where $A$ is a  DKP matrix in $\Omega$. Suppose that $(\Di_{p'}^{L^*})$ is solvable in $\Omega$. Then $(\Reg_p^L)$ is solvable in $\Omega$, and  the $(\Reg_p^L)$ constant depends only on $p$, $n$, $\lambda$, the corkscrew constant, the uniform $n$-rectifiability constants, the DKP constant, and the $(\Di_{p'}^{L^*})$ constant.
\end{theorem} 

\begin{remark} In the case of the Laplacian, a result analogous to  Theorem \ref{thm.regularity} was recently obtained in \cite[Theorem 1.2]{mt22} by two of us for $p\in(1,2+\ep_0)$, but the proof relied heavily on the $L^2$ boundedness of the Riesz transform and the double layer potential  on uniformly $n$-rectifiable sets, while this theory is not available for   layer potentials associated to DKP operators. Because of this, our proof of Theorem \ref{thm.regularity} must by necessity be substantially different from the argument in \cite{mt22}.  In fact, our proof of Theorem \ref{thm.regularity}   gives an alternative argument to \cite[Theorem 1.2]{mt22} and improves upon it, since our result holds for $p\in(1,\infty)$. 
\end{remark} 

As mentioned before, for any DKP matrix $A$ there exists $p>1$ so that $(\Di_{p'}^{L^*})$ is solvable in $\Omega$ whenever $\Omega$ satisfies the corkscrew condition and the IBPCAD condition, and its boundary $\partial\Omega$ is $n$-Ahlfors regular  \cite{kp3, dj, ahmmt}; and moreover, if $\Omega$ satisfies the IBPCAD condition and $\partial\Omega$ is $n$-Ahlfors regular, then $\partial\Omega$ is uniformly $n$-rectifiable \cite[Theorem 1.5]{ahmmt}  (see Section \ref{sec.geom} for the definitions). Hence we have the following corollary.
 
\begin{corollary}[$(\Reg_p^L)$ for DKP operators in IBPCAD domains]\label{cor.reg} Let $\Omega\subset\bb R^{n+1}$, $n\geq2$, be a bounded domain satisfying the corkscrew condition and the IBPCAD condition,   and assume that $\partial\Omega$ is $n$-Ahlfors regular. Let $L=-\dv A\nabla$ where $A$ is a DKP matrix in $\Omega$. Then there exists $p>1$ such that $(\Reg_p^L)$ is solvable in $\Omega$.	
\end{corollary}

Corollary \ref{cor.reg} fully answers the question of the solvability of $(\Reg_p^L)$ for some $p>1$ when $L$ is a DKP operator on a bounded rough domain, but we mention again that this problem had been open even when $\Omega$ is the unit ball.  If, in addition to the hypothesis of Corollary \ref{cor.reg}, we have that $\Omega$ satisfies the Harnack chain condition, then   $\Omega$ is a chord-arc domain \cite{ahmnt} (see Section \ref{sec.geom} for definitions), and in this case we can use the theory of Carleson perturbations for the regularity problem \cite{kp2, dfm22} to deduce  the following improvement to Corollary \ref{cor.reg}.

\begin{corollary}[$(\Reg_p^L)$ for DPR operators]\label{cor.reg2} Let $\Omega\subset\bb R^{n+1}$, $n\geq2$, be a bounded chord-arc domain, and let $L=-\dv A\nabla$ where $A$ is a DPR matrix in $\Omega$. Then there exists $p>1$ such that $(\Reg_p^L)$ is solvable in $\Omega$.	
\end{corollary}

Let us now give a very quick overview on the method of proof of Theorem \ref{thm.regularity}, which is presented in Sections \ref{sec.strat}-\ref{sec.reg}. First, as a main innovation, we use the duality between non-tangential maximal functions and Carleson functionals (Proposition \ref{prop.duality}) to reduce the regularity problem to the study of the conormal derivative of solutions to Poisson Dirichlet problems. Indeed, for a solution $w$ to the Poisson-Dirichlet problem (\ref{eq.poisson}) with $H\equiv0$, we show in Proposition \ref{prop.conormal}  that the variational conormal derivative $\partial_{\nu_A}w$ (see Section \ref{sec.strat} for the definition) is a bounded linear functional on the Haj\l{}asz-Sobolev space $\dt W^{1,p}(\partial\Omega)$, and the  new estimate (\ref{eq.boundconormal}) is the main ingredient in our proof of Theorem \ref{thm.regularity}, since it replaces a $1$-sided Rellich estimate.

Given Proposition \ref{prop.conormal}, Theorem \ref{thm.regularity} follows rather easily  from the duality between $\m N$ and $\n C$ and the $L^p$-solvability of the Poisson-Dirichlet problem, as shown early in Section \ref{sec.strat}. However,  the main estimate (\ref{eq.boundconormal}) of Proposition \ref{prop.conormal} requires a completely new argument, even if one were to cut out the sharp geometric considerations.  

To prove Proposition \ref{prop.conormal}, first we conceive of an appropriate corona decomposition which is able to invoke the results of \cite{dpr17} as ``black boxes'' on each corona subdomain. With the corona decomposition at hand, we construct an ``almost $L$-elliptic'' extension of Lipschitz boundary data. Both the corona decomposition and the construction of the almost $L$-elliptic extension  are based on analogous ideas from \cite{mt22} where the regularity problem for the Laplace equation in uniformly $n$-rectifiable domains was solved. Nevertheless, the corona decomposition in the present manuscript requires highly non-trivial modifications from the one in \cite{mt22} to ensure that the hypotheses of \cite[Theorem 2.10]{dpr17} are satisfied by $L$ on each corona subdomain.

The next step  in the proof  is to show Proposition \ref{prop.ext}, which is an $L^p$ estimate on the modified non-tangential maximal function of the gradient of the almost $L$-elliptic extension. This estimate  is interesting in its own right, since it illustrates the naturality of the almost $L$-elliptic extension, and it is new even for the Laplacian.  Indeed, Proposition \ref{prop.ext} together with \cite[Lemma 4.7]{mt22} imply that the almost harmonic extension is a correct analogue of the Varopoulos extension for one ``smoothness level'' up. However, the proof   is very delicate, and the tools which we use to achieve it are quite sharp, such as the localization theorem for the regularity problem, and the dyadic Carleson embedding results that we also invoked in the proof of Theorem \ref{thm.poisson}.

The last step in the proof of Proposition \ref{prop.conormal} is to show the estimate
\begin{equation}\label{eq.test}
\Big|\int_\Omega A\nabla w\nabla v_\varphi\,dm\Big|\lesssim\big[\Vert g\Vert_{L^{p'}(\partial\Omega)}+\Vert\n C_2(F)\Vert_{L^{p'}(\partial\Omega)}\big]\Vert\nabla_{H,p}\varphi\Vert_{L^p(\partial\Omega)},
\end{equation}
where $F\in{\bf C}_{2,p'}$, $\varphi\in\Lip(\partial\Omega)$, $w$ is the unique weak solution to the problem (\ref{eq.poisson}) with  $H\equiv0$, and $v_\varphi$ is the almost $L^*$-elliptic extension of $\varphi$. To prove (\ref{eq.test}), we combine the $L^{p'}$-solvability of the Poisson-Dirichlet problem (it is only here where we use  that $(\Di_{p'}^{L^*})$ is solvable in $\Omega$) with the $\m N$-$\n C$ duality  and the heuristic that $L^*v_f$ should belong to ${\bf C}_{2_*,p}$. Related estimates were shown in Sections 4 and 5 of \cite{mt22}, although our argument for (\ref{eq.test}) is different and does not follow from the arguments in \cite{mt22}; instead, we appeal substantially to  the proof of Proposition \ref{prop.ext}.

Our manuscript is the first paper that considers the natural connections between the $\m N$-$\n C$ duality, the $L^p$-solvability of the Poisson problem, and the duality between $(\Di_{p'}^{L^*})$ and $(\Reg_p^L)$. It is the new marriage of these ideas that affords us the solution of the regularity problem for DKP operators in such large geometric generality.

\subsection{Historical remarks}\label{sec.history}  Let us provide some historical context to our investigation. For elliptic operators $L=-\dv A\nabla$,   solvability of $(\Di_{p'}^{L^*})$ and $(\Reg_p^L)$ in rough domains has garnered a lot of attention in the last several decades. For a  historical overview   of the case  $L=-\Delta$   in rough domains, we refer the reader to the introduction of \cite{mt22}.  Hence we shall focus now on the rich history related to the boundary value problems for the operators $L$ with rough (real) coefficients. The problems are always homogeneous, unless stated otherwise.

\subsubsection{Early history of the homogeneous problems with singular data} In 1924, the celebrated result of Wiener  characterized the domains for which the continuous Dirichlet problem for the Laplacian is solvable (now known as \emph{Wiener regular domains}). At the end of the 1950s, the De Giorgi-Nash-Moser theory opened the door for the systematic study of boundary value problems for elliptic operators $L=-\dv A\nabla$ whose matrices $A$ are merely bounded and elliptic, and only a few years later, Littman, Stampacchia, and Weinberger \cite{lsw} proved that the continuous Dirichlet problem for $L$ with merely bounded and elliptic coefficients is well-posed on a domain if and only if the domain is Wiener regular. In other words, the \emph{continuous} Dirichlet problem is well-posed for the Laplacian if and only if it is well-posed for $L$.

Near the end of the 1970s, Dahlberg  proved that $(\Di_2^{-\Delta})$ is solvable in Lipschitz domains \cite{dah1, dah2}, and it was then understood that in Lipschitz domains, the solvability of $(\Di_{p'}^L)$ for some $p>1$ is equivalent to  quantitative absolute continuity of the $L$-elliptic measure $\omega_L$ (namely, that $\omega_L\in A_{\infty}(\sigma)$ in a local sense). Despite the results of Caffarelli, Fabes, Mortola, and Salsa \cite{cfms}, who proved that non-negative solutions $u$ to $Lu=0$ in the unit ball have nontangential limits at $\omega_L$-a.e.\ point in the  boundary for symmetric operators $L$ with merely bounded and elliptic coefficients,  examples were soon found by Caffarelli, Fabes, Kenig \cite{cfk}, and independently by Modica and Mortola \cite{mm}, of elliptic measures which were mutually singular with respect to the surface measure on smooth domains, in   contrast to the situation of the continuous Dirichlet problem described earlier.

In the positive direction, for symmetric operators $L$ with smooth and bounded coefficients, Jerison and Kenig \cite{jk1} solved $(\Di_2^L)$ in Lipschitz domains, relying on a so-called Rellich identity, available only in the symmetric case.  On $C^1$ domains and for symmetric, globally continuous operators, the solvability of $(\Di_p^L)$ for any $p>1$ became completely understood in the work of Fabes, Jerison, and Kenig \cite{fjk}, who characterized the symmetric operators $L$ for which $\omega_L\ll\sigma$ in terms of a Dini  condition; this condition   also guaranteed that $(\Di_p^L)$ is solvable for any $p>1$.

Although \cite{fjk} had completely characterized the case of symmetric, globally continuous operators on $C^1$ domains, the characterization remained open for non-symmetric elliptic operators with rough coefficients, and perhaps more importantly, it was hard to verify the Dini-type condition of \cite{fjk} for certain important operators that arose in practice, except for $t$-independent operators, which we will briefly discuss a bit later.

\subsubsection{Dahlberg's conjectures} As mentioned in   \cite{kp3,dpp07},  Dahlberg posed two conjectures in 1984.   The first conjecture concerned whether the $L^p$-solvability of the Dirichlet problem was stable under certain perturbations of the coefficients. More precisely, if $L_0=-\dv A_0\nabla$ and $L=-\dv A\nabla$ are two  elliptic operators, such that $(\Di_p^{L_0})$ is solvable in the half-space $\bb R^{n+1}_+$ and
\begin{equation}\label{eq.fkp}
\frac{\esssup_{y\in B(x,\delta(x)/2)}|A(y)-A_0(y)|^2}{\delta(x)}\,dx\quad\text{is a Carleson measure,}
\end{equation}
does it follow that $(\Di_q^L)$ is solvable in $\bb R^{n+1}_+$, for some $q>1$ possibly larger than $p$? This question was resolved in the affirmative by Fefferman, Kenig, and Pipher in \cite{fkp}. The condition (\ref{eq.fkp}) has come to be known as the \emph{Fefferman-Kenig-Pipher} condition (FKP), and its applicability in rough domains has continued to be heavily studied in recent years \cite{mpt14,chm,chmt,ahmt,mp20,fp}. For a more thorough review of these results, see \cite{ahmt} or \cite{fp}. Let us also remark that the FKP condition is known to also preserve the solvability of $(\Reg_p^L)$ (with some $q$ perhaps larger than $p$); this was shown for symmetric operators on the ball by Kenig and Pipher in \cite{kp2}, and it has been recently extended to the case of rough domains with Harnack chains and non-symmetric operators by Dai, Feneuil and Mayboroda in \cite{dfm22}.
 
Dahlberg's second conjecture asked whether for some $p>1$, $(\Di_p^L)$ is solvable in the half-space when $L$ is a DKP operator (see Definition \ref{def.cond}). As mentioned in Section \ref{sec.dkp}, this question was resolved in the affirmative by Kenig and Pipher in \cite{kp3}. Dindo\v{s}, Petermichl, and Pipher later proved in \cite{dpp07} that for any $p>1$, if $A$ is a (possibly non-symmetric) $K$-DKP matrix with small enough constant $K=K(p)>0$ and $\Omega$ is a Lipschitz domain with small enough Lipschitz constant $\ell=\ell(p)>0$, then $(\Di_p^L)$ is solvable in $\Omega$; under the same assumptions, $(\Reg_p^L)$ is also solvable in $\Omega$, but this was shown a decade later by Dindo\v{s}, Pipher, and Rule \cite{dpr17}. Via the theory of FKP perturbations described above, the respective results of \cite{dpp07,dpr17} were also generalized in the respective articles to the $\tau$-DPR operators with $\tau=\tau(p)>0$ small enough. 
 
\subsubsection{Further results} Let us return to the case of $t$-independent real operators. Although the symmetric $t$-independent case was already understood by the work of \cite{jk1},    the first non-trivial positive result for non-symmetric $t$-independent operators on the half-space $\bb R^{n+1}_+$ was obtained by Kenig, Koch, Pipher, and Toro \cite{kkpt} when $n=1$. There, the authors showed that there exists $p>1$ such that $(\Di_{p'}^L)$ is solvable, and unlike the symmetric $t$-independent case, $p$ may not be made precise.  Later, Kenig and Rule showed in \cite{kr} that $(\Di_{p'}^{L^*})\implies(\Reg_p^L)$ for the $t$-independent operators when $n=1$, and thus it followed that $(\Reg_p^L)$ was solvable for some $p>1$. These results of \cite{kkpt} and \cite{kr} were generalized to $n\geq2$ by Hofmann, Kenig, Mayboroda, and Pipher in \cite{hkmps} and \cite{hkmpr}, where the solution to the Kato problem \cite{hlm02, ahlmct} played a   role in the proof.     We remark that there is an extensive literature regarding solvability of boundary value problems in the complex, $t$-independent case, which we do not review. For a recent review in this direction, see  \cite{bhlmp}.
 
We now comment on the solvability of the boundary value problems for DKP operators in rough domains beyond the Lipschitz setting. As mentioned in Section \ref{sec.dkp}, it has been known from \cite{kp3, jk3, dj, ahmmt} that if $L$ is a DKP operator and $\Omega$ satisfies the corkscrew condition and the IBPCAD condition, then there exists $p>1$ so that $(\Di_{p'}^{L^*})$ is solvable in $\Omega$. However, the solvability of $(\Reg_p^{-\Delta})$ beyond the Lipschitz setting remained an open problem for over 30 years until the recent work of the first and third author of this manuscript \cite{mt22}, where it was shown that $(\Di_{p'}^{-\Delta})\iff(\Reg_p^{-\Delta})$ whenever the bounded domain $\Omega$ satisfies the corkscrew condition and has $n$-Ahlfors regular boundary, extending the duality between the Dirichlet problem and regularity problem for the Laplacian shown by Jerison and Kenig in the Lipschitz setting  \cite{jk5}. 

We turn to a few words on the duality between $(\Di_{p'}^{L^*})$ and $(\Reg_p^L)$. In the case that  $\Omega$ is a star-like Lipschitz domain, Kenig and Pipher showed in \cite{kp} that $(\Reg_p^L)\implies(\Di_{p'}^{L^*})$, but the converse direction has remained an open problem. Shen proved in \cite{shen07} that if $(\Di_{p'}^{L^*})$ is solvable and $(\Reg_q^L)$ is solvable for some $q>1$, then $(\Reg_p^L)$ is solvable; later this result was extended by Dindo\v{s} and Kirsch \cite{dk12}. Complex analogues of the duality result of \cite{kp} are hopeless in the full generality, as was shown by Mayboroda in \cite{may10}; however, Hofmann, Kenig, Mayboroda, and Pipher  \cite{hkmpr} proved $(\Di_{p'}^{L^*})\iff(\Reg_p^L)$  in the range $p \in (1,2+\ve)$ for  elliptic equations with bounded $t$-independent coefficients under the assumption that De Giorgi-Nash-Moser estimates hold, and Auscher and one of us \cite{am14}  proved the equivalence for $p\in(p_0,2]$ for elliptic systems with $t$-independent complex coefficients assuming De Giorgi-Nash-Moser estimates, for some $p_0<1$ determined by the  exponent in the assumed interior H\"older condition. 

\subsubsection{Connections to the literature on the Poisson problem}\label{sec.connections} So far, we have discussed results only for the homogeneous problem (\ref{eq.dirichlet}). As remarked in Section \ref{sec.main1}, estimates controlling the left-hand side of (\ref{eq.ntmaxn}) have never appeared explicitly for the Poisson-Dirichlet problem (\ref{eq.poisson})  in a domain $\Omega\subsetneq\bb R^{n+1}$. Nevertheless, related estimates for the inhomogeneous case have been studied extensively \cite{jk4, fmm98, mm07, dindos08, mm11, hmm, bm16, barton21}. Let us briefly describe some of these results and explain their connection to our Theorem \ref{thm.poisson}. Given $p\in[1,\infty]$ and $\theta\in(-\infty,+\infty)$, the Bessel potential space   is defined by $L^p_\theta:=\{(I-\Delta)^{-\theta/2}H:H\in L^p(\bb R^{n+1})\}$ with norm
\[
\Vert w\Vert_{L^p_\theta}:=\Vert(I-\Delta)^{\theta/2}w\Vert_{L^p(\bb R^{n+1})}.
\]
For $\theta\geq0$, define the Sobolev space $L^p_\theta(\Omega)$ as the space of restrictions of functions in $L^p_\theta$ to $\Omega$. For $\theta>0$ and $p\in(1,\infty)$, let $L^{p'}_{\theta,0}(\Omega)$ be the space of functions in $L^{p'}_\theta$ supported on $\overline{\Omega}$, and let $L^p_{-\theta}(\Omega)=(L^{p'}_{\theta,0}(\Omega))^*$. By layer potential techniques and the Calder\'on-Zygmund theory of singular integrals, one has the following result on $C^\infty$ bounded domains \cite[Theorem 0.3]{jk4}: if $p\in(1,\infty)$ and $\theta>1/p$, then for every $H\in L^p_{\theta-2}(\Omega)$, there is a unique solution $w\in L^p_\theta(\Omega)$ to the   problem $-\Delta w=H$ in $\Omega$, $w=0$ on $\partial\Omega$; and moreover,
\begin{equation}\label{eq.estp}
\Vert w\Vert_{L^p_\theta(\Omega)}\leq C\Vert H\Vert_{L^p_{\theta-2}(\Omega)},\qquad\text{for each }H\in L^p_{\theta-2}(\Omega).
\end{equation}
Note that when $p=2$ and $\theta=1$, (\ref{eq.estp}) reduces to a very simple estimate obtained via the Lax-Milgram theorem.  Furthermore, the estimates (\ref{eq.estp}) are equivalent to estimates on $L^p$ for singular integral operators whose kernels are fractional gradients of the Green's function $G(x,y)$. However,  in the case of Lipschitz domains, the estimates (\ref{eq.estp}) do not hold anymore for the full range $p\in(1,\infty)$ and $\theta>1/p$. The sharp ranges of $p$ and $\theta$ for which (\ref{eq.estp}) holds in Lipschitz domains were found by Jerison and Kenig in \cite{jk4}. 

The estimates (\ref{eq.estp}) are linked   to the solvability of the homogeneous Dirichlet problem (\ref{eq.dirichlet}) with boundary data in Besov spaces $\dt B_\theta^{p,p}(\partial\Omega)$, $\theta\in(0,1)$ (for a definition, see \cite[Chapter 2]{bm16}). These boundary value problems with fractional data constitute a scale between the homogeneous Dirichlet problem at one end, and the homogeneous regularity problem at the other end. For $p>1$ and $\theta\in(0,1)$, define $Z(p,\theta)$ as the space of functions $F\in L^2_{\loc}(\Omega)$  for which
\begin{equation}\label{eq.z}
	\Vert F\Vert_{Z(p,\theta)}:=\Big(\int_{\Omega}\Big(\dashint_{B(x,\delta(x)/2)}|F|^2\,dm\Big)^{\frac p2}\delta(x)^{p(1-\theta)-1}\,dm(x)\Big)^{\frac1p}<+\infty.
\end{equation}
If $\Omega=\bb R^{n+1}_+$, $n\geq2$, and $A$ is a real, $t$-independent matrix, then for certain choices of $p\in(1,\infty)$ and $\theta\in(0,1)$, it was shown in \cite[Corollary 3.24]{bm16} that there exists a unique solution $u$ to the problem (\ref{eq.dirichlet}) with $g\in\dt B_\theta^{p,p}(\partial\Omega)$ and such that $\Vert\nabla u\Vert_{Z(p,\theta)}\lesssim\Vert g\Vert_{\dt B_\theta^{p,p}(\partial\Omega)}$. 

Similar solvability results in Besov spaces were already known for the Laplacian in Lipschitz domains \cite{jk4,fmm98,mm07, mm11}. Barton and Mayboroda also obtained the well-posedness of the inhomogeneous Dirichlet problem with Besov boundary data: they show in \cite[Theorem 3.25]{bm16} that for the same choices of $p$ and $\theta$ as for the homogeneous problem, the problem $Lw=-\dv F$ in $\Omega$, $w=g$ on $\partial\Omega$ for $g\in\dt B_\theta^{p,p}(\partial\Omega)$ and $F\in Z(p,\theta)$ is uniquely solvable, and
\begin{equation}\label{eq.z3}
\Vert\nabla w\Vert_{Z(p,\theta)}\lesssim\Vert g\Vert_{\dt B_\theta^{p,p}(\partial\Omega)}+\Vert F\Vert_{Z(p,\theta)}.
\end{equation}
Note that  (\ref{eq.z3}) is an estimate on a weighted energy integral of $ w$, and  is quite different in nature from our $L^{p'}$ estimate on $\wt{\m N}_{2^*}(w)$ (\ref{eq.ntmaxn}).

We now turn to results in the literature which are close to  (\ref{eq.ntmaxn}) and (\ref{eq.estreg}). Some closely related estimates for the Poisson problem appeared in \cite[Appendix B]{dindos08}.  In the  case that $\Omega=\bb R^{n+1}_+$, $n\geq1$, if $L$ and $L^*$ are complex bounded elliptic operators with $t$-independent coefficients satisfying De Giorgi-Nash-Moser bounds,  it was shown in  \cite{hmm}  that if $L^{-1}_{\bb R^{n+1}}$ is the   operator whose integral kernel is the fundamental solution of $L$ in $\bb R^{n+1}$, then \cite[Propositions 4.6 and 5.1]{hmm}
\begin{equation}\label{eq.hmm1}
\Vert\wt{\m N}_2\big(\nabla L^{-1}_{\bb R^{n+1}}(H-\dv F)\big)\Vert_{L^p(\partial\bb R^{n+1}_+)}\lesssim\Vert\wt{\m A}_1(tH)\Vert_{L^p(\partial\bb R^{n+1}_+)}+\Vert\m A_2(F)\Vert_{L^p(\partial\bb R^{n+1}_+)},\qquad p\in(1,2+\ep).
\end{equation}
This estimate is related to (\ref{eq.estreg}), in light of (\ref{eq.tentandcar}). They used (\ref{eq.hmm1}) to eventually obtain non-tangential and area-integral estimates for the boundary layer potentials of complex elliptic operators whose coefficients satisfy a small Carleson measure condition. Finally,  in \cite{barton21}, for $\Omega=\bb R^{n+1}_+$ and $L$ a bounded  elliptic operator with $t$-independent complex coefficients, Barton also obtained  (\ref{eq.hmm1}) in the case that $H\equiv0$, and  the estimate  $\Vert\wt{\m N}_2(L^{-1}_{\bb R^{n+1}}\dv F)\Vert_{L^{p'}(\partial\bb R^{n+1}_+)}\lesssim\Vert\n C_2(F)\Vert_{L^{p'}(\partial\bb R^{n+1}_+)}$ (see \cite[Corollary 4.27]{barton21}), for essentially optimal ranges of $p$. Note that this last estimate is related to  (\ref{eq.pd1}) when $H\equiv0$. In fact, Barton's estimates are more general, since she considers higher-order elliptic equations and square function estimates as well.

\subsection{Related results} As we were finishing the first version of this manuscript, we learned that M. Dindo\v{s}, S. Hofmann, and J. Pipher \cite{dhp22} had simultaneously and independently obtained the conclusion of Corollary \ref{cor.reg2} in unbounded Lipschitz graph domains in $\bb R^{n+1}$, $n\geq1$, when $L$ is a DPR operator, via a different method. Both the first version of this manuscript \cite{mpt22} and the first version of the paper of  M. Dindo\v{s}, S. Hofmann, and J. Pipher \cite{dhp22} were posted on arXiv.org on July 21, 2022. Their proof of their analogue of Corollary \ref{cor.reg2} is significantly shorter than ours, owing to the more restrictive geometric setting that they study. Indeed, they perform a clever reduction of the problem to the solvability of $(\Reg_q^{L_0})$ for some $q>1$, where $L_0$ is a block form operator, but this approach does not work in more general domains since it necessitates the existence of a uniform ``preferred direction'' at the boundary, as occurs in a domain above a Lipschitz graph. The passage from Lipschitz graph domains to uniformly $n$-rectifiable domains, or even to chord-arc domains, is highly non-trivial, and requires a novel approach via the Poisson-Dirichlet problem. Even in the case of the Laplacian, the same problem had remained open for over 30 years until its solution in \cite{mt22}, illustrating that the geometry of the domain creates significant difficulties which are hard to overcome.

\subsection{Outline} The rest of the paper is organized as follows. In Section \ref{sec.prelim}, we provide basic definitions and lemmas. In Section \ref{sec.poisson}, we prove Theorems \ref{thm.poisson} and \ref{thm.poisson2} on the $L^p$-solvability of the Poisson problem. In Section \ref{sec.strat}, we give an overview of the structure of the proof of Theorem \ref{thm.regularity} on the regularity problem for DKP operators. In Section \ref{sec.corona}, we construct the corona decomposition tailored to exploit the results of \cite{dpr17} on each corona subdomain. In Section \ref{sec.extension}, we define the almost $L$-elliptic extension based on the corona decomposition of Section \ref{sec.corona}. Then, in Section \ref{sec.reg} we finish the proof of Theorem \ref{thm.regularity} by proving an estimate on the conormal derivative of solutions to Poisson problems (see Proposition \ref{prop.conormal}). Finally, in Appendix \ref{sec.app} we provide proofs for some  lemmas of Section \ref{sec.prelim}.

\section{Preliminaries}\label{sec.prelim}

We  write $a\lesssim b$ if there exists a constant $C>0$ so that $a\leq Cb$ and $a\lesssim_tb$ if $C$ depends on the parameter $t$. We write $a\approx b$ to mean $a\lesssim b\lesssim a$ and define $a\approx_tb$ similarly.

All measures in this paper are assumed to be Radon measures. We denote by $m=m_{n+1}$ the Lebesgue measure on $\bb R^{n+1}$. Given a measure $\mu$ on $\bb R^{n+1}$, the \emph{Hardy-Littlewood maximal function} of a measurable function $f$ on $\supp\mu$ is defined as
\[
\m M_\mu(f)(\xi):=\sup_{B:\xi\in B}\frac1{\mu(B)}\int_B|f|\,d\mu,\qquad\xi\in\supp\mu,
\]
where the supremum is taken over all possible balls $B$ centered at $\supp\mu$.

Given $q\in(0,\infty]$, a set $U\subset\bb R^{n+1}$, a measure $\mu$ on $\bb R^{n+1}$, and a function $H\in L^q(U,\mu)$, denote the \emph{$L^q(\mu)$ mean of $H$ over $U$} as
\[
m_{q,U,\mu}(H):=\Big(\frac1{\mu(U)}\int_U|H(x)|^q\,d\mu\Big)^{1/q}.
\]
If $\mu$ is the Lebesgue measure,  we   write $m_{q,U}=m_{q,U,\mu}$. If $q=1$, we   write $m_{U,\mu}=m_{q,U,\mu}$.

\subsection{Geometric preliminaries}

\subsubsection{Lipschitz domains}\label{sec.lip}  We say that $Z\subset \R^{n+1}$ is
an $\ell$-cylinder of diameter $d$ if there is a coordinate system $(x,t)\in\R^{n}\times \R$ such that $Z=\{(x,t):|x|\leq d,-2\ell d\leq |t|\leq 2\ell d\}$. Also, for all $s>0$, we denote $sZ=\{(x,t):|x|\leq sd,-2\ell d\leq |t|\leq 2\ell d\}$.

We say that $\Omega$ is a \emph{Lipschitz domain} with Lipschitz character $(\ell,N,C_0)$ if there is $r_0>0$ and at most $N$ $\ell$-cylinders
$Z_j$, $j=\overline{1,N}$ of diameter $d$ with $C_0^{-1}r_0\leq d\leq C_0r_0$ such that\footnote{We could also say that $\Omega$ has character
	$(\ell,N,r_0,C_0)$. However, notice that, since $\pom$ is covered by at most $N$ cylinders of diameter comparable to $r_0$, we have $r_0\gtrsim \diam(\pom)/N$. So $r_0/\diam(\pom)$ depends on $N$.}
\begin{itemize}
	\item $8Z_j \cap\pom$ is the graph of a Lipschitz function $\phi_j$ with $\|\nabla \phi_j\|_\infty\leq \ell$, $\phi_j(0)=0$,
	\item $\pom=\bigcup_j (Z_j\cap \pom)$,
	\item We have that $8Z_j\cap\Omega=\{(x,t)\in8Z_j: \phi_j(x)>t\}$.
\end{itemize}
We also say that $\Omega$ is a  Lipschitz domain  with Lipschitz constant $\ell$.

\subsubsection{Quantitative conditions on the geometry of domains}\label{sec.geom} A measure $\mu$ in $\bb R^{n+1}$ is called \emph{$n$-Ahlfors regular}   if there exists some constant $C_0>0$ such that
\[
C_0^{-1}r^n\leq\mu(B(x,r))\leq C_0r^,\qquad\text{for all }x\in\supp\mu\text{ and }0<r\leq\diam(\supp\mu).
\]
A measure $\mu$ is \emph{uniformly $n$-rectifiable} if it is $n$-Ahlfors regular and there exist constants $\theta,M>0$ such that for each $x\in\supp\mu$ and each $r\in(0,\diam(\supp\mu)]$, there is a Lipschitz mapping $g$ from the $n$-dimensional ball $B_n(0,R)$ to $\bb R^{n+1}$ with $\Lip(g)\leq M$ and satisfying the bound $\mu\big(B(x,r)\cap g(B_n(0,r))\big)\geq\theta r^n$. A set $E\subset\bb R^{n+1}$ is $n$-Ahlfors regular if $\m H^n|_E$ is $n$-Ahlfors regular, where $\m H^n$ is the $n$-dimensional Hausdorff measure, which we assume to be normalized so that it coincides with $m_n$ in $\bb R^n$.   Also, $E$ is uniformly $n$-rectifiable if $\m H^n|_E$ is uniformly $n$-rectifiable.  The notion of  uniform rectifiability was introduced in \cite{ds1} and \cite{ds2}, and it  should be considered a quantification of rectifiability.

By a \emph{domain} we mean a connected open set. In this paper, $\Omega$ is always a domain in $\bb R^{n+1}$ with $n\geq2$. As mentioned in the introduction, we denote the restriction of the $n$-Hausdorff measure to $\partial\Omega$ by $\sigma$, and we call it the \emph{surface measure} on $\partial\Omega$. 

We say that $\Omega$ satisfies the \emph{corkscrew condition} if there exists   $c>0$ such that for each $x\in\partial\Omega$ and every $r\in(0,2\diam\Omega)$, there exists a ball $B\subset B(x,r)\cap\Omega$ so that $r(B)\geq cr$.

Given two points $x,x'\in\Omega$, and a pair of numbers $M,N\geq1$, an $(M,N)$-\emph{Harnack chain connecting $x$ to $x'$}, is a chain of open balls $B_1,\ldots,B_N\subset\Omega$, with $x\in B_1,x'\in B_N, B_k\cap B_{k+1}\neq\varnothing$ and $M^{-1}\diam B_k\leq\dist(B_k,\partial\Omega)\leq M\diam B_k$. We say that $\Omega$ satisfies the \emph{Harnack chain condition} if there is a uniform constant $M$ such that for any $x,x'\in\Omega$, there exists an $(M,N)$-Harnack chain connecting them, with $N$ depending only on $M$ and on $|x-x'|/\min\{\delta(x),\delta(x')\}$.

We say that  $\Omega$ is a \emph{chord-arc domain} if $\Omega$ satisfies the Harnack chain condition, if both $\Omega$ and $\bb R^{n+1}\backslash\overline{\Omega}$ satisfy the corkscrew condition, and if $\partial\Omega$ is $n$-Ahlfors regular.

We say that  $\Omega$ has \emph{Interior Big Pieces of Chord-Arc Domains}, or that $\Omega$ satisfies the IBPCAD condition, if there exist positive constants $\eta$ and $C$, and $N\geq2$, such that for every $x\in\Omega$, with $\delta_\Omega(x)<\diam(\partial\Omega)$, there is a chord-arc domain $\Omega_x\subset\Omega$ satisfying
\begin{itemize}
	\item $x\in\Omega_x$.
	\item $\dist(x,\partial\Omega_x)\geq\eta\delta_\Omega(x)$.
	\item $\diam(\Omega_x)\leq C\delta_\Omega(x)$.
	\item $\sigma(\partial\Omega_x\cap B(x,N\delta_\Omega(x)))\geq\eta\sigma(B(x,N\delta_\Omega(x)))\approx_N\eta\delta_\Omega(x)^n$.
	\item The chord-arc constants of the domains $\Omega_x$ are uniform in $x$.
\end{itemize}

\subsubsection{Dyadic lattices and the $\beta$ numbers} \label{sec.lattice}

Given an  $n$-Ahlfors regular measure $\mu$ in $\bb R^{n+1}$, we consider the dyadic lattice $\m D_\mu$ of ``cubes'' built by David and Semmes \cite[Chapter 3 of Part I]{ds2}. The properties satisfied by $\m D_\mu$ are the following. Assume first, for simplicity, that $\diam(\supp\mu)=\infty$. Then for each $j\in\bb Z$ there exists a family $\m D_{\mu,j}$ of Borel subsets of $\supp\mu$ (the dyadic cubes of the $j$-th generation) such that:
\begin{enumerate}[(a)]
	\item each $\m D_{\mu,j}$ is a partition of $\supp\mu$, i.e. $\supp\mu=\bigcup_{Q\in\m D_{\mu,j}}Q$ and $Q\cap Q'=\varnothing$ whenever $Q,Q'\in\m D_{\mu,j}$ and $Q\neq Q'$;
	\item if $Q\in\m D_{\mu,j}$ and $Q'\in\m D_{\mu,k}$ with $k\leq j$, then either $Q\subset Q'$ or $Q\cap Q'=\varnothing$;
	\item if for all $j\in\bb Z$ and $Q\in\m D_{\mu,j}$, we have $2^{-j}\lesssim\diam(Q)\leq2^{-j}$ and $\mu(Q)\approx2^{-jn}$.
	\item there exists $C>0$ such that, for all $j\in\bb Z$, $Q\in\m D_{\mu,j}$, and $0<\tau<1$,
	\begin{equation}\nonumber
		\mu\big(\{x\in Q:\dist(x,\supp\mu\backslash Q)\leq\tau2^{-j}\}\big)+\mu\big(\{x\in\supp\mu\backslash Q:\dist(x,Q)\leq\tau2^{-j}\}\big)\leq C\tau^{1/C}2^{-jn}.
	\end{equation}
	This property is known as the \emph{thin boundary condition}, and it implies the existence of a point $x_Q\in Q$ (the \emph{center} of $Q$) such that $\dist(x_Q,\supp\mu\backslash Q)\gtrsim2^{-j}$ (see \cite[Lemma 3.5 of Part I]{ds2}).
\end{enumerate}
We set $\m D_{\mu}:=\bigcup_{j\in\bb Z}\m D_{\mu,j}$. If $\diam(\supp\mu)<\infty$, the families $\m D_{\mu,j}$ are only defined for $j\geq j_0$, with $2^{-j_0}\approx\diam(\supp\mu)$, and the same properties above hold for $\m D_{\mu}:=\bigcup_{j\geq j_0}\m D_{\mu,j}$.

Given a cube $Q\in\m D_{\mu,j}$, we say that its side length is $2^{-j}$ and denote it by $\ell(Q)$. Notice that $\diam Q\leq\ell(Q)$. We also denote
\begin{equation}\label{eq.ball}\nonumber
	B(Q):=B(x_Q,c_1\ell(Q)),\qquad B_Q:=B(x_Q,\ell(Q))\supset Q,
\end{equation}
where $c_1>0$ is chosen so that $B(Q)\cap\supp\mu\subset Q$, for all $Q\in\m D_\mu$.  For $\lambda>1$, we write $$\lambda Q = \bigl\{x\in \supp\mu:\, \dist(x,Q)\leq (\lambda-1)\,\ell(Q)\bigr\}.$$

The side length of a ``true cube'' $P\subset\R^{n+1}$ is also denoted by $\ell(P)$. On the other hand, given a ball $B\subset\R^{n+1}$, its radius is denoted by $r(B)$. For $\lambda>0$,   $\lambda B$ is the ball concentric with $B$ with radius $\lambda\,r(B)$.

Given $E\subset\R^{n+1}$, a ball $B$, and a hyperplane $L$, we denote
$$b\beta_{E}(B,L) =  \sup_{y\in E\cap B} \frac{\dist(y,L)}{r(B)} + 
\sup_{y\in L\cap B}\!\! \frac{\dist(x,E)}{r(B)} .$$
We set $b\beta_{E}(B,L) = \inf_L b\beta_{E}(x,r,L)$, where the infimum is taken over all hyperplanes $L\subset\R^{n+1}$. For   $B=B(x,r)$, we also write $b\beta_{E}(x,r,L)=b\beta_{E}(B,L)$, and $b\beta_{E}(x,r)=b\beta_{E}(B)$.

For $p\geq1$, a measure $\mu$, a ball $B$, and a hyperplane $L$, we set
$$\beta_{\mu,p}(B,L) = \left(\frac1{r(B)^n}\int_B \left(\frac{\dist(x,L)}{r(B)}\right)^p\,d\mu(x)\right)^{1/p}.$$
We define $\beta_{\mu,p}(B) = \inf_L \beta_{\mu,p}(B,L)$, where the infimum is taken over all hyperplanes $L$.
For $B=B(x,r)$, we also write $\beta_{\mu,p}(x,r,L) = \beta_{\mu,p}(B,L)$, and $\beta_{\mu,p}(x,r) = \beta_{\mu,p}(B)$. For a given cube $Q\in\DD_\mu$, we define:
\begin{align*}
	\begin{array}{ll}
		\beta_{\mu,p}(Q,L)  = \beta_{\mu,p}(B_Q,L)
		,&\quad \beta_{\mu,p}(\lambda Q,L)= \beta_{\mu,p}(\lambda B_Q,L),\\
		\quad \,\beta_{\mu,p}(Q) = \beta_{\mu,p}(B_Q),
		& \quad \quad\,\beta_{\mu,p}(\lambda Q)= \beta_{\mu,p}(\lambda B_Q).
	\end{array}
\end{align*}
Also, we define similarly $$b\beta_\mu(Q,L),\quad b\beta_\mu(\lambda Q,L),\quad b\beta_\mu(Q),\quad b\beta_\mu(\lambda Q),$$
by identifying these coefficients with the analogous ones in terms of $B_Q$. These coefficients are defined in the same way as $b\beta_{\supp\mu}(B,L)$ and $b\beta_{\supp\mu}(B)$,
replacing again $B$ by $Q\in\DD_\mu$ or $\lambda Q$.

\subsubsection{The Haj\l{}asz-Sobolev spaces}\label{sec.sobolev}    Let $(\Sigma,\sigma)$ be a metric space with $\sigma$ a doubling
measure on $\Sigma$, which means that there is a uniform constant $C_\sigma\geq1$ such that $\sigma(B(x,2r))\leq C_\sigma\, \sigma(B(x,r))$, for all $x\in \Sigma$ and $ r>0$. 

For a Borel function $f:\Sigma\to\R$, we say that a non-negative Borel function $g:\Sigma \to \R$ is a {\it Haj\l asz upper gradient of  $f$} if  
\begin{equation}\label{eq:H-Sobolev}\nonumber
	|f(x)-f(y)| \leq |x-y| \,(g(x)+g(y))\quad \mbox{ for $\sigma$-a.e.\ $x, y \in \Sigma$.} 
\end{equation}
We denote the collection of all the Haj\l asz upper gradients of $f$ by $D(f)$. For $p\geq1$, we denote by $\dt{W}^{1,p}(\Sigma)$ the space of Borel functions $f$ which have 
a Haj\l asz upper gradient in $L^p(\sigma)$, and we let $W^{1,p}(\Sigma)$ be the space of functions $f\in L^p(\sigma)$ which have a Haj\l asz upper gradient in $L^p(\sigma)$; that is,  $W^{1,p}(\Sigma)=  \dt W^{1,p}(\Sigma) \cap L^p(\sigma)$.
We  define the semi-norm  
\begin{equation}\label{eqseminorm}\nonumber
	\| f \|_{\dt W^{1.p}(\Sigma)} = \inf_{g \in D(f)} \| g\|_{L^p(\Sigma)}
\end{equation}
and the scale-invariant norm
\begin{equation}\label{eqnorm}\nonumber
	\| f\|_{W^{1,p}(\Sigma)} = \diam(\Sigma)^{-1} \|f\|_{L^p(\Sigma)} +   \inf_{g \in D(f)} \| g\|_{L^p(\Sigma)}.
\end{equation}
For any   metric space $\Sigma$, if $p\in (1,\infty)$, from the uniform convexity of $L^p(\sigma)$, one easily deduces
that the infima in the definitions of   $\|\cdot\|_{W^{1,p}(\Sigma)}$ and $\|\cdot\|_{\dt W^{1,p}(\Sigma)}$   are uniquely attained. We denote by $\nabla_{H,p} f$ the function $g$ which attains the infimum. The spaces $\dt{W}^{1,p}(\Sigma)$ and ${W}^{1,p}(\Sigma)$ were introduced in \cite{haj96}, and are known as the {\it Haj\l{}asz-Sobolev spaces}.

\subsubsection{Properties of the Carleson and non-tangential maximal operators}\label{sec.carleson} 

First we state the well-known change-of-aperture lemma for the non-tangential maximal functions.
\begin{lemma}\label{lm.ntchange} Fix $p>1$, $\alpha,\beta>0$, and $\hat c_1,\hat c_2\in(0,1/2]$. We have that
\begin{equation}\label{eq.changeavg}\nonumber
\Vert\widetilde{\m N}_{\alpha,\hat c_1,r}(u)\Vert_{L^p(\partial\Omega)}\approx_{\alpha,\beta,n,p,(\hat c_2/\hat c_1)}\Vert\widetilde{\m N}_{\beta,\hat c_2,r}(u)\Vert_{L^p(\partial\Omega)},\qquad\text{for any }u\in L^r_{\loc}(\Omega).
\end{equation}
 \end{lemma}

The $L^p$ norms of the Carleson functionals $\n C_{\hat c,q}$ defined in (\ref{eq.carlesonfn}) are equivalent under a change of the parameter $\hat c$, as the next lemma states. We defer its proof to   Appendix \ref{sec.proofchangeavg}.

\begin{lemma}\label{lm.changeavg}Let $\hat c_1,\hat c_2\in(0,1/2]$, $q\in[1,\infty)$, and $p>1$. Then
\begin{equation}\label{eq.changeavg2}
\Vert\n C_{\hat c_1,q}(H)\Vert_{L^p(\partial\Omega)}\approx_{n, q, p, (\hat c_2/\hat c_1)}\Vert\n C_{\hat c_2,q}(H)\Vert_{L^p(\partial\Omega)},\qquad\text{for any }H\in L^q_{\loc}(\Omega).
\end{equation}
\end{lemma}

As remarked in the introduction, there is a duality between the operators $\wt {\m N}$ and $\n C$. The precise results which we shall use are stated in Proposition \ref{prop.duality} below. When $\Omega$ is the half-space, these results were shown by Hyt\"onen and Ros\'en in \cite[Theorem 3.1, Theorem 3.2]{hr13}. For the benefit of the reader, in Appendix \ref{sec.proofduality}, we give the proof of the proposition in our  setting, by appealing to the discrete vector-valued model in \cite{hr13}.

\begin{proposition}[Duality between $\n C$ and $\m N$]\label{prop.duality} Let $\Omega\subset\bb R^{n+1}$, $n\geq1$, be a domain satisfying the corkscrew condition and such that $\partial\Omega$ is $n$-Ahlfors regular. Suppose that either $\Omega$ is bounded, or that $\partial\Omega$ is unbounded. Let $p,q\in(1,\infty)$ and $p'$, $q'$ their H\"older conjugates. Then ${\bf N}_{q,p}=({\bf C}_{q',p'})^*$, and moreover,
\begin{equation}\label{eq.hr7}\nonumber
\Vert uH\Vert_{L^1(\Omega)}\lesssim\Vert\wt{\m N}_q(u)\Vert_{L^p(\partial\Omega)}\Vert\n C_{q'}H\Vert_{L^{p'}(\partial\Omega)},\qquad u\in L^q_{\loc}(\Omega), H\in L^{q'}_{\loc}(\Omega),
\end{equation}
\begin{equation}\label{eq.hr8}\nonumber
\Vert\widetilde{\m N}_q(u)\Vert_{L^p(\partial\Omega)}\lesssim\sup_{H:\Vert\n C_{q'}(H)\Vert_{L^{p'}(\partial\Omega)}=1}\Big|\int_{\Omega}Hu\,dm\Big|,\qquad u\in L^q_{\loc}(\Omega),
\end{equation}
\begin{equation}\label{eq.hr9}\nonumber
	\Vert\n C_{q'}H\Vert_{L^{p'}(\partial\Omega)}\lesssim\sup_{u:\Vert\wt{\m N}_qu\Vert_{L^p(\partial\Omega)}=1}\Big|\int_\Omega Hu\,dm\Big|,\qquad H\in L^{q'}_{\loc}(\Omega).
\end{equation}
\end{proposition}

\begin{remark} For any $p,q\in(1,\infty)$, the space ${\bf C}_{q',p'}$ is not reflexive; indeed, we have that ${\bf C}_{q',p'}\subsetneq({\bf N}_{q,p})^*$. See \cite[Theorems 2.4 and 3.2]{hr13}.
\end{remark}

To obtain the $L^{p'}$-solvability of the Poisson-Dirichlet problem with general data, we need the following approximation lemma. Its proof is deferred to Appendix \ref{sec.prooflipschitz}.

\begin{lemma}\label{lm.lipschitz} Fix $p,q\in(1,\infty)$ and $H\in{\bf C}_{q,p}$.
\begin{enumerate}[(i)]
	\item There exists a sequence of functions $\{H_k\}_{k\in\bb N}\subset{\bf C}_{q,p}$, such that each $H_k$ is compactly supported in $\Omega$, $H_k\ra H$ strongly in $L^q_{\loc}(\Omega)$, $H_k\ra H$ pointwise a.e.\ in $\Omega$, and $H_k\ra H$ strongly in ${\bf C}_{q,p}$ as $k\ra\infty$.
	\item Suppose that $H$ is compactly supported in $\Omega$. Then there exists a family of functions $\{H_\ep\}_{\ep>0}\subset\Lip(\Omega)\cap{\bf C}_{q,p}$, each of which is compactly supported in $\Omega$, such that $H_\ep\ra H$ strongly in $L^q_{\loc}(\Omega)$, and $\Vert H_\ep-H\Vert_{{\bf C}_{q,p}}\ra0$ as $\ep\ra0$.
\end{enumerate}
\end{lemma}

\subsubsection{Whitney decompositions}\label{sec.whitney} Given an open set $U\subsetneq\bb R^{n+1}$ and a number $\hat c\in(0,4]$, let $k\in\bb N$ be the unique integer such that $2^{-k}\leq\hat c/4< 2^{-k+1}$, and we say that a collection of closed cubes, $\m W_{\hat c}(U)$, is a \emph{$\hat c$-Whitney decomposition of $U$} if the interiors of the cubes in $\m W(U)=\m W_{\hat c}(U)$ are pairwise disjoint, $\cup_{I\in\m W(U)}I=U$, and moreover
\[
2^{k+2}\diam I\leq\dist(I,\partial U)\leq 2^{k+5}\diam I,\qquad\text{for each }I\in\m W(U).
\]
Such a $\hat c$-Whitney decomposition always exists; this can be attained by dyadically subdividing $k+2$ times each Whitney cube from the standard decomposition. 

Suppose that $\partial\Omega$ is  $n$-Ahlfors regular and consider the dyadic lattice $\m D_\sigma$ defined   in Section \ref{sec.lattice}. Then for each   $I\in\m W=\m W(\Omega)$ there is some cube $Q\in\m D_\sigma$ such that   
\begin{equation}\label{eq.bi}
\ell(Q)=\ell(I),\qquad\text{and}\qquad\dist(I,Q)\leq\frac{256}{\hat c}\diam I.
\end{equation} 
For any $I\in\m W(\Omega)$, we let $b(I)=b_\Omega(I)$ be the collection of all cubes $Q\in\m D_\sigma$ which satisfy (\ref{eq.bi}). If $Q\in b(I)$, we say that $Q$ is a \emph{boundary cube  of $I$}. There is a uniformly bounded number of $Q\in b(I)$, depending only on $n$ and the $n$-Ahlfors regularity of $\partial\Omega$. Conversely, given $Q\in\m D_\sigma$, we let
\begin{equation}\label{eq.wq}\nonumber
	w(Q):=\bigcup_{I\in\m W: Q\in b(I)}I.
\end{equation}
It is easy to check that $w(Q)$ is made up of at most  a uniformly bounded number of cubes $I$, but it may happen that $w(Q)=\varnothing$. Moreover, if $I\in\m W$ and $I\cap w(Q)\neq\varnothing$, then $I\subset w(Q)$ since the Whitney cubes are disjoint.

Given a Whitney cube $I\in\m W(U)$, we use the notation $I^*$ for the slightly larger cube of same center as $I$ and such that $\diam I^*=(1+\theta)\diam I$, where $\theta\in(0,1/10)$ is a small fixed constant. We also define
\begin{equation}\label{eq.wq2}
w^*(Q):=\bigcup_{I\in\m W: Q\in b(I)}I^*.
\end{equation}

\subsection{Elliptic PDE preliminaries}\label{sec.pde}

Throughout this subsection,   we assume only that $\Omega$ is a domain in $\bb R^{n+1}$, $n\geq2$. We write  $2^*=\frac{2(n+1)}{n-1}$ and that $2_*=(2^*)'=\frac{2(n+1)}{n+3}$.

Recall that $C_c^{\infty}(\Omega)$ is the space of compactly supported smooth functions in $\Omega$,  and  that for $p\in[1,\infty)$, $W^{1,p}(\Omega)$ is the Sobolev space of $p$-th integrable functions in $\Omega$ whose weak derivatives exist in $\Omega$ and are $p$-th integrable functions, while $W_0^{1,p}(\Omega)$ is the completion of $C_c^{\infty}(\Omega)$ under the norm $\Vert u\Vert_{W^{1,p}(\Omega)}:=\Vert u\Vert_{L^p(\Omega)}+\Vert\nabla u\Vert_{L^p(\Omega)}$. Moreover,   $\dt W^{1,p}(\Omega)$ consists of the $L^1_{\loc}(\Omega)$ functions whose weak gradient is $p$-th integrable over $\Omega$, and   we denote by $L^p_c(\Omega)$ the space of  $p$-th integrable functions with compact support in $\Omega$. We let $Y_0^{1,2}(\Omega)$ be the completion of $C_c^{\infty}(\Omega)$ under the norm $\Vert u\Vert_{Y^{1,2}(\Omega)}:=\Vert u\Vert_{L^{2^*}(\Omega)}+\Vert\nabla u\Vert_{L^2(\Omega)}$.

We assume throughout that $A$ is a real, not necessarily symmetric matrix satisfying (\ref{eq.elliptic}). Define the elliptic operator $L$ acting formally on real-valued functions $u$ by
\[
Lu=-\dv(A\nabla u)=-\sum_{i,j=1}^{n+1}\frac{\partial}{\partial x_i}\Big(a_{ij}\frac{\partial u}{\partial x_j}\Big).
\]
We write $A^T$ for the transpose of $A$, and $L^*=-\dv A^T\nabla$.

\subsubsection{Properties of weak solutions}

Given $H\in L^{2_*}_{\loc}(\Omega)$ and $F\in L^2_{\loc}(\Omega)$, we say that a function $w\in W^{1,2}_{\loc}(\Omega)$ solves $Lw=H-\dv F$ \emph{in the weak sense}, or that $w$ is a \emph{weak solution} of the equation $Lw=H-\dv F$, if for any $\phi\in C_c^{\infty}(\Omega)$, we have that
\[
\int_\Omega A\nabla w\nabla\phi\,dm=\int_\Omega\Big[H\phi+F\nabla\phi\Big]\,dm.
\]

The proof of the following basic inequality which weak solutions verify is standard.

\begin{lemma}[Caccioppoli inequality]\label{lm.cacc} Let $B$ be a ball in $\Omega$ and let $w\in W_{\loc}^{1,2}(B^*)$ be a weak solution to the equation $Lw=H-\dv F$ in $B^*=(1+\theta)B$, where $F\in L^2(B^*)$, $H\in L^{2_*}(B^*)$, and $\theta\in(0,1)$. Then
\begin{equation}\nonumber
\int_B|\nabla w|^2\lesssim_\theta\frac1{r(B)^2}\int_{B^*}|w|^2\,dm+\Big(\int_{B^*}|H|^{2_*}\,dm\Big)^{\frac2{2_*}}+\int_{B^*}|F|^2\,dm.
\end{equation}
\end{lemma}

The next lemma comprises the conclusions of the De Giorgi-Nash-Moser theory, a capstone result in the regularity of elliptic PDEs with rough coefficients.

\begin{lemma}\label{lm.moser} Let $B$ be a ball in $\Omega$, $B^*=(1+\theta)B$ for some $\theta\in(0,1]$, and let $u\in W^{1,2}(B^*)$ be a weak solution to the equation $Lu=0$ in $B^*$. Then there exists $\eta\in(0,1)$ so that $u$ is H\"older continuous with exponent $\eta$ in $B^*$. Moreover, for any $q>0$,
\begin{equation}\label{eq.moser}\nonumber
\max_{x\in B}|u(x)|\lesssim_{\theta,n,q}\Big(\dashint_{B^*}|u|^q\,dm\Big)^{\frac1q}\lesssim\min_{x\in\frac12B}|u(x)|,
\end{equation}
and
\[
\sup_{x,y\in\frac12B}\frac{|u(x)-u(y)|}{|x-y|^\eta}\lesssim(r(B))^{-\eta}\Big(\dashint_B|u|^2\,dm\Big)^{\frac12}.
\]
\end{lemma}

When $A$ is Lipschitz, we have a control on the size of the gradient of the solution. The precise technical result which we shall use is the following;  it is a consequence of the proof of \cite[Theorem 5.19]{gm12}.

\begin{lemma}\label{lm.gradient} Let $B$ be a ball in $\Omega$, $L=-\dv A\nabla$, with $A$ a Lipschitz matrix function verifying (\ref{eq.elliptic}) on $2B$, and suppose that $\Vert\nabla A\Vert_{L^{\infty}(2B)}\leq\frac{C}{r(B)}$. Let $u$ solve $Lu=0$ in $2B$. Then $\nabla u$ is continuous on $2B$, and
\[
\sup_{x\in B}|\nabla u(x)|\lesssim\Big(\dashint_{2B}|\nabla u(y)|^2\,dm(y)\Big)^{\frac12}.
\]
\end{lemma}

\subsubsection{Green's function and  elliptic measure}\label{sec.pdeprelim}

We start with the definition and properties of the integral kernel to the solution operator $L^{-1}$, known as the Green's function.

\begin{definition}[Green's function]\label{def.green} Assume that $\Omega\subset\bb R^{n+1}$ is an open, connected set, and let $L=-\dv A\nabla$ where $A$ is a real, strongly elliptic, not necessarily symmetric matrix of bounded measurable coefficients. There exists a unique non-negative function $G_L:\Omega\times\Omega\ra\bb R$, called \emph{Green's function} for $L$, satisfying the following properties:
\begin{enumerate}[(i)]
	\item For each $x,y\in\Omega$, $G_L(x,y)\lesssim|x-y|^{1-n}$, and $G_L(x,y)>0$ in $\Omega$.
	\item For each $\hat c\in(0,1)$, $x\in\Omega$ and any $y\in B(x,\hat c\delta(x))$, we have $|x-y|^{1-n}\lesssim_{\hat c} G_L(x,y)$.
	\item For each $x\in\Omega$, $G_L(x,\cdot)\in C(\overline{\Omega}\backslash\{x\})\cap W^{1,2}_{\loc}(\Omega\backslash\{x\})$  and  $G_L(x,\cdot)|_{\partial\Omega}\equiv0$.
	\item For each $x\in\Omega$,   the identity $LG_L(\cdot,x)=\delta_x$ holds in the distributional sense; that is,
	\[
	\int_\Omega A(y)\nabla_yG_L(y,x)\nabla\Phi(y)\,dm(y)=\Phi(x),\qquad\text{for any }\Phi\in C_c^{\infty}(\Omega).
	\]
	\item\label{it.represent} If $H\in L^{\frac n2,1}(\Omega)$ and $F\in L^{n,1}(\Omega)$, we have that the function
	\begin{equation}\label{eq.represent}
	w(x)=\int_\Omega G_L(y,x)H(y)\,dm(y)+\int_\Omega\nabla_yG_L(y,x)F(y)\,dm(y)
	\end{equation}
	solves the equation $L^*w=H-\dv F$ in the weak sense in $\Omega$, $w\in Y_0^{1,2}(\Omega)$, and $\Vert w\Vert_{L^{\infty}(\Omega)}\lesssim\Vert H\Vert_{L^{\frac n2,1}(\Omega)}+\Vert F\Vert_{L^{n,1}(\Omega)}$. 
	\item For each $x,y\in\Omega$ with $x\neq y$,
	\begin{equation}\label{eq.transpose}
		G_L(x,y)=G_{L^*}(y,x).
	\end{equation}
	\item\label{item.map} Define the operators $L^{-1}_\Omega$ and $L^{-1}_{\Omega}\dv$ by
	\begin{equation}\label{eq.op}
	(L^{-1}_\Omega H)(x):=\int_\Omega G_L(x,y)H(y)\,dm(y),\qquad (L^{-1}_\Omega\dv F)(x):=\int_\Omega\nabla_yG_L(x,y)F(y)\,dm(y),\qquad x\in\Omega.
	\end{equation} 
	Then we have that $L^{-1}_\Omega$ maps $L^{2_*}(\Omega)$ into $L^{2^*}(\Omega)$, and  the operator $L^{-1}_\Omega\dv$ maps $L^2(\Omega)$ into $L^{2^*}(\Omega)$.
\end{enumerate}
\end{definition}

The construction of the Green's function in bounded domains for real non-symmetric elliptic matrices with bounded and measurable coefficients may be found in \cite{gw}, while the unbounded case is shown in \cite{hk}. The sharp representation formula (\ref{eq.represent}) for $H$ and $F$ in Lorentz spaces has been shown in  \cite{mourg19}.

To define the elliptic measure, we borrow the setting of \cite{agmt22}.  Assume that $\partial\Omega$ is $n$-Ahlfors regular; in this setting, $\Omega$ is Wiener regular, so that the continuous Dirichlet problem for $L$ is solvable in $\Omega$. By the maximum principle and the Riesz Representation Theorem, there exists a family of probability measures $\{\omega^x_L\}_{x\in\Omega}$ on $\partial\Omega$ so that for each $f\in C_c(\partial\Omega)$ and each $x\in\Omega$, the solution $u$ to the continuous Dirichlet problem with data $f$ satisfies that $u(x)=\int_{\partial\Omega}f(\xi)\,d\omega^x(\xi)$. We call $\omega^x_L$ the \emph{$L$-elliptic measure with pole at $x$}.

We now turn to several well-known estimates that the $L$-elliptic measure verifies, see \cite{cfms} and \cite{hkm93}. For a proof of the next lemma, see \cite[Lemma 11.21]{hkm93}.

\begin{lemma}[Bourgain's estimate]\label{lm.bourgain} Let $\Omega\subsetneq\bb R^{n+1}$ be open with $n$-Ahlfors regular boundary. Then there exists $c>0$ depending only on $n$, $\lambda$, and the $n$-Ahlfors regularity constant, such that for any $\xi\in\partial\Omega$ and $r\in(0,\diam(\partial\Omega)/2]$, we have that $\omega^x(B(\xi,2r)\cap\partial\Omega)\geq c$,  for all $x\in\Omega\cap B(\xi,r)$.
\end{lemma}

Moreover, the preceding lemma implies the following result via a standard argument involving the maximum principle and the pointwise upper bound for the Green's function in Definition \ref{def.green} (i). For a proof in the case that $L=-\Delta$, which is readily generalized to our setting, see \cite[Lemma 3.3]{ahmmmtv}.

\begin{lemma}\label{lm.upper} Let $\Omega\subsetneq\bb R^{n+1}$ be open with $n$-Ahlfors regular boundary. Let $B=B(x_0,r)$ be a closed ball with $x_0\in\partial\Omega$ and $0<r<\diam(\partial\Omega)$. Then
\begin{equation}\label{eq.upper}\nonumber
G_L(x,y)\lesssim\frac{\omega^x(4B)}{r^{n-1}},\qquad\text{for all }x\in\Omega\backslash2B \text{ and }y\in B\cap\Omega,
\end{equation}
where the implicit constant depends only on $n$, $\lambda$, and the $n$-Ahlfors regularity constant.
\end{lemma}

The following generalization of the formula (\ref{eq.represent}) is well-known; for a detailed proof in our setting, see \cite[Lemma 2.6]{agmt22}.
\begin{lemma}\label{lm.formula} Let $\Omega\subsetneq\bb R^{n+1}$ be an open, connected set with $n$-Ahlfors regular boundary. For any $\Phi\in C_c^{\infty}(\bb R^{n+1})$, we have that
\begin{equation}\nonumber
\Phi(x)=\int_{\partial\Omega}\Phi\,d\omega^x~+~\int_\Omega A^T(y)\nabla_yG_L(x,y)\nabla\Phi(y)\,dm(y),\qquad\text{for }m\text{-a.e.\ }x\in\Omega.
\end{equation}
\end{lemma}

A boundary version of the H\"older continuity of solutions will prove useful; see \cite[Chapter 6]{hkm93}.

\begin{lemma}\label{lm.boundaryholder} Let $\Omega\subsetneq\bb R^{n+1}$ be open with $n$-Ahlfors regular boundary, and let $\xi\in\partial\Omega$, $r\in(0,\diam\partial\Omega)$. Suppose that $u$ is a non-negative solution of $Lu=0$ in $\Omega\cap B(\xi,2r)$, which vanishes continuously on $B(\xi,2r)\cap\partial\Omega$. Then there exist $\eta\in(0,1)$ and $C\geq1$, depending only on $n$, $\lambda$, and the $n$-Ahlfors regularity constant, so that
\begin{equation}\nonumber
u(x)\leq C\Big(\frac{\delta_\Omega(x)}{r}\Big)^\eta\frac1{|B(\xi,2r)|}\int_{B(\xi,2r)\cap\Omega}u\,dm,\qquad\text{for all }x\in\Omega\cap B(\xi,r).
\end{equation}
\end{lemma}

Now, we record a well-known equivalence between the solvability of $(\Di_{p'}^L)$ and a weak-$RH_p$ property of the Poisson kernel. Under our lax geometric assumptions, the proposition is essentially shown in \cite[Theorem 9.2]{mt22} for the case of the Laplacian. The direction (b)$\implies$(a) in the  proposition below is already shown in \cite{hl2018}; for completeness, we give a full proof of the   direction (a)$\implies$(b), but we defer it to Appendix \ref{sec.proofwrhp}. 

\begin{proposition}\label{prop.wrhp} Let $\Omega\subsetneq\bb R^{n+1}$ be open with $n$-Ahlfors regular boundary and satisfying the corkscrew condition. Fix $p\in(1,\infty)$, $\frac1p+\frac1{p'}=1$. The following are equivalent.
\begin{enumerate}[(a)]
	\item $(\Di_{p'}^L)$ is solvable in $\Omega$.
	\item The $L$-elliptic measure $\omega_L$ is absolutely continuous with respect to   $\sigma$, and   for every ball $B$ centered at $\partial\Omega$ satisfying that $\diam B\leq2\diam(\partial\Omega)$, 
	\begin{equation}\label{eq.wrhp}
		\Big(\dashint_B\Big|\frac{d\omega_L^x}{d\sigma}\Big|^p\,d\sigma\Big)^{\frac1p}\lesssim\frac{\omega_L^x(8B)}{\sigma(B)},\qquad\text{for all }x\in\Omega\backslash2B.
	\end{equation}
	\item The $L$-elliptic measure $\omega_L$ is absolutely continuous with respect to $\sigma$ and there is some $\Lambda>1$ big enough such that, for every ball $B$ centered in $\pom$ with $\diam(B)\leq 2\diam(\pom)$
	and all $x \in \Lambda B$ such that $\dist(x,\pom)\geq \Lambda^{-1}r(B)$, it holds that
	$$
	\Big(\dashint_{\Lambda B} \Big|\frac{d\omega^{x}_L}{d\sigma}\Big|^p\,d\sigma\Big)^{1/p} \lesssim_\Lambda\frac1{\sigma(B)}.
	$$
\end{enumerate}
\end{proposition}

In Section \ref{sec.reg}, we will make use of the following powerful localization result for the regularity problem, shown in the proof of \cite[Theorem 5.24]{kp}. Define
\begin{equation}\label{eq.truncated}
	\widetilde{\m N}_{\alpha,\hat c,r}^s(u)(\xi):=\sup_{x\in\gamma_\alpha(\xi)\cap B(\xi,s)}\Big(\dashint_{B(x,\hat c\delta(x))}|u(y)|^r\,dm(y)\Big)^{1/r},\qquad\xi\in\partial\Omega.
\end{equation} 

\begin{theorem}[{\cite[Theorem 1.8.13]{kbook}}]\label{thm.loc} Let $\Omega$ be a bounded star-like Lipschitz domain in $\bb R^{n+1}$, $n\geq2$, and $L=-\dv A\nabla$, where $A$ is a matrix function in $\Omega$ verifying (\ref{eq.elliptic}). Suppose that $(\Reg_2^L)$ is solvable in $\Omega$, that $f\in\dt{W}^{1,p}(\partial\Omega)$ and that $u$ is the solution to the Dirichlet problem $Lu=0$ in $\Omega$, $u=f$ on $\partial\Omega$. Let $s\in(0,\diam(\partial\Omega)/4)$, $\xi\in\partial\Omega$. Then,
\begin{equation}\label{eq.loc}\nonumber
\dashint_{B(\xi,s/2)}|\widetilde{\m N}_2^{s/2}(\nabla u)|^2\,d\sigma\leq C\Big[~\dashint_{B(\xi,2s)}|\nabla_tf|^2\,d\sigma+\Big(s^{-(n+1)}\int_{B(\xi,2s)\cap\Omega\backslash B(\xi,s)}|\nabla u|\,dm\Big)^2\Big].
\end{equation}
\end{theorem}

\section{The Poisson-Dirichlet problem}\label{sec.poisson}

In this section we prove Theorems \ref{thm.poisson} and \ref{thm.poisson2}. Recall that $\delta:=\dist(\cdot,\partial\Omega)$, $2^*=\frac{2(n+1)}{n-1}$ and that $2_*=(2^*)'=\frac{2(n+1)}{n+3}$.

\subsection{Proof of Theorem \ref{thm.poisson}}\label{sec.poisson1} Retain the setting of Theorem \ref{thm.poisson}.  Set
\begin{equation}\nonumber
w(x):=\int_{\partial\Omega}g(\xi)\,d\omega^x(\xi)+\int_\Omega\nabla_2 G(x,y)F(y)\,dm(y)+\int_{\Omega}G(x,y)H(y)\,dm(y)= w_1(x)+w_2(x)+w_3(x),
\end{equation}
where $G=G_{L}$ is the Green's function for $L$, $\omega^x$ is the $L$-elliptic measure  on $\partial\Omega$ with pole at $x$, and $\nabla_2$ denotes the gradient in the second variable. Under our qualitative assumptions on the data, it is known (see Section \ref{sec.pdeprelim}, (\ref{eq.represent}), and (\ref{eq.transpose})) that $Lw=H-\dv F$ holds in the weak sense in $\Omega$, and that $w$ is continuous on $\overline{\Omega}$, so that $w=g$ on $\partial\Omega$. It remains to show that (\ref{eq.ntmaxn}) holds. Since $(\Di_{p'}^L)$ is solvable, we have that
\begin{multline}\label{eq.break}
\Vert\widetilde{\m N}_{2^*}(w)\Vert_{L^{p'}(\partial\Omega)}\leq  \Vert\widetilde{\m N}_{2^*}(w_1)\Vert_{L^{p'}(\partial\Omega)}+\Vert \widetilde{\m N}_{2^*}(w_2)\Vert_{L^{p'}(\partial\Omega)}+\Vert \widetilde{\m N}_{2^*}(w_3)\Vert_{L^{p'}(\partial\Omega)}\\ \leq C\Vert g\Vert_{L^{p'}(\partial\Omega)}+\Vert\widetilde{\m N}_{2^*}(w_2)\Vert_{L^{p'}(\partial\Omega)}+\Vert \widetilde{\m N}_{2^*}(w_3)\Vert_{L^{p'}(\partial\Omega)}.
\end{multline} 

\subsubsection{Step 1: A decomposition of $\Omega$}
We will now focus our attention on estimating the second term on the right-hand side of (\ref{eq.break}). Fix an aperture $\alpha>0$ and $\xi\in\partial\Omega$, and for each $x\in\gamma_{\alpha}(\xi)$ denote $B^x:=B(\xi,\tau\delta(x))\cap\Omega$ and $B_x:=B(x,\delta(x)/8)$, where   $\tau\in(0,2^{-10})$ is a small constant to be determined later. We think of $B^x$ as a ``shrunk'' Carleson region associated to $x$, and of $B_x$ as a small Whitney region associated to $x$. With this notation, we claim that
\begin{multline}\label{eq.split}
|\widetilde{\m N}_{\alpha,2^{-5},2^*}(w_2)(\xi)|\leq\sup_{x\in\gamma_{\alpha}(\xi)}\Big\{\Big(\dashint_{B(x,\frac{\delta(x)}{32})}\Big|\int_{B_x}\nabla_2 G(z,y)F(y)\,dm(y)\Big|^{2^*}\,dm(z)\Big)^{\frac1{2^*}}\\ \qquad+~\sup_{z\in B(x,\delta(x)/32)}\Big[\Big|\int_{B^x}\nabla_2 G(z,y)F(y)\,dm(y)\Big|~+~\Big|\int_{\Omega\backslash(B^x\cup B_x)}\nabla_2G(z,y)F(y)\,dm(y)\Big|\Big]\Big\}\\ =:\sup_{x\in\gamma_{\alpha}(\xi)}\big\{T_1(x)+T_2(x)+T_3(x)\big\}.
\end{multline}
To see (\ref{eq.split}),  simply note that, 
\begin{equation}\nonumber
\begin{array}{lll}
|\widetilde{\m N}_{\alpha,2^{-5},2^*}(w_2)(\xi)|&=&\sup_{x\in\gamma_{\alpha}(\xi)}\Big(\dashint_{B(x,\frac{\delta(x)}{32})}\Big|\int_{\Omega}\nabla_2 G(z,y)F(y)\,dm(y)\Big|^{2^*}\,dm(z)\Big)^{\frac1{2^*}}\\ & \leq& \sup_{x\in\gamma_{\alpha}(\xi)}\Big(\dashint_{B(x,\frac{\delta(x)}{32})}\Big|\int_{B_x}\nabla_2 G(z,y)F(y)\,dm(y)\Big|^{2^*}\,dm(z)\Big)^{\frac1{2^*}}\\& &+\sup_{x\in\gamma_{\alpha}(\xi)}\Big(\dashint_{B(x,\frac{\delta(x)}{32})}\Big|\int_{B^x}\nabla_2 G(z,y)F(y)\,dm(y)\Big|^{2^*}\,dm(z)\Big)^{\frac1{2^*}}\\&&+\underset{x\in\gamma_{\alpha}(\xi)}{\sup}\Big(\dashint_{B(x,\frac{\delta(x)}{32})}\Big|\int_{\Omega\backslash(B^x\cup B_x)}\nabla_2 G(z,y)F(y)\,dm(y)\Big|^{2^*}\,dm(z)\Big)^{\frac1{2^*}}\\&\leq&\sup_{x\in\gamma_{\alpha}(\xi)}\Big(\dashint_{B(x,\frac{\delta(x)}{32})}\Big|\int_{B_x}\nabla_2 G(z,y)F(y)\,dm(y)\Big|^{2^*}\,dm(z)\Big)^{\frac1{2^*}}\\&&+\sup_{x\in\gamma_{\alpha}(\xi)}~\sup_{z\in B(x,\delta(x)/{32})}\Big|\int_{B^x}\nabla_2 G(z,y)F(y)\,dm(y)\Big|\\&&+\sup_{x\in\gamma_{\alpha}(\xi)}~\sup_{z\in B(x,\delta(x)/32)}\Big|\int_{\Omega\backslash(B^x\cup B_x)}\nabla_2 G(z,y)F(y)\,dm(y)\Big|
\end{array}
\end{equation}
  
\subsubsection{Step 2: Estimate for $T_1$}   Note that
\[
T_1(x)\leq C\delta(x)^{-(n+1)/2^*}\Big(\int_{\Omega}|w_{21}(z)|^{2^*}\,dm(z)\Big)^{\frac1{2^*}},
\]
where $w_{21}:= L^{-1}_\Omega\dv(F{\1}_{B_x})$ (see (\ref{eq.op})), and $C$ depends only on $n$.  By the mapping properties of $L^{-1}_\Omega$ (see Definition \ref{def.green} \ref{item.map}), it follows that
\begin{multline}\nonumber
T_1(x)\lesssim\delta(x)^{-(n+1)/2^*}\Big(\int_{\Omega}|F|^2{\1}_{B_x}\,dm\Big)^{1/2}\lesssim C\delta(x)\Big(\dashint_{B(x,\delta(x)/8)}|F|^2\,dm\Big)^{1/2}\\ \lesssim\frac1{\delta(x)^n}\int_{B(x,\delta(x)/8)}\Big(\dashint_{B(x,\delta(x)/8)}|F|^2\,dm\Big)^{1/2}\,dm(z)\\ \lesssim \frac1{\delta(x)^n}\int_{B(\xi,C\delta(x))\cap\Omega}\Big(\dashint_{B(z,\delta(z)/2)}|F|^2\,dm\Big)^{1/2}\,dm(z)\lesssim_{\alpha} \n C_{\frac12,2}(F)(\xi),
\end{multline}
where we used  the fact that $B_x\subset B(y,\delta(y)/2)$ for each $y\in B_x$, and the definition of $\n C_{\frac12,2}(F)$. Hence
\begin{equation}\label{eq.i1done}
\Vert\sup_{x\in\gamma_{\alpha'}(\xi)} T_1(x)\Vert_{L^{p'}(\partial\Omega,d\sigma(\xi))}\lesssim\Vert\n C_2(F)\Vert_{L^{p'}(\partial\Omega)}.
\end{equation}

\subsubsection{Step 3: Estimate for $T_2$}   Let $\m W$ be a $(1/2)$-Whitney decomposition of $\Omega$ (see Section \ref{sec.whitney}), let $\m W_x$ consist of $I\in\m W$ such that $I\cap B^x\neq\varnothing$, and let $\m F$ be the collection of $Q\in\m D_\sigma$ for which there exists $I\in\m W_x$ with $Q\in b(I)$.  Then,
\begin{align}\label{eq.i2first}
T_2(x)&\leq\sup_{z\in B(x,\delta(x)/32)}\sum_{I\in\m W_x}\int_I|\nabla_2G(z,y)||F(y)|\,dm(y)\\ \nonumber &\leq\sup_{z\in B(x,\delta(x)/32)}\sum_{I\in\m W_x}\Big(\int_I|\nabla_2G(z,y)|^2 \,dm(y)\Big)^{1/2}\Big(\int_I|F(y)|^2\,dm(y)\Big)^{1/2}\\ \nonumber &\lesssim\sup_{z\in B(x,\delta(x)/32)}\sum_{I\in\m W_x}\Big(\int_{I^*}G(z,y)^2 \,dm(y)\Big)^{1/2}\frac1{\ell(I)}\ell(I)^{\frac{n+1}2}m_{2,I}(F)\\ \nonumber &\lesssim\sum_{Q\in\m F}~\sum_{I\in\m W_x:Q\in b(I)}\Big(\int_{(1+2\theta)I}G(x,y)^2 \,dm(y)\Big)^{1/2}\ell(I)^{\frac{n-1}2}m_{2,I}(F),
\end{align} 
where   we have used the H\"older inequality, (\ref{eq.transpose}), Lemma \ref{lm.cacc}, and Lemma \ref{lm.moser}. Note that if   $Q\in\m F$, then $Q\subset2^6B^x\cap\partial\Omega$, and $\ell(Q)\leq(32\sqrt{n+1})^{-1}\tau\delta(x)$; moreover, if $I\cap w(Q)\neq\varnothing$, then $(1+2\theta)I\subset 2^{10}B_Q$. Using these facts, we may guarantee that $x\in\Omega\backslash2^{11}B_Q$ if $\tau$ is chosen smaller than $2^{-6}$. Hence, for any $I\in\m W_x$ and any $Q\in\m D_{\sigma}$ which satisfies $Q\in b(I)$, we apply Lemma \ref{lm.upper} with $B=2^{10}B_Q$ to see that
\begin{equation}\label{eq.useupper}
G(x,y)\lesssim\frac{\omega^x(2^{12}B_Q)}{\ell(Q)^{n-1}},\qquad\text{for each }y\in(1+2\theta)I.
\end{equation}
We thus obtain that
\begin{multline}\label{eq.i2}
T_2(x)\lesssim\sum_{Q\in\m F}\frac{\omega^x(2^{12}B_Q)}{\ell(Q)^n}\sum_{I:Q\in b(I)}\int_Im_{2,I}(F)\,dm(z) \lesssim\sum_{Q\in\m F}\frac{\omega^x(2^{12}B_Q)}{\ell(Q)^n}\int_{w(Q)}m_{2,B(z,\delta(z)/2)}(F)\,dm(z)\\ \leq \Big(\sum_{Q\in\m F}\Big\{\frac{\omega^x(2^{12}B_Q)}{\ell(Q)^n}\Big\}^{p_1}\frac1{\inf_{\zeta\in Q}(\n C_2(F)(\zeta))}\int_{w(Q)}m_{2,B(z,\delta(z)/2)}(F)\,dm(z)\Big)^{\frac1{p_1}}\\ \times\Big(\sum_{Q\in\m F}(\inf_{\zeta\in Q}(\n C_2(F)(\zeta)))^{p_1'-1}\int_{w(Q)}m_{2,B(z,\delta(z)/2)}(F)\,dm(z)\Big)^{\frac1{p_1'}},
\end{multline}
where $p_1>p$ will be determined later. Next we let
\begin{equation}\label{eq.aq}\nonumber
a_Q:=\frac1{\inf_{\zeta\in Q}\n C_2(F)(\zeta)}\int_{w(Q)}m_{2,B(z,\delta(z)/2)}(F)\,dm(z),
\end{equation}
and remark that $\{a_Q\}_{Q\in\m F}$ satisfies a Carleson packing condition: For any $S\in\m D_\sigma$,  
\begin{multline}\label{eq.packing}
\sum_{Q\in\m F:Q\subset S}a_Q\leq\frac1{\inf_{\zeta\in S}\n C_2(F)(\zeta)}\sum_{Q\in\m F:Q\subset S}\int_{w(Q)}m_{2,B(z,\delta(z)/2)}(F)\,dm(z)\\ \lesssim\frac1{\inf_{\zeta\in S}\n C_2(F)(\zeta)}\int_{CB_S}m_{2,B(z,\delta(z)/2)}(F)\,dm(z)\lesssim\sigma(S),
\end{multline}
where $B_S$ is a ball with center in $S$ and radius $\ell(S)$, and we have used that there is uniformly bounded overlap of the $w(Q)$'s when summing over all dyadic cubes $Q\subset S$,  $C$ is chosen large enough to contain $\cup_{Q\subset S}w(Q)$, and in the last inequality we used the definition of $\n C_2(F)$. By the dyadic Carleson's embedding theorem (see \cite[Theorem 5.8]{tolsa14}), we deduce that
\begin{multline}\label{eq.i21}
\sum_{Q\in\m F}a_Q\Big\{\frac{\omega^x(2^{12}B_Q)}{\ell(Q)^n}\Big\}^{p_1}\lesssim\int_{B(\xi,2^6\tau\delta(x))}~\sup_{Q\in\m F: \zeta\in Q}\Big\{\frac{\omega^x(2^{12}B_Q)}{\ell(Q)^n}\Big\}^{p_1}\,d\sigma(\zeta)\\ \lesssim\int_{B(\xi,2^6\tau\delta(x))}\Big(\m M_\sigma\Big(\frac{d\omega^x}{d\sigma}{\1}_{B(\xi,2^8\tau\delta(x))\cap\partial\Omega}\Big)(\zeta)\Big)^{p_1}\,d\sigma(\zeta).
\end{multline}
On the other hand,
\begin{multline}\label{eq.i22}
\sum_{Q\in\m F}(\inf_{\zeta\in Q}(\n C_2(F)(\zeta)))^{p_1'-1}\int_{w(Q)}m_{2,B(z,\delta(z)/2)}(F)\,dm(z) \leq\sum_{Q\in\m F}\Big(\dashint_Q\n C_2(F)(\zeta)\,d\sigma(\zeta)\Big)^{p_1'}a_Q \\ \lesssim \int_{B(\xi,2^6\tau\delta(x))}\Big(\sup_{Q\in\m F:\zeta\in Q}\dashint_Q\n C_2(F)(\wt\zeta)\,d\sigma(\wt\zeta)\Big)^{p_1'}\,d\sigma(\zeta)\\ \leq\int_{B(\xi,2^6\tau\delta(x))}\big(\m M_\sigma\big(\n C_2(F){\1}_{B(\xi,2^6\tau\delta(x))\cap\partial\Omega}\big)\big)^{p_1'}\,d\sigma(\zeta),
\end{multline}
where we   used Carleson's embedding theorem. From (\ref{eq.i2}), (\ref{eq.i21}), and (\ref{eq.i22}), it follows that
\begin{multline}\label{eq.i23}
T_2(x)\lesssim\Big(\int_{B(\xi,2^8\tau\delta(x))}\Big|\frac{d\omega^x}{d\sigma}(\zeta)\Big|^{p_1}\,d\sigma(\zeta)\Big)^{\frac1{p_1}}\Big(\int_{B(\xi,2^6\tau\delta(x))}|\n C_2(F)(\zeta)|^{p_1'}\,d\sigma(\zeta)\Big)^{\frac1{p_1'}} \\ \lesssim\Big(\dashint_{B(\xi,2^6\tau\delta(x))}|\n C_2(F)(\zeta)|^{p_1'}\,d\sigma(\zeta)\Big)^{\frac1{p_1'}},
\end{multline}
where in the last estimate we used Proposition \ref{prop.wrhp} with $p_1>p$ sufficiently close to $p$. Finally, we have that
\begin{multline}\label{eq.i2done}
\Vert\sup_{x\in\gamma_{\alpha}(\xi)}T_2(x)\Vert_{L^{p'}(\partial\Omega,d\sigma(\xi))}\lesssim\Big(\int_{\partial\Omega}\sup_{x\in\gamma(\xi)}\Big(\dashint_{B(\xi,2^6\tau\delta(x))}\n C_2(F)^{p_1'}\,d\sigma\Big)^{\frac{p'}{p_1'}}d\sigma(\xi)\Big)^{\frac1{p'}}\\ \leq\Big(\int_{\partial\Omega}\big(\m M_\sigma(\n C_2(F)^{p_1'})(\xi)\big)^{p'/p_1'}\,d\sigma(\xi)\Big)^{\frac{p_1'}{p'}\frac1{p_1'}}  \lesssim \Vert\n C_2(F)\Vert_{L^{p'}(\partial\Omega)}.
\end{multline}

\subsubsection{Step 4: Estimate for $T_3$} It remains to estimate $T_3$. To this end, we split it further as follows:
\begin{multline}\label{eq.split2}
T_3(x)\leq\sup_{z\in B(x,\delta(x)/32)}\int_{\gamma_\beta(\xi)\backslash(B^x\cup B_x)}|\nabla_2G(z,y)F(y)|\,dm(y)\\+\sup_{z\in B(x,\delta(x)/32)}\Big|\int_{\Omega\backslash(\gamma_{\beta}(\xi)\cup B^x)}\nabla_2G(z,y)F(y)\,dm(y)\Big| =:T_{31}(x)+T_{32}(x),
\end{multline}
where $\beta\geq\alpha$ will be determined soon. We control $T_{32}$ first. For each integer $m\geq-3$, let $\m F^m$ be the family of cubes $Q\in\m D_\sigma$ verifying $2^m\tau\delta(x)\leq\dist(Q,\xi)<2^{m+1}\tau\delta(x)$ and $\ell(Q)\leq \wt c2^m\tau\delta(x)$ for a small $\wt c\in(0,1)$ to be determined later. We claim that, for large enough $\beta$,
\begin{equation}\label{eq.inc}
\Omega\backslash(\gamma_\beta(\xi)\cup B^x)\subset\medcup_{m=-3}^{\infty}\bigcup_{Q\in\m F^m}w(Q).
\end{equation}
To see this, let $y\in\Omega\backslash(\gamma_\beta(\xi)\cup B^x)$, and let $I\in\m W$ be such that $y\in I$; we show that if $Q\in b(I)$, then $Q\in\m F^m$ for some $m\geq-3$. For any $Q\in b(I)$, it is easy to see that
\[
\dist(I,Q)+\diam Q\leq 5\delta(y)< 5(1+\beta)^{-1}|y-\xi|,
\]
On the other hand, we have that $\dist(I,\xi)\geq\frac34|y-\xi|$. By fixing $\beta\geq\alpha$ large enough that $(1+\beta)^{-1}\leq\frac1{10}$, we thus guarantee that
\[
\dist(Q,\xi)\geq\dist(I,\xi)-(\dist(I,Q)+\diam Q)\geq|y-\xi|/4\geq\tau\delta(x)/4\geq2^{-3}\tau\delta(x).
\]
Now let $m=m_Q\geq-3$ be the unique integer such that $2^m\tau\delta(x)\leq\dist(Q,\xi)<2^{m+1}\tau\delta(x)$, and suppose that $\ell(Q)\geq\wt c2^m\tau\delta(x)$. Then
\begin{equation}\label{eq.contradict}\nonumber
(1+\beta)\ell(Q)\leq\frac1{2^5\sqrt{n+1}}|y-\xi| \leq\frac1{2^5\sqrt{n+1}}\big\{\diam I+\dist(Q,I)+\diam Q+\dist(Q,\xi)\big\} \leq\frac{C}{\wt c}\ell(Q),
\end{equation}
which gives a contradiction if $\beta$ is chosen large enough (depending on $\wt c$). Thus we have that $\ell(Q)<\wt c2^m\tau\delta(x)$, so that $Q\in\m F^m$, finishing the proof of (\ref{eq.inc}).

With the claim at hand, for each $m\geq-3$ we let $\{Q_j^m\}_j$ be the smallest family of cubes of uniform  generation $k$ with $2^{-k}\leq\wt c2^m\tau\delta(x)<2^{-k+1}$ such that  for each $Q\in\m F^m$, there exists $Q_j^m$ with $Q\subseteq Q_j^m$.  We write $\m F_j^m$ for the family of cubes contained in $Q_j^m$. We now claim that for any $m\geq-3$, $j$, and $Q\in\m F_j^m$, we have that $x\in\Omega\backslash 2^{14}B_Q$. Fix a large integer $m_0\geq-3$ to be determined later. If $m\geq m_0$, the triangle inequality implies that
\[
|x-x_Q|\geq\dist(Q,\xi)-|x-\xi|\geq2^{m-1}\tau\delta(x)-(1+\alpha)\delta(x)\geq\frac{2^{m-1}\tau-2(1+\alpha)}{\wt c2^m\tau}\ell(Q_j^m),
\] 
so that if we choose $m_0$ large enough so that $2^{m_0-2}\geq\frac{2(1+\alpha)}{\tau}$, we may guarantee that  $2^{m-1}\tau-2(1+\alpha)\geq2^{m-2}\tau$ for each $m\geq m_0$. Then, by choosing $\wt c<2^{-16}$, we obtain that $|x-x_Q|>2^{14}\ell(Q)$ for all $m\geq m_0$. If $m\leq m_0$, we have that $|x-x_Q|\geq\delta(x)\geq\frac1{\wt c2^m\tau}\ell(Q)\geq\frac1{\wt c}(2^{m_0}\tau)^{-1}\ell(Q)$, so that if $\wt c<(2^{m_0+14}\tau)^{-1}$, we obtain that $|x-x_Q|\geq2^{14}\ell(Q)$ for all $-3\leq m\leq m_0$.

We have thus shown that for any $m\geq-3$, any $j$, and any $Q\in\m F_j^m$, we have that $x\in\Omega\backslash2^{14}B_Q$.  Proceeding as in (\ref{eq.i2first}) and (\ref{eq.i2}), and using the preceding claim to apply Lemma \ref{lm.upper} and obtain the appropriate analogue of (\ref{eq.useupper}),   we deduce that
\begin{equation}\label{eq.i321}
T_{32}(x)\lesssim\sum_{m\geq-3}\sum_j\sum_{Q\in\m F_j^m}\frac{\omega^x(2^{12}B_Q)}{\ell(Q)^n}\int_{w(Q)}m_{2,B(z,\delta(z)/2)}(F)\,dm(z).
\end{equation}

Arguing as in the proof of (\ref{eq.i23}), for fixed $m$ and $j$ we may show that
\begin{multline}\label{eq.i322} 
\sum_{Q\in\m F_j^m}\frac{\omega^x(2^{12}B_Q)}{\ell(Q)^n}\int_{w(Q)}m_{2,B(z,\delta(z)/2)}(F)\,dm(z) \lesssim \Big(\int_{2^{13}B_{Q_j^m}}\Big(\frac{d\omega^x}{d\sigma}\Big)^{p_1}\,d\sigma\Big)^{\frac1{p_1}}\Big(\int_{Q_j^m}\n C_2(F)^{p_1'}\,d\sigma\Big)^{\frac1{p_1'}} \\ \lesssim\omega^x\big(2^{16}B_{Q_j^m}\big)\Big(\dashint_{B(\xi,2^{m+2}\tau\delta(x))}\n C_2(F)^{p_1'}\,d\sigma\Big)^{\frac1{p_1'}},
\end{multline}
where we used the  $n$-Ahlfors regularity of $\partial\Omega$, the fact that $\ell(Q_j^m)\approx 2^m\tau\delta(x)$, and Proposition \ref{prop.wrhp}. It follows from (\ref{eq.i321}) and (\ref{eq.i322}) that
\begin{equation}\label{eq.i323}
T_{32}(x)\lesssim\Big(\sum_{m\geq-3}\omega^x\big(\medcup_j2^{16}B_{Q_j^m}\big)\Big)\big(\m M_\sigma(\n C_2(F)^{p_1'})(\xi)\big)^{1/p_1'}.
\end{equation}
Let $E:=\medcup_j2^{16}B_{Q_j^m}$. By choosing $\wt c$ small enough, we may guarantee that $\dist(E,\xi)>2^{m-2}\tau\delta(x)$, and therefore $\omega^y(E)\equiv0$ on $B(\xi,2^{m-2}\tau\delta(x))\cap\partial\Omega$. Moreover, for all $m$ large enough such that $2^m\geq8/\tau$, we have that $x\in B(\xi,2^{m-2}\tau\delta(x))$, and for such $m$, since $v(x):=\omega^x(E)$ solves $Lv=0$ in $\Omega$, the boundary H\"older inequality, Lemma \ref{lm.boundaryholder}, implies that $\omega^x(E)\lesssim\Big(\frac{(1+\alpha)\delta(x)}{2^{m-2}\tau\delta(x)}\Big)^\eta\lesssim2^{-m\eta}$, so that 
\[
\sum_{m\geq-3}\omega^x(E)\lesssim\sum_{m\geq-3}2^{-m\eta}\leq C.
\]
Using this result on (\ref{eq.i323}), it follows that
\begin{equation}\label{eq.i32done}
\Vert\sup_{x\in\gamma(\xi)} T_{32}(x)\Vert_{L^{p'}(\partial\Omega,d\sigma(\xi))}\lesssim\Vert\big(\m M_\sigma(\n C_2(F)^{p_1'})\big)^{1/p_1'}\Vert_{L^{p'}(\partial\Omega)}\lesssim\Vert\n C_2(F)\Vert_{L^{p'}(\partial\Omega)}.
\end{equation}

For $T_{31}$, let $\m W_1$ be the collection of Whitney cubes $I\in\m W$ such that $I\cap(\gamma_\beta(\xi)\backslash(B^x\cup B_x))\neq\varnothing$. It is easy to see that $\cup_{I\in\m W_1}I\subset\gamma_{5+4\beta}(\xi)$. For any $I\in\m W_1$ and $y\in I^*$, we distinguish two cases: either $\delta(y)\leq K\delta(x)$ or $\delta(y)>K\delta(x)$, for $K$  large  to be specified. If $\delta(y)\leq K\delta(x)$, then $|y-x|\geq2^{-4}\delta(x)\geq 2^{-4}K^{-1}\delta(y)$. If $\delta(y)> K\delta(x)$, then
\[
|y-x|>|y-\xi|-|x-\xi|\geq\delta(y)-(1+\alpha)\delta(x)>\delta(y)-(1+\alpha)K^{-1}\delta(y),
\]
so if we choose $K$ large enough that $\frac{1+\alpha}K\leq\frac12$, then we have that $|y-x|>\delta(y)/2$. Thus, in any case we have that $|y-x|\gtrsim\delta(y)$, and similarly, $|y-z|\gtrsim\delta(y)$ for all $z\in B(x,\delta(x)/32)$.   With this in mind, we see that
\begin{multline}\label{eq.i31}
T_{31}(x)\leq\sup_{z\in B(x,\delta(x)/32)}\sum_{I\in\m W_1}\Big(\int_I|\nabla_2G(z,y)|^2 \,dm(y)\Big)^{1/2}\Big(\int_I|F(y)|^2\,dm(y)\Big)^{1/2}\\ \lesssim\sup_{z\in B(x,\delta(x)/32)}\sum_{I\in\m W_1}\Big(\int_{I^*}G(z,y)^2 \,dm(y)\Big)^{1/2}\frac1{\ell(I)}\ell(I)^{\frac{n+1}2}m_{2,I}(F)\\ \lesssim\sum_{I\in\m W_1}\ell(I)^{-n}\int_Im_{2,I}(F) \lesssim\int_{\gamma_{5+4\beta}(\xi)}\delta(y)^{-n-1}\{\delta(y)m_{2,B(y,\delta(y)/2)}(F)\}\,dm(y),
\end{multline}
where we have used that  $G(z,y)\lesssim|y-z|^{1-n}\lesssim\delta(y)^{1-n}$ for each $y\in\gamma_\beta(\xi)\backslash(B^x\cup B_x)$ and each $z\in B(x,\delta(x)/32)$. From (\ref{eq.i31}) it follows that
\begin{multline}\label{eq.i31done}
\int_{\partial\Omega}\big|\sup_{x\in\gamma_{\alpha}(\xi)}T_{31}(x)\big|^{p'}\,d\sigma(\xi) \lesssim\int_{\partial\Omega}\Big|\int_{\gamma_{5+4\beta}(\xi)}\delta(y)^{-n-1}\{\delta(y)m_{2,B(y,\delta(y)/2)}(F)\}\,dm(y)\Big|^{p'}\,d\sigma(\xi)\\ \lesssim\int_{\partial\Omega}\wt{\m A}_1(\delta F)(\xi)^{p'}\,d\sigma(\xi) \lesssim\int_{\partial\Omega}\n C_2(F)(\xi)^{p'}\,d\sigma(\xi),
\end{multline}
where in the last inequality we have used (\ref{eq.tentandcar}).

\subsubsection{Conclusion} From the estimates   (\ref{eq.i1done}), (\ref{eq.i2done}), (\ref{eq.i32done}) and (\ref{eq.i31done}), we deduce that
\begin{equation}\label{eq.w2done}
\Vert\widetilde{\m N}_{2^*}(w_2)\Vert_{L^{p'}(\partial\Omega)}\lesssim\Vert\n C_2(F)\Vert_{L^{p'}(\partial\Omega)}.
\end{equation}

It remains only to estimate the last term in the right-hand side of (\ref{eq.break}). However, this term is almost completely analogous to the term we have just bounded in (\ref{eq.w2done}), thus let us only comment on the changes. First, as in (\ref{eq.split}) we break up $|\widetilde{\m N}_{\alpha,2^{-5},2^*}(w_3)(\xi)|$ into three analogous parts $T_1', T_2', T_3'$. The term $T_1'$ is controlled in the same way as $T_1$ was controlled, except now we use the fact that $L^{-1}_\Omega:L^{2_*}(\Omega)\ra L^{2^*}(\Omega)$. The term $T_2'$ is controlled similarly as $T_2$ was controlled, but now we do not need to use the Caccioppoli inequality; instead, for instance  in the analogue of (\ref{eq.i2first}) we use the basic estimate
\begin{equation}\label{eq.change}
\int_IG(z,y)|H(y)|\,dm(y)\lesssim\Big(\dashint_IG(z,y)^{2^*}\,dm(y)\Big)^{\frac1{2^*}}\frac1{\ell(I)^n}\ell(I)^{n+1}m_{2_*,I}(\delta H),
\end{equation}
and mimic the proof of (\ref{eq.i2done}) to finish the estimate for $T_2'$. The term $T_3'$ is split into two terms $T_{31}'$ and $T_{32}'$ analogously as in (\ref{eq.split2}), and we bound each of these separately; the only difference with the detailed proofs above will be the use of (\ref{eq.change}) instead of the Caccioppoli inequality at the very beginning. As such, we control all three terms $T_1'$, $T_2'$ and $T_3'$ in the expected ways, and obtain the estimate
\begin{equation}\label{eq.w3done}
	\Vert\widetilde{\m N}_{2^*}(w_3)\Vert_{L^{p'}(\partial\Omega)}\lesssim\Vert\n C_{2_*}(\delta H)\Vert_{L^{p'}(\partial\Omega)}.
\end{equation}
From (\ref{eq.break}), (\ref{eq.w2done}), and (\ref{eq.w3done}),  the desired bound (\ref{eq.ntmaxn}) follows.\hfill{$\square$}

We now consider an extension of Theorem \ref{thm.poisson} to the case of  $F\in{\bf C}_{2,p'}$ and $H$ with $\delta_\Omega H\in{\bf C}_{2_*,p'}$.

\begin{theorem}\label{thm.poisson3} Let $\Omega\subsetneq\bb R^{n+1}$, $n\geq2$ be a domain satisfying the corkscrew condition and with $n$-Ahlfors regular boundary. Let $p>1$, $p'$ its H\"older conjugate, and $L=-\dv A\nabla$. Assume that $(\Di_{p'}^L)$ is solvable in $\Omega$.   Let $g\in L^{p'}(\partial\Omega)$, $H$ such that  $\delta_\Omega H\in{\bf C}_{2_*,p'}$, and $F\in{\bf C}_{2,p'}$. Then there exists a weak solution $w\in W^{1,2}_{\loc}(\Omega)$ to the problem (\ref{eq.poisson}) satisfying the estimate (\ref{eq.ntmaxn}).  The equation $w=g$ on $\partial\Omega$ is understood as
\begin{equation}\label{eq.gnt}\nonumber
\lim_{\gamma(\xi)\ni x\ra\xi}\dashint_{B(x,\delta_\Omega(x)/2)}w\,dm=g(\xi),\qquad\text{for }\sigma\text{-a.e.\ }\xi\in\partial\Omega.
\end{equation} 
\end{theorem}

\noindent\emph{Proof.}   Since $(\Di_{p'}^L)$ is solvable in $\Omega$, it is well-known that for any $g\in L^{p'}(\partial\Omega)$, there exists a   weak solution $u\in W^{1,2}_{\loc}(\Omega)\cap C_{\loc}^{\eta}(\Omega)$ to the equation $Lu=0$ in $\Omega$ satisfying that $\lim_{\gamma(\xi)\ni x\ra\xi}u(x)=g(\xi)$, for $\sigma\text{-a.e.\ }\xi\in\partial\Omega$. Hence, by linearity, it is enough for us to consider the case $g\equiv0$. By Lemma \ref{lm.lipschitz}, there exist sequences $\{F_k\}$, $\{h_k\}$ of compactly supported Lipschitz functions on $\Omega$ such that $F_k\ra F$ strongly in ${\bf C}_{2,p'}$ and $h_k\ra\delta_\Omega H$ strongly in ${\bf C}_{2_*,p'}$, as $k\ra\infty$. Let $w_k$ be the unique weak solution in $C(\overline\Omega)\cap Y_0^{1,2}(\Omega)$ to the equation $Lw_k=\delta_\Omega^{-1}h_k-\dv F_k$. For any $k,\ell\in\bb N$, we have that
\[
L(w_k-w_\ell)=\delta_{\Omega}^{-1}(h_k-h_\ell)-\dv(F_k-F_\ell),
\]
and $w_k-w_\ell\in Y_0^{1,2}(\Omega)$. Therefore, since we have already shown the bound (\ref{eq.ntmaxn}) for the case of compactly supported Lipschitz data, we see that
\begin{equation}\label{eq.ntc}
\Vert\wt{\m N}_{2^*}(w_k-w_\ell)\Vert_{L^{p'}(\partial\Omega)}\lesssim\Vert\n C_{2_*}(h_k-h_\ell)\Vert_{L^{p'}(\partial\Omega)}+\Vert\n C_{2}(F_k-F_\ell)\Vert_{L^{p'}(\partial\Omega)}\lra0,
\end{equation}
as $k\ra\infty$. This shows that the sequence $\{w_k\}$ is strongly convergent in the Banach space ${\bf N}_{2^*,p'}$, defined in Section \ref{sec.main1}. Let $w\in{\bf N}_{2^*,p'}$ be the limit of $\{w_k\}$ in ${\bf N}_{2^*,p'}$, and furthermore, note that $w_k\ra w$ strongly in $L^{2^*}_{\loc}(\Omega)$ as $k\ra\infty$. Let us remark that since we also have that $F_k\ra F$ strongly in $L^2_{\loc}(\Omega)$ and that $h_k\ra\delta_\Omega H$ strongly in $L^{2_*}_{\loc}(\Omega)$, then by the Caccioppoli inequality (Lemma \ref{lm.cacc}), we see that the sequence $\{\nabla(w_k-w_\ell)\}$ is strongly convergent in $L^2_{\loc}(\Omega)$. Then it is an exercise to check that $w\in W^{1,2}_{\loc}(\Omega)$ with $\nabla w_k\ra\nabla w$ strongly in $L^2_{\loc}(\Omega)$. From this fact, it easily follows that the equation $Lw=H-\dv F$ holds in the weak sense in $\Omega$. 

Finally, we claim that
\[
\lim_{\gamma(\xi)\ni x\ra\xi}\Big(\dashint_{B(x,\delta(x)/2)}|w|^{2^*}\,dm\Big)^{\frac1{2^*}}=0,\qquad\text{for }\sigma\text{-a.e.\ }\xi\in\partial\Omega.
\]
However, this follows easily since  $\wt{\m N}_{2^*}(w-w_{k'})(\xi)\ra0$ pointwise $\sigma$-a.e.\ on $\partial\Omega$  as $k'\ra\infty$ (after passing to a subsequence $\{w_{k'}\}$), and by boundary H\"older continuity applied to the solution $w_{k'}$ in small enough balls centered at the boundary.\hfill{$\square$}

\subsection{Proof of Theorem \ref{thm.poisson2}}\label{sec.proofpoisson} The implication (a)$\implies$(b) is   contained in Theorem \ref{thm.poisson} already, while   (b)$\implies$(c) and (d)$\implies$(e) follow immediately by definition. We will show (c)$\implies$(a), (e)$\implies$(c), (c)$\implies$(d), (e)$\implies$(f), (f)$\implies$(a), (g)$\implies$(a), (a)$\implies$(g), (h)$\implies$(a), (a)$\implies$(h).  

{\bf Proof of }(c)$\implies$(a){\bf.} We follow part of the argument for \cite[Theorem 1.5]{mt22}. Assume that (c) holds, and we  show that (a) holds by proving that the assumptions in Proposition \ref{prop.wrhp} (c) hold. To this end, we   prove that for $B$, $\Lambda$, and $x \in\Omega$ as in Proposition \ref{prop.wrhp} (c),
\begin{equation}\label{eqclau883}
	\left(\avint_{\Lambda B} \big(\cM_{\sigma,0}\, \omega_L^x\big)^p \,d\sigma\right)^{1/p} \lesssim_\Lambda \sigma(B)^{-1},
\end{equation}
where $\cM_{\sigma,0}$ is the truncated maximal operator defined by 
$$\cM_{\sigma,0} \,\tau(\xi) = \sup_{0<r\leq \dist(x,\pom)/4} \frac{|\tau|(B(\xi,r))}{\sigma(B(\xi,r))},$$
for any signed Radon measure $\tau$. 

Given a ball $B_\xi=B(\xi,r)$, with $\xi\in\pom\cap\Lambda B$, $0<r\leq \dist(x,\pom)/4$, (so that $x\not \in 4B_\xi$), consider a smooth non-negative function $\vphi_{B_\xi}$ which equals $1$ on $B_\xi$ and vanishes away from $2B_\xi$. Then  we have that
\begin{multline}\label{eq.compb}
\omega_L^x(B_\xi) \lesssim\dashint_{B(x,\delta(x)/8)}\omega_L^z\,dm(z)\leq\dashint_{B(x,\delta(x)/8)}\int_{\partial\Omega} \vphi_{B_\xi}\,d\omega_L^z\,dm(z)\\ = -\,\dashint_{B(x,\delta(x)/8)}\int_\Omega A^T\nabla_y G_{L}(z,y)\nabla \vphi_{B_\xi}(y)\,dm(y)\,dm(z) \lesssim \frac1r\int_{2B_\xi}\Big|\dashint_{B(x,\delta(x)/8)}\nabla_y G_{L}(z,y)\,dm(z)\Big|\,dm(y)\\ \lesssim\int_{CB_\xi}\wt{\m N}_2^{C'\delta(x)}\Big(\1_{\Omega\backslash B(x,\delta(x)/4)}~\dashint_{B(x,\delta(x)/8)}\nabla_2 G_{L}(z,\cdot)\,dm(z)\Big)\,d\sigma,
\end{multline}
where we used Lemma \ref{lm.moser}, (\ref{eq.represent}),  and Fubini's theorem.  Therefore,
$$\frac{\omega^x_L(B(\xi,r))}{\sigma(B(\xi,r))}\lesssim \avint_{B(\xi,Cr)} \wt{\m N}_2^{C'\delta(x)}\Big(\1_{\Omega\backslash B(x,\delta(x)/4)}~\dashint_{B(x,\delta(x)/8)}\nabla_2 G_{L}(z,\cdot)\,dm(z)\Big)\,d\sigma.$$
Taking the supremum over $0<r\leq \dist(x,\pom)/4$, we derive 
$$\cM_{\sigma,0}\,\omega^x_L(\xi) \lesssim \cM_\sigma\Big(\1_{B(x,C'\delta(x))}\wt{\m N}_2^{C'\delta(x)}\Big(\1_{\Omega\backslash B(x,\delta(x)/4)}~\dashint_{B(x,\delta(x)/8)}\nabla_2 G_{L}(z,\cdot)\,dm(z)\Big)\Big)(\xi),$$
so that
\begin{multline}\label{eq.fkf4}
	\|\cM_{\sigma,0}\, \omega^x_L\|_{L^p(\Lambda B,\sigma)} \lesssim 
	\Big\Vert\cM_\sigma\Big(\1_{B(x,C'\delta(x))}\wt{\m N}_2^{C'\delta(x)}\Big(\1_{\Omega\backslash B(x,\delta(x)/4)}~\dashint_{B(x,\delta(x)/8)}\nabla_2 G_{L}(z,\cdot)\,dm(z)\Big)\Big)\Big\Vert_{L^p(\sigma)} \\ \lesssim
	\Big\Vert\wt{\m N}_2^{C'\delta(x)}\Big(\1_{\Omega\backslash B(x,\delta(x)/4)}~\dashint_{B(x,\delta(x)/8)}\nabla_2 G_{L}(z,\cdot)\,dm(z)\Big)\Big\Vert_{L^p(B(x,C'\delta(x)),\sigma)}.
\end{multline}

Next, we use the $\m N$-$\n C$ duality to control the right-hand side of (\ref{eq.fkf4}). We have that
\begin{multline}\label{eq.casi}
\Big\Vert\wt{\m N}_2\Big(\1_{\Omega\backslash B(x,\delta(x)/4)}~\dashint_{B(x,\delta(x)/8)}\nabla_2 G_{L}(z,\cdot)\,dm(z)\Big)\Big\Vert_{L^p(\sigma)}\\ \lesssim\sup_{F\in{\bf C}_{2,p'}:\Vert\n C_2(F)\Vert_{L^{p'}(\partial\Omega)}=1}\Big|\int_{\Omega}\Big(\1_{\Omega\backslash B(x,\delta(x)/4)}(y)\dashint_{B(x,\delta(x)/8)}\nabla_2 G_{L}(z,y)\,dm(z)\Big)F(y)\,dm(y)\Big|\\ =\sup_{F\in{\bf C}_{2,p'}:\Vert\n C_2(F)\Vert_{L^{p'}(\partial\Omega)}=1}\Big|\dashint_{B(x,\delta(x)/8)}\Big(\int_{\Omega}\nabla_2 G_{L}(z,y)\big\{F\1_{\Omega\backslash B(x,\delta(x)/4)}\big\}\,dm(y)\Big)\,dm(z)\Big|,
\end{multline}
where we used Proposition \ref{prop.duality} and Fubini's theorem. By Lemma \ref{lm.lipschitz}, we may assume without loss of generality that the supremum in the right-hand side of (\ref{eq.casi}) runs over the compactly supported Lipschitz functions $F$ such that $\Vert\n C_2(F)\Vert_{L^{p'}(\partial\Omega)}=1$. Then, let
\[
w_{x,F}(z):=\int_{\Omega}\nabla_2 G_{L}(z,y)\big\{F(y)\1_{\Omega\backslash B(x,\delta(x)/4)}(y)\big\}\,dm(y),\qquad z\in\Omega,
\]
so that $Lw_{x,F}=-\dv F\1_{\Omega\backslash B(x,\delta(x)/4)}$, $w_{x,F}\in Y_0^{1,2}(\Omega)$, and by hypothesis, we have 
\[
\Vert\wt{\m N}_{2^*}(w_{x,F})\Vert_{L^{p'}(\partial\Omega)}\lesssim\Vert\n C_2(F\1_{\Omega\backslash B(x,\delta(x)/4)})\Vert_{L^{p'}(\partial\Omega)}\leq\Vert\n C_2(F)\Vert_{L^{p'}(\partial\Omega)}.
\]
Therefore, if $\zeta\in\partial\Omega$ satisfies that $|x-\zeta|=\delta_\Omega(x)$, then
\begin{multline}\label{eq.casi4}
\Big|\dashint_{B(x,\delta(x)/8)}\Big(\int_{\Omega}\nabla_2 G_{L}(z,y)\big\{F(y)\1_{\Omega\backslash B(x,\delta(x)/4)}(y)\big\}\,dm(y)\Big)\,dm(z)\Big|\\ \lesssim\Big|\dashint_{B(x,\delta(x)/8)}w_{x,F}\,dm\Big|\leq\Big(\dashint_{B(x,\delta(x)/8)}|w_{x,F}|^{2^*}\,dm\Big)^{\frac1{2^*}}\\ \lesssim\dashint_{B(\zeta,C\delta(x))}\wt{\m N}_{2^*}(w_{x,F})\,d\sigma\lesssim\delta(x)^{-n/p'}\Vert\n C_2(F)\Vert_{L^{p'}(\partial\Omega)}\approx\delta(x)^{-n/p'},
\end{multline}
so that from (\ref{eq.fkf4}) and (\ref{eq.casi}) we conclude that
\begin{equation}\label{eq.casi3}
\|\cM_{\sigma,0}\, \omega^x_L\|_{L^p(\Lambda B,\sigma)}\lesssim\delta(x)^{-n/p'}\lesssim_{\Lambda}\sigma(B)^{-1/p'},
\end{equation}
which readily implies the desired bound (\ref{eqclau883}).
 
{\bf Proof of }(e)$\implies$(c){\bf.} Fix $F\in L^{\infty}_c(\Omega)$ and let $w$ be the unique weak solution in $Y_0^{1,2}(\Omega)$ to the equation $Lw=-\dv F$ given by the Green's representation formula (\ref{eq.represent}) (with the Green's function $G_L$). Using Proposition \ref{prop.duality}, we see that
\begin{equation}\label{eq.casi2b}
	\Vert\wt{\m N}_{2^*}(w)\Vert_{L^{p'}(\partial\Omega)}\lesssim\sup_{H\in{\bf C}_{2_*,p}:\Vert\n C_{2_*}(H)\Vert_{L^{p}(\partial\Omega)}=1}\Big|\int_\Omega Hw\,dm\Big|.
\end{equation}
By Lemma \ref{lm.lipschitz}, we may assume without loss of generality that the supremum in the estimate above runs over all compactly supported Lipschitz functions $H$ in $\Omega$ satisfying that $\Vert\n C_{2_*}(H)\Vert_{L^{p}(\partial\Omega)}=1$. Let $v\in Y_0^{1,2}(\Omega)$ be the unique weak solution to the equation $L^*v=H$. Then we have that
\[
\int_\Omega Hw\,dm=\int_\Omega A^T\nabla v\nabla w\,dm=\int_\Omega A\nabla w\nabla v\,dm=\int_\Omega F\nabla v\,dm,
\]
where we used that $v,w\in Y_0^{1,2}(\Omega)$. Using this identity in (\ref{eq.casi2b}) and Proposition \ref{prop.duality} again, we obtain that
\begin{multline}\nonumber
	\Vert\wt{\m N}_{2^*}(w)\Vert_{L^{p'}(\partial\Omega)}\lesssim\sup_{H:\Vert\n C_{2_*}(H)\Vert_{L^{p}(\partial\Omega)}=1}\Vert\n C_{2}(F)\Vert_{L^{p'}(\partial\Omega)}\Vert\wt{\m N}_{2}(\nabla v)\Vert_{L^{p}(\partial\Omega)}\\ \lesssim\sup_{H:\Vert\n C_{2_*}(H)\Vert_{L^{p}(\partial\Omega)}=1}\Vert\n C_{2_*}(H)\Vert_{L^p(\partial\Omega)}\Vert\n C_2(F)\Vert_{L^{p'}(\partial\Omega)}=\Vert\n C_{2}(F)\Vert_{L^{p'}(\partial\Omega)},
\end{multline}
where we used that the estimate (\ref{eq.estreg}) holds (with $F\equiv0$).
 
{\bf Proof of }(c)$\implies$(d){\bf.} First, let us note that one may prove (c)$\implies$(e) in a perfectly analogous manner to the proof of (e)$\implies$(c)   shown just above, and for all other applications in this paper, only the equivalence (c)$\iff$(e) is needed. Nevertheless, we prefer to show the stronger result (\ref{eq.estreg}) with non-trivial $F$, and that is what we account for in this argument.  Fix $H,F\in L^{\infty}_c(\Omega)$ and let $v$ be the unique weak solution in $Y_0^{1,2}(\Omega)$ to the equation $L^*v=H-\dv F$. Using Proposition \ref{prop.duality}, we see that
\begin{equation}\label{eq.casi2}
	\Vert\wt{\m N}_2(\nabla v)\Vert_{L^p(\partial\Omega)}\lesssim\sup_{G\in{\bf C}_{2,p'}:\Vert\n C_2(G)\Vert_{L^{p'}(\partial\Omega)}=1}\Big|\int_\Omega G\nabla v\,dm\Big|.
\end{equation}
By Lemma \ref{lm.lipschitz}, we may assume without loss of generality that the supremum in the estimate above runs over all compactly supported Lipschitz functions $G$ in $\Omega$ satisfying that $\Vert\n C_2(G)\Vert_{L^{p'}(\partial\Omega)}=1$. Let $w\in Y_0^{1,2}(\Omega)$ be the unique weak solution to the equation $Lw=-\dv G$. Then we have that
\begin{equation}\label{eq.cd2}
\int_\Omega G\nabla v\,dm=\int_\Omega A\nabla w\nabla v\,dm=\int_\Omega A^T\nabla v\nabla w\,dm=\int_\Omega Hw\,dm+\int_\Omega F\nabla w\,dm,
\end{equation}
where we used that $v,w\in Y_0^{1,2}(\Omega)$. Next, using Proposition \ref{prop.duality} again, we see that
\begin{equation}\label{eq.cd3}
\Big|\int_\Omega Hw\,dm\Big|\lesssim\Vert\n C_{2_*}(H)\Vert_{L^p(\partial\Omega)}\Vert\wt{\m N}_{2^*}(w)\Vert_{L^{p'}(\partial\Omega)}\lesssim\Vert\n C_{2_*}(H)\Vert_{L^p(\partial\Omega)}\Vert\n C_2(G)\Vert_{L^{p'}(\partial\Omega)}=\Vert\n C_{2_*}(H)\Vert_{L^p(\partial\Omega)},
\end{equation}
where we used that $(\operatorname{PD}_{p'}^L)$ is solvable in $\Omega$. Furthermore,
\begin{equation}\label{eq.cd4}
\Big|\int_\Omega F\nabla w\,dm\Big|\lesssim\Vert\n C_2(F/\delta)\Vert_{L^p(\partial\Omega)}\Vert\wt{\m N}_2(\delta\nabla w)\Vert_{L^{p'}(\partial\Omega)}.
\end{equation}
Assume for the time being the following estimate:
\begin{equation}\label{eq.cd5}
\wt{\m N}_{\alpha,\frac18,2}(\delta\nabla w)(\xi)\lesssim\wt{\m N}_{\alpha_1,\frac12,2}(w)(\xi)+\n C_{\frac12,2}(G)(\xi),\qquad\sigma-\text{a.e. }\xi\in\partial\Omega,
\end{equation}
for some $\alpha_1>\alpha$. If (\ref{eq.cd5}) holds, then we have that
\begin{equation}\label{eq.cd6}
\Vert\wt{\m N}_2(\delta\nabla w)\Vert_{L^{p'}(\partial\Omega)}\lesssim\Vert\wt{\m N}_2(w)\Vert_{L^{p'}(\partial\Omega)}+\Vert\n C_2(G)\Vert_{L^{p'}(\partial\Omega)}\lesssim\Vert\n C_2(G)\Vert_{L^{p'}(\partial\Omega)}=1.
\end{equation}
From (\ref{eq.casi2})-(\ref{eq.cd6}), the desired estimate (\ref{eq.estreg}) follows. It remains to show that (\ref{eq.cd5}) holds. Fix $\xi\in\partial\Omega$, $x\in\gamma_\alpha(\xi)$, and write $B=B(x,\delta(x)/8)$, $B^*=1.1B$. By the Caccioppoli inequality (Lemma \ref{lm.cacc}), 
\[
\Big(\dashint_B|\delta\nabla w|^2\,dm\Big)^{\frac12}\lesssim\Big(\dashint_{B^*}|w|^2\,dm\Big)^{\frac12}+\delta(x)\Big(\dashint_{B^*}|G|^2\,dm\Big)^{\frac12}\lesssim\wt{\m N}_{\alpha_1,\frac12,2}(w)(\xi)+\delta(x)\Big(\dashint_{B^*}|G|^2\,dm\Big)^{\frac12},
\]
and note that
\begin{multline*}
\delta(x)\Big(\dashint_{B^*}|G|^2\,dm\Big)^{\frac12}\lesssim\delta(x)^{-n}\int_{B(x,\delta(x)/8)}\Big(\dashint_{B(y,\delta(y)/2)}|G|^2\,dm\Big)^{\frac12}\,dm(y)\\ \leq\frac1{\delta(x)^n}\int_{B(\xi,(2+\alpha)\delta(x))\cap\Omega}\Big(\dashint_{B(y,\delta(y)/2)}|G|^2\,dm\Big)^{\frac12}\,dm(y)\lesssim\n C_{\frac12,2}(G)(\xi),
\end{multline*}
which completes the proof of (\ref{eq.cd5}).

{\bf Proof of }(e)$\implies$(f){\bf.} This follows from definition, using (\ref{eq.computew}).

{\bf Proof of }(f)$\implies$(a){\bf.} We show that (a) holds by proving that the assumptions in Proposition \ref{prop.wrhp} (c) hold. Thus we endeavor to prove (\ref{eqclau883}), and to do this, we mimic the argument of (c)$\implies$(a) written above with some small changes. Given $B$, $\Lambda$, and $x\in\Omega$ as in Proposition \ref{prop.wrhp} (c), $\xi\in\partial\Omega\cap\Lambda B$, $r\in(0,\dist(x,\partial\Omega)/4)$, then similarly to (\ref{eq.compb}) we obtain the estimate
\begin{equation}\label{eq.compd}\nonumber
\omega_L^x(B_\xi) \lesssim\dashint_{I}\omega_L^z\,dm(z) \lesssim\int_{CB_\xi}\wt{\m N}_2^{C'\ell(I)}\Big(\nabla_2\dashint_{I} G_{L}(z,\cdot)\,dm(z)\Big)\,d\sigma \lesssim\ell(I)^{-(n+1)}\int_{CB_\xi}\wt{\m N}_2^{C'\ell(I)}(\nabla u_I)\,d\sigma
\end{equation}
where $I$ is a Whitney cube such that $x\in\overline I$ and $u_I$ is the local landscape function (for the operator $L^*$) subordinate to $I$. Proceeding as in the proof of (c)$\implies$(a), we use that (f) holds to deduce that
\[
\|\cM_{\sigma,0}\, \omega^x_L\|_{L^p(\Lambda B,\sigma)}\lesssim\ell(I)^{-(n+1)}\Vert\wt{\m N}_2^{C'\ell(I)}(\nabla u_I)\Vert_{L^p(C'B_Q,\sigma)}\lesssim\ell(I)^{-(n+1)}\ell(I)^{1+\frac np}\approx\ell(I)^{-n/p'}\approx\delta(x)^{-n/p'}.
\] 
The desired result follows.

{\bf Proof of }(g)$\implies$(a){\bf.} This is very similar to the proof of (f)$\implies$(a); we omit further details.

{\bf Proof of }(a)$\implies$(g){\bf.} We use Proposition \ref{prop.duality} to see that
\begin{equation}\nonumber
\Vert\wt{\m N}_q\big(\nabla_2G_L(x,\cdot)\big)\Vert_{L^p(\sigma)}\lesssim\sup_{F\in{\bf C}_{q',p'}:\Vert\n C_{q'}(F)\Vert_{L^{p'}(\partial\Omega)}=1}\Big|\int_{\Omega}\nabla_y G_{L}(x,y)F(y)\,dm(y)\Big|.
\end{equation}
Let $w=L^{-1}_\Omega\dv F$. Then, we have that
\begin{equation}\nonumber
\Big|\int_{\Omega}\nabla_y G_{L}(x,y)F(y)\,dm(y)\Big|=|w(x)|\lesssim\dashint_{B(x,C'\delta(x))}\m N_{\alpha}(w)\,d\sigma\lesssim\delta(x)^{-n/p'}\Vert\m N(w)\Vert_{L^{p'}(\sigma)}.
\end{equation}
Since (a) holds, then by Theorem \ref{thm.poisson} (more specifically, Remark \ref{rm.degnm}), we have that 
\[
\Vert\m N(w)\Vert_{L^{p'}(\sigma)}\lesssim\Vert\n C_{q'}(F)\Vert_{L^{p'}(\sigma)}=1,
\]
which yields (g) by the  estimates above.\hfill{$\square$}

{\bf Proof of }(h) $\iff$ (a). This is shown very similarly to the proofs of (f)$\implies$(a) (using the Caccioppoli inequality) and (a)$\implies$(g) (using Theorem \ref{thm.poisson}), we leave the details to the interested reader.

\section{Overview of the proof of Theorem \ref{thm.regularity}}\label{sec.strat}

The main goal of the rest of this manuscript is to prove Theorem \ref{thm.regularity}. In this section, we begin the proof of this theorem and delegate to further sections the more technical yet important details of the proof. Throughout this section we assume that $\Omega\subset\R^{n+1}$, $n\geq2$, is a bounded open set with uniformly $n$-rectifiable
boundary satisfying the corkscrew condition, and that $A$ is a real, not necessarily symmetric $(n+1)\times (n+1)$ matrix in $\Omega$ with measurable coefficients satisfying (\ref{eq.elliptic}) and the   DKP condition in $\Omega$.

First, to prove Theorem \ref{thm.regularity}, due to the linearity of the PDE and the density of the Lipschitz functions in $\dt W^{1,p}(\partial\Omega)$ for arbitrary $p\in(1,\infty)$, it suffices to establish the main estimate (\ref{eq.regest}) for $f\in\Lip(\partial\Omega)$.

Thus  assume that $f\in\Lip(\partial\Omega)$, and let $u$ be the unique solution to the problem $Lu=0$ in $\Omega$ and $u=f$ on $\partial\Omega$. By Proposition \ref{prop.duality}, we have that
\begin{equation}\label{eq.ntbound1}
\Vert\widetilde{\m N}_2(\nabla u)\Vert_{L^p(\partial\Omega)}\lesssim\sup_{F~:~\Vert\n C_{2}(F)\Vert_{L^{p'}(\partial\Omega)}=1}\Big|\int_{\Omega}F\nabla u\,dm\Big|.
\end{equation}
For the next step, we want to be able to control the right-hand side of (\ref{eq.ntbound1}), and it turns out that this may be done via studying the conormal derivative of solutions to Poisson problems. Given a weak solution $w$ to the Poisson-Dirichlet problem (\ref{eq.poisson}), define the functional
\begin{equation}\label{eq.funct}
	\ell_w(\varphi)=B[w,\Phi]:=\int_\Omega A\nabla w\nabla\Phi\,dm-\int_\Omega H\Phi\,dm-\int_\Omega F\nabla\Phi\,dm,\qquad\varphi\in\Lip(\partial\Omega),
\end{equation}
where $\Phi\in\Lip(\overline{\Omega})$ is a Lipschitz extension of $\varphi$ to $\Omega$, with $\Phi|_{\partial\Omega}=\varphi$. It is easy to see that $\ell_w(\varphi)$ is well-defined, since $w$ solves the equation $Lw=H-\dv F$ in the weak sense. We call $\ell_w(\varphi)$ the \emph{conormal derivative of $w$}, and denote $\partial_{\nu_A}w=\ell_w$.

With the concept of the conormal derivative at hand, let us see how to control the right-hand side of (\ref{eq.ntbound1}). By  Lemma \ref{lm.lipschitz}, we may assume without loss of generality that the supremum in the right-hand side of (\ref{eq.ntbound1}) runs over $F\in L^{\infty}_c(\Omega)$ with $F\in{\bf C}_{2,p'}$. Fix such $F$, and let $w$ be the unique weak solution to the transpose equation $L^*w=-\dv F$ in the space $Y_0^{1,2}(\Omega)$. It follows that
\begin{equation}\label{eq.ntbound2}
	\int_\Omega F\nabla u\,dm=\int_\Omega A^T\nabla w\nabla u\,dm-\partial_{\nu_{A^T}}w(f)=-\partial_{\nu_{A^T}}w(f),
\end{equation}
where in the last identity we used that $Lu=0$ in $\Omega$ and that $w\in Y_0^{1,2}(\Omega)$. Note that, since $A$ is a   DKP matrix, then so is $A^T$.

From (\ref{eq.ntbound2}), it is clear that we want an estimate on the conormal derivative of Poisson solutions. The following proposition shows that the conormal derivative $\partial_{\nu_A}w$ is a bounded linear functional on $\dt{W}^{1,p}$. It is the main ingredient for our solution to the regularity problem with non-smooth coefficients, as it completely takes the place of a 1-sided Rellich estimate. Indeed, the result can be seen as an alternative to Lemmas 5.1 and 5.2 of \cite{mt22} (see also Remark \ref{rm.lemmas}).  

\begin{proposition}[Conormal derivative of the Poisson-Dirichlet problem]\label{prop.conormal} Let $\Omega\subset\bb R^{n+1}$, $n\geq2$ be a bounded domain satisfying the corkscrew condition and with uniformly $n$-rectifiable boundary. Let $p>1$, $p'$ its H\"older conjugate, and $L=-\dv A\nabla$, where $A$ is a   DKP matrix in $\Omega$.  Assume that the homogeneous Dirichlet problem $(\Di_{p'}^L)$ is solvable   in $\Omega$, that $g\in L^{p'}(\partial\Omega)$ and that $F\in{\bf C}_{2,p'}$. Let $w$ be the unique solution to the Poisson Dirichlet problem (\ref{eq.poisson}) with $H\equiv0$. Then there exists a constant $C>0$ such that 
	\begin{equation}\label{eq.boundconormal}
		|\partial_{\nu_A}w(\varphi)|\leq C\big[\Vert g\Vert_{L^{p'}(\partial\Omega)}+\Vert\n C_{2}(F)\Vert_{L^{p'}(\partial\Omega)}\big]\Vert\nabla_{H,p}\varphi\Vert_{L^p(\partial\Omega)},\qquad\text{for each }\varphi\in\Lip(\partial\Omega).
	\end{equation}
\end{proposition}

\begin{remark}\label{rm.lemmas} Although Proposition \ref{prop.conormal} has a similar flavor to results in Section 5 of \cite{mt22}, our proof is quite different, in that we do not study the properties of the signed measure $L^*v_{\varphi}$ for $v_\varphi$ the almost $L^*$-elliptic extension of $\varphi\in\Lip(\partial\Omega)$ defined in Section \ref{sec.extension}.
\end{remark}

\begin{remark}\label{rm.detail} One detail that may seem purely technical but   is in fact quite important for our choreography is that we show (\ref{eq.boundconormal}) with the Carleson functional $\n C_2$ in the right-hand side, and not merely with $\n C_{p'}$, although controlling the latter is easier for the relevant term $T_{222}$, defined in (\ref{eq.break4}) and studied in Section \ref{sec.t222}. Since $\n C_2(F)\leq\n C_{p'}(F)$ pointwise on $\partial\Omega$ for $p\in(1,2]$, the result with $\n C_2$ is stronger.
\end{remark}

Proposition \ref{prop.conormal} is proved in Section \ref{sec.proofconormal} below, after proving an  important auxiliary estimate in Section \ref{sec.proofext}. Assuming for a moment that the above proposition holds, then putting the pieces (\ref{eq.ntbound1}), (\ref{eq.ntbound2}), and Proposition \ref{prop.conormal} together, we have that  
\[
\Vert\widetilde{\m N}_2(\nabla u)\Vert_{L^p(\partial\Omega)}\lesssim\sup_{F~:~\Vert\n C_{2}(F)\Vert_{L^{p'}(\partial\Omega)}=1}\Vert\n C_{2}(F)\Vert_{L^{p'}(\partial\Omega)}\Vert\nabla_{H,p}f\Vert_{L^p(\partial\Omega)}=\Vert\nabla_{H,p}f\Vert_{L^p(\partial\Omega)},
\]
which gives the proof of Theorem \ref{thm.regularity} modulo the proof of Proposition \ref{prop.conormal}. 


Let us finish this section with a few words about the idea of the proof of Proposition \ref{prop.conormal} and how the following sections fit together. Notice that when defining the conormal derivative $\partial_{\nu_A}w(\varphi)$ of a given boundary function $\varphi$, we are free to choose the extension $\Phi$ to the whole domain $\Omega$, and so a main part of our argument is to select an extension with certain desirable properties. Given $f\in\Lip(\partial\Omega)$, we construct the almost $L$-elliptic extension $v_f$ of $f$, which is designed using a corona decomposition of the rough domain into Lipschitz subdomains where both the Lipschitz character as well as the DKP condition are controlled. Roughly speaking, the extension $v_f$ is then chosen to solve the equation $Lv_f=0$ in each subdomain, while on the exteriors of the corona subdomains, $v_f$ equals a certain Lipschitz extension of $f$ with good gradient estimates. 

In the next section, Section \ref{sec.corona}, the details of the corona decomposition into Lipschitz subdomains is described. Then in Section \ref{sec.extension}, with the corona decomposition already at hand, the almost $L$-elliptic extension is defined and its properties addressed. Finally, in Section \ref{sec.reg} we carry out the details of the estimates that prove Proposition \ref{prop.conormal}.

\section{A corona decomposition into Lipschitz subdomains adapted to the DKP property}\label{sec.corona}

Throughout this section we assume that $\Omega\subset\R^{n+1}$, $n\geq1$, is a bounded open set with uniformly $n$-rectifiable
boundary satisfying the corkscrew condition, and that $A$ is a real, not necessarily symmetric $(n+1)\times (n+1)$ matrix in $\Omega$ with measurable coefficients satisfying (\ref{eq.elliptic}) and the   DKP condition in $\Omega$. Our objective in this section is to show that, if $\ell, \tau>0$ are arbitrary parameters, then there exists a corona decomposition of $\Omega$ into Lipschitz subdomains $\Omega_R$, with $R\in\ttt\subset\DD_\sigma$, and a buffer region $\n H$ like the one in \cite[Section 3.3]{mt22} such that
\begin{equation}\label{eq.corona}\nonumber
	\Omega=\bigcup_{R\in\ttt}\Omega_R\cup \n H,\qquad\text{where }\n H=\Omega\backslash\bigcup_{R\in\ttt}\Omega_R.
\end{equation}
Moreover, the domains $\Omega_R$ have uniform Lipschitz character and Lipschitz constant $\ell$ and the matrix $A$ satisfies the $\tau$-DPR condition in each $\Omega_R$. Obviously, we allow the packing condition of the corona decomposition to depend 
on $\tau$ and $\ell$.

We will follow very closely the arguments in Section 3 from \cite{mt22}.
The main difference is that in the construction below we introduce another stopping condition to ensure that $A$ satisfies the $\tau$-DPR condition in each domain $\Omega_R$, with
$R\in\ttt$.

Let us begin. For each $Q\in\DD_\sigma$ and a fixed $M\geq2$ large to be chosen below, we let
\begin{equation}\label{eqaaa1}\nonumber
	\lambda_\Omega(Q)= \ell(Q)^2 \sup\{|\nabla A(x)|^2:x\in B(x_Q,M\ell(Q))\cap\Omega,\,\delta_\Omega(x)\geq M^{-1}\ell(Q)\}.
\end{equation}

\begin{lemma}
	Suppose that $A$ satisfies the   DKP condition in $\Omega$. Then there exists $C(M)>0$ such that
	\begin{equation}\label{eqDKP*}
		\sum_{Q\in\DD_\sigma:Q\subset R}\lambda_\Omega(Q)\sigma(Q) \leq C(M)\,\sigma(R),\qquad\text{for any }R\in\m D_\sigma.
	\end{equation}
\end{lemma}

\begin{proof}
	For $a>0$, we denote 
	\begin{equation}\label{eqomegaa} 
		\Omega_a = \{x\in\Omega:\delta_\Omega(x)\geq a\}.
	\end{equation}
 	Given $Q\in\DD_\sigma$, applying the Besicovitch covering theorem we can cover $B(x_Q,M\ell(Q))\cap \Omega_{M^{-1}\ell(Q)}$ by a family of balls $B_i$, $i\in I_Q$, centered in points $x_i\in B(x_Q,M\ell(Q))\cap \Omega_{M^{-1}\ell(Q)}$, with radii $r_i= \delta_\Omega(x_i)/8$,
	so that moreover the family $\{B_i\}_{i\in I_Q}$ has finite superposition. From the fact that $r_i\approx_M \ell(Q)$ for each $i$ and the finite superposition, it easily follows that $\# I_Q\lesssim_M 1$.
	Then we  have
	\begin{multline}\nonumber
		\int_{B(x_Q,M\ell(Q))\cap \Omega_{M^{-1}\ell(Q)}}   \sup_{x\in B(z,\delta_\Omega(z)/2)} (\delta_\Omega(x)\,|\nabla A(x)|^2)\,dm(z)	\gtrsim_M \ell(Q) \sum_{i\in I_Q} \int_{B_i}  \sup_{x\in B(z,\delta_\Omega(z)/2)} |\nabla A(x)|^2\,dm(z)\\
		 \geq \ell(Q) \sum_{i\in I_Q}   \sup_{x\in B_i} |\nabla A(x)|^2\,m(B_i) \approx_M \ell(Q)^2 \sum_{i\in I_Q}   \sup_{x\in B_i} |\nabla A(x)|^2\,\sigma(Q) \\ \geq \,\ell(Q)^2 \sup_{B(x_Q,M\ell(Q))\cap\Omega_{M^{-1}\ell(Q)}}
		|\nabla A(x)|^2\,\sigma(Q) = \lambda_\Omega(Q)\,\sigma(Q),
	\end{multline}
	where in the second inequality we took into account that $B_i\subset B(z,\delta_\Omega(z)/2)$ for any $z\in B_i$.
	Therefore,
	\begin{align*}
		\sum_{Q\in\DD_\sigma:Q\subset R}\lambda_\Omega(Q)\sigma(Q) & \lesssim_M\sum_{Q\in\DD_\sigma:Q\subset R}
		\int_{B(x_Q,M\ell(Q))\cap \Omega_{M^{-1}\ell(Q)}}  \sup_{x\in B(z,\delta_\Omega(z)/2)} (\delta_\Omega(x)\,|\nabla A(x)|^2)\,dm(z)\\
		& \lesssim \int_{B(x_R,C(M)\ell(R))}  \sup_{x\in B(z,\delta_\Omega(z)/2)} (\delta_\Omega(x)\,|\nabla A(x)|^2)\,dm(z)
		\lesssim_M \sigma(R),
	\end{align*}
	by the finite superposition of   $B(x_Q,M\ell(Q))\cap \Omega_{M^{-1}\ell(Q)}$, with $Q\in\DD_\sigma$, $Q\subset R$, and the   DKP
	condition.
\end{proof}

\subsection{The approximating Lipschitz graph}\label{subs:approxLipgraph}

In this subsection we describe how to associate an approximating Lipschitz graph to a cube $R\in \DD_\sigma$, assuming $b\beta_{\sigma}(k_1R)$ to be small enough for some big constant $k_1>2$ {(where we denoted $b\beta_\sigma\equiv b\beta_{\supp\sigma}$).}
The construction is based on the arguments in \cite[Chapters 7, 8, 12, 13, 14]{ds1}.
The first step consists in defining suitable stopping cubes.

Given $x\in\R^{n+1}$, we write $x= (x',x_{n+1})$.
For a given cube $R\in\DD_\sigma$, we denote by $L_R$ a best approximating hyperplane for
$b\beta_\sigma(k_1R)$.
We also assume, without loss of generality, that 
$$
L_R \,\,\textup{is the horizontal  hyperplane}\,\, \{x_{n+1}=0\}.
$$ 
We let $x_R^0$ be the orthogonal projection of $x_R$ on $L_R$ (i.e., $x_R^0=((x_R)',0)$). For a given $\gamma\in(0,1/4)$ to be chosen below,
we let
$$x_R^+ = \big((x_R)',(\tfrac14 - \gamma) r(B(R))\big), \qquad x_R^- = \big((x_R)',(-\tfrac14 + \gamma) r(B(R))\big),$$
and we consider the balls
$$B^+(R) = B(x_R^+,\tfrac14r(B)),\qquad B^-(R) = B(x_R^-,\tfrac14r(B)),$$
and the cylinders 
$$C(R) = \big\{x\in \R^{n+1}: |x'-(x_R)'|\leq \tfrac14\gamma^{1/2} r(B(R)),\,
{|x_{n+1}| }\leq \tfrac14\gamma\, r(B(R))\big\},$$
$$C'(R) = \big\{x\in \R^{n+1}: |x'-(x_R)'|\leq \tfrac14 r(B(R)),\,
{|x_{n+1}| }\leq \gamma\, r(B(R))\big\}.$$
It is easy to check that $C(R) \subset B^+(R) \cap B^-(R)\subset C'(R)$. Also, in the case when $\dist(x_R,L_R)\leq r(B(R))/4$, say, we have 
\begin{equation}\label{eqinc77}
	B^+(R)\cup B^-(R)\subset B(R).
\end{equation}

We consider constants $\ve,\delta$, and $\theta$ such that 
$0<\ve\ll\delta\ll\gamma\leq 1/4$ and $0<\theta\ll1$ to be chosen later (depending on the corkscrew condition and the uniform rectifiability constants), $k_1>2$, and we denote by $\cB$ or $\cB(\ve,\theta)$ the family of cubes $Q\in
\DD_\sigma$ such that either
$b\beta_\sigma(k_1Q) > \ve$ or $\lambda_\Omega(Q)>\theta$.

We consider now $R\in\DD_\sigma\setminus \cB$. 
Notice that in this case \rf{eqinc77} holds, and also
\begin{equation}\label{eqinc78}
	\big\{x\in R: |x'-(x_R)'|\leq \tfrac14\gamma^{1/2} r(B(R))\big\}\subset C(R).
\end{equation}
We let $\sss(R)$ be the family of maximal cubes $Q\in\DD_\sigma(R)$ such that at least one of the following holds:
\begin{itemize}
	\item[(a)] {$Q\cap B^+(R) \cap B^-(R)= \varnothing$.}
	\item[(b)] $Q\in\cB(\ve,\theta)$, i.e., $b\beta_\sigma(k_1Q) > \ve$ or $\lambda_\Omega(Q)>\theta$.
	\item[(c)] $\angle(L_Q,L_R)> \delta$, where $L_Q$, $L_R$ are best approximating hyperplanes for $\beta_{\sigma,\infty}(k_1Q)$ and $\beta_{\sigma,\infty}(k_1R)$, respectively, and {$\angle(L_Q,L_R)$ denotes the angle between $L_Q$ and $L_R$.}
	\item[(d)] $\sum_{P\in\DD_\sigma:Q\subset P\subset R}\lambda_\Omega(P)\geq \theta$.
\end{itemize}
We denote by $\tree(R)$ the family of cubes in $\DD_\sigma(R)$ which are not strictly contained
in any cube from $\sss(R)$. We also consider the function
$$d_R(x) = \inf_{Q\in\tree(R)} \big(\dist(x,Q) + \diam(Q)\big).$$
Notice that $d_R$ is $1$-Lipschitz.
Assuming $k_1$ big enough (but independent of $\ve$ and $\delta$) and arguing as in the proof of  \cite[Proposition 8.2]{ds1}, the following holds:

\begin{lemma}\label{lemgraf}
	Denote by $\Pi_R$ the orthogonal projection on $L_R$.
	There is a Lipschitz function $A:L_R \to L_R^\bot$ with slope at most $C\delta$ such that
	$$\dist(x,(\Pi_R(x),A(\Pi_R(x)))) \leq C_1\ve\,d_R(x)\quad \mbox{ for all $x\in k_1R$.}$$
\end{lemma}

Remark that in this lemma, and in the whole subsection, we assume that {$R$ is} as 
above, so that, in particular, $b\beta_\sigma(k_1R)\leq \ve$.

We denote
$$D_R(x)= \inf_{y\in\Pi_R^{-1}(x)}d_R(y).$$
It is immediate to check that $D_R$ is also a $1$-Lipschitz function. Further, as shown in \cite[Lemma
8.21]{ds1}, there is some fixed constant $C_2$ such that
\begin{equation}\label{eqDR}
	C_2^{-1}d_R(x) \leq D_R(x) \leq d_R(x)\quad \mbox{ for all $x\in {3B(R)}$.}
\end{equation}

We denote by $Z(R)$ the set of points $x\in R$ such that {$d_R(x)=0$.}
The following lemma is an easy consequence of {the} results  in  \cite[Chapters 7, 12-14]{ds1}, although it is not stated explicitly in \cite{ds1}.

\begin{lemma}\label{lempack1}
	There are some constants $C_3(\ve,\delta)>0$ and $k_1\geq 2$ such that
	\begin{align}\label{eqlempack1}
		\sigma(R) \approx_\gamma \sigma(C(R))& \leq 2\,\sigma(Z(R)) + 2\sum_{Q\in\sss(R)\cap\cB(\ve,\theta)} \!\sigma(Q) \\
		& \quad + C_3 \sum_{Q\in\tree(R)} \!\beta_{\sigma,1}(k_1Q)^2\,\sigma(Q)+  2\,\theta^{-1} \sum_{Q\in\tree(R)} \!\lambda_\Omega(Q)\,\sigma(Q).
		\notag
	\end{align}
\end{lemma}
 
\begin{proof}
	The fact that $\sigma(R) \approx_\gamma \sigma(C(R))$ is an immediate consequence of the inclusion \rf{eqinc78} and the $n$-Ahlfors regularity of $\sigma$. Denote $\cF_1= \big\{R\in\DD_\sigma:\sigma\big(\bigcup_{Q\in \sss(R)\cap {\rm{(c)}}} Q\big)\geq \sigma(C(R))/2\big\}$, 	where $Q\in \sss(R)\cap \rm{(c)}$ means that $Q$ satisfies the condition (c) in the above definition of $\sss(R)$.
	Notice that $\cF_1$ is very similar to the analogous set $\cF_1$ defined in \cite[p.39]{ds1}. A (harmless) difference is that we wrote
	$\sigma(C(R))/2$ in the definition above, instead of $\sigma(R)/2$ as in \cite{ds1}. Assuming $\ve>0$ small enough (depending on $\delta$) in the definition of 
	$\cB(\ve,\theta)$,
	in equation (12.2) from \cite{ds1} (proved along the Chapters 12-14) it is shown that 
	there exists some $k>1$ (independent of $\ve$ and $\delta$) such that if $R\in \cF_1$, then 
	$$\iint_X  \beta_{\sigma,1}(x,kt)^2\,\frac{d\sigma(x)\,dt}t \gtrsim_\delta \sigma(R),$$
	where $X= \big\{(x,t)\in \supp\sigma \times (0,+\infty): x\in kR,\,k^{-1}d_R(x)\leq t\leq k\ell(R)\big\}$. 	It is easy to check that, choosing $k_1>k$ large enough,
	$$\iint_X  \beta_{\sigma,1}(x,t)^2\,\frac{d\sigma(x)\,dt}t \lesssim_k\sum_{Q\in \tree(R)}\beta_{\sigma,1}(k_1Q)^2\,\sigma(Q).$$
	Hence, \rf{eqlempack1} holds when $R\in\cF_1$.
	
	In the case $R\not\in\cF_1$, by the definition of $\sss(R)$, taking into account that $C(R)\subset B^+(R)\cap B^-(R)$, we have
	\begin{align*}
		\sigma(R) \approx_\gamma \sigma(C(R)) & \leq \sigma(Z(R)\cap C(R)) + \sum_{Q\in\sss(R)\cap\cB(\ve,,\theta)} \sigma(Q\cap C(R))\\
		\quad &+ \sum_{Q\in\sss(R)\cap \rm{(c)}} \!\sigma(Q\cap C(R)) +  \sum_{Q\in\sss(R)\cap \rm{(d)}} \!\sigma(Q\cap C(R))
		.
	\end{align*}
	Since the third sum on the right hand side does not exceed $\sigma(C(R))/2$, we deduce that
	\begin{equation}\label{eqstop88} 
		\frac12\,\sigma(C(R)) \leq \sigma(Z(R)\cap C(R)) + \sum_{Q\in\sss(R)\cap\cB(\ve,\theta)}\! \sigma(Q\cap C(R))  +  \sum_{Q\in\sss(R)\cap \rm{(d)}} \!\sigma(Q\cap C(R)).
	\end{equation}
	finally with the last sum. By condition (d) we have
	\begin{align*}
		\sum_{Q\in\sss(R)\cap \rm{(d)}} \sigma(Q\cap C(R))& \leq \theta^{-1}\!\sum_{Q\in\sss(R)}\,\sum_{P\in\DD_\sigma:Q\subset P\subset R}\!\lambda_\Omega(P)\,\sigma(Q) \\ &= \theta^{-1}\!\sum_{P\in\tree(R)} \!\lambda_\Omega(P) \!\sum_{Q\in\sss(R):Q\subset P}\sigma(Q) \leq \theta^{-1}\!\sum_{P\in\tree(R)} \!\lambda_\Omega(P) 
		\sigma(P).
	\end{align*}
	Plugging this equation into \rf{eqstop88}, the lemma follows.
\end{proof}

\subsection{The starlike Lipschitz  subdomains $\Omega_R^\pm$}\label{subs:star-Lip}

Abusing notation,   we write  $D_R(x')=D_R(x)$ for $x=(x',x_{n+1})$. 

\begin{lemma}\label{lem333}
	Let $$U_R = \{x\in B^+(R): x_{n+1}> A(x') + C_1C_2\ve D_R(x')\},$$
	$$V_R = \{x\in B^-(R): x_{n+1}< A(x') - C_1C_2\ve D_R(x')\},$$
	and
	$$W_R = \{x\in B(R): A(x') - C_1C_2\ve D_R(x') \leq x_{n+1}\leq A(x') + C_1C_2\ve D_R(x')\}.$$
	Then $\pom\cap B(R)\subset W_R$. Also,
	$U_R$ is either contained in $\Omega$ or in  $\R^{n+1}\setminus
	\overline\Omega$, and the same happens with $V_R$. Further, at least one of the sets $U_R$, $V_R$ is contained in $\Omega$.
\end{lemma}

Remark that it may happen that $U_R$ and $V_R$ are both contained in $\Omega$, or that one set is contained in $\Omega$ and the other in $\R^{n+1}\setminus
\overline\Omega$.

\begin{proof}
	Let us see that $\pom\cap B(R) \subset W(R)$. Indeed, we have
	$ \partial\Omega\cap B(R) \subset R$, by the definition of $B(R)$.
	Then, by Lemma \ref{lemgraf} and \eqref{eqDR}, for all $x\in \partial\Omega\cap B(R)$ we have
	$$|x- (x',A(x'))|\leq C_1\ve\,d_R(x) \leq C_1C_2\ve\,D_R(x),$$
	which is equivalent to saying that $x\in W_R$.

	Next we claim that if $U_R\cap \Omega\neq \varnothing$, then $U_R\subset \Omega$.
	This follows from connectivity, taking into account that if $x\in U_R\cap \Omega$ and $r=\dist(x,\partial U_R)$, then $B(x,r)\subset \Omega$. {Otherwise, there exists some point $y\in B(x,r)\setminus
		\overline \Omega$, and thus there exists some $z\in\pom$ which belongs to the segment $\overline{xy}$.} This would contradict the fact that $\pom\subset W_R$.
	The same argument works replacing $U_R$ and/or $\Omega$ by $V_R$ and/or $\R^{n+1}\setminus\overline{\Omega}$, and thus we deduce that any of the sets $U_R$, $V_R$ is contained either in $\Omega$ or in $\R^{n+1}\setminus\overline{\Omega}$. 
	
	Finally, from the 	corkscrew condition we can find  a point $y\in {B(x_R,r(B(R)))}\cap \Omega$ with $\dist(y,\pom)\gtrsim r(B(R))$. {So if $\ve,\delta$ are small enough we deduce that $y\in (U_R \cup V_R)\cap \Omega$ because $b\beta_\sigma(k_1R)\leq \ve$ and both $\pom\cap B(R)$ and the graphs of $A$ in $B(R)$ are contained in
		a $C\delta\ell(R)$-neighborhood of the hyperplane $\{x_{n+1}=0\}$.}
	Then by the discussion in the previous paragraph, we infer that either 
	$U_R\subset\Omega$ or $V_R\subset\Omega$. 
\end{proof}

Suppose that $U_R\subset \Omega$. We denote $B^0(R) = B(x_R^0,r(B(R)))$ and we let $\Gamma_R^+$ be the Lipschitz graph of the function 
$B^0(R)\cap L_R\ni x'\mapsto A(x') + \delta\,D_R(x')$. Notice that this is a Lipschitz function with slope at most 
$C\delta< 1$ (assuming $\delta$ small enough).  Then we define
$$\Omega_R^+ =\big\{x=(x',x_{n+1}) \in B^+(R): x_{n+1}> A(x') + \delta \,D_R(x')\big\}.$$
Observe that $\Omega_R^+$ is a {starlike Lipschitz domain (with uniform Lipschitz character) }and that $\Omega_R^+\subset U_R$, assuming that $C_1C_2\ve\ll\delta$.

In case that $V_R\subset \Omega$, we define $\Gamma_R^-$ and $\Omega_R^-$ analogously, replacing the above function $A(x') + \delta\,D_R(x')$ by $A(x') - \delta\,D_R(x')$. 
If $V_R\subset \R^{n+1}\setminus \overline\Omega$, then we set $\Omega_R^-=\varnothing$.
In any case, we define
$$\Omega_R=\Omega_R^+ \cup\Omega_R^-.$$
From Lemma \ref{lem333} and the assumption that $C_1C_2\ve\ll\delta$,
it is immediate to check that
\begin{equation}\label{eqsep99}
	\dist(x,\partial\Omega)\geq \frac\delta2\,D_R(x)\quad\mbox{ for all $x\in\Omega_R$.}
\end{equation}

For a given $a>1$, we say that two cubes $Q,Q'$ are $a$-close if 
$$ a^{-1}\ell(Q)\leq\ell(Q')\leq a\ell(Q) \; \text{ and }\;\dist(Q,Q')\leq a(\ell(Q)+\ell(Q')).
$$
We say that $Q\in\DD_\sigma$ is $a$-close to $\tree(R)$ if there exists some $Q'\in\tree(R)$ such that
$Q$ and $Q'$ are $a$-close.
For $1<a^*<a^{**}$ to be fixed below, {we} define the augmented trees
$$\tree^*(R) = \{Q\in\DD_\mu: \text{$Q$ is $a^*$-close to $\tree(R)$}\},$$
$$\tree^{**}(R) = \{Q\in\DD_\mu: \text{$Q$ is $a^{**}$-close to $\tree(R)$}\}.$$
Obviously, $\tree(R)\subset\tree^*(R)\subset\tree^{**}(R)$. Notice also that the families of cubes from $\tree^*(R)$ or $\tree^{**}(R)$ may not be trees.

We now let $\m W(U)$ be a $(1/4)$-Whitney decomposition of $\Omega$ (see Section \ref{sec.whitney}).

\begin{lemma}\label{lemcontingtree}
	Assuming $a^*>1$ to be big enough, we have
	$$\overline\Omega_R\cap\Omega\subset \bigcup_{Q\in\tree^*(R)} w(Q).$$
\end{lemma}

{Recall that $w(Q)$ is the Whitney region associated with $Q$ (see Section \ref{sec.whitney}).}
Notice that if $\Omega_R^-\neq\varnothing$, it may happen that $w(Q)$ is the union of some
Whitney cubes contained in $\Omega^+_R$ and others in $\Omega^-_R$, for example.

\begin{proof}
	Let $P\in\WW(\Omega)$ be such that $P\cap \overline\Omega_R\neq \varnothing$ and fix $Q\in b(P)$. It suffices to show that $Q\in\tree^*(R)$ if $a^*$ is taken big enough.
	To this end, we   show that there exists some $Q'\in\tree(R)$ which is $a^*$-close to $Q$.
	
	Notice first that $\ell(P)\leq C_5\,\ell(R)$ for some fixed constant $C_5$, because $P$ intersects $\overline\Omega_R$ and thus $\overline{B(R)}$. Let $x\in P\cap \overline \Omega_R$. Then by \eqref{eqsep99} we have
	$\frac\delta2\,D_R(x)\leq \dist(x,\partial\Omega) \approx \ell(P)=\ell(Q)$, 	for $Q$ as above. Thus, $d_R(x_Q)\approx D_R(x_Q) \leq { D_R(x)} + C\,\ell(Q) \lesssim \delta^{-1}\ell(Q)$. 	From the definition of $d_R$ we infer that there exists some cube $Q'\in\tree(R)$ such that $\ell(Q') + \dist(Q,Q') \leq C\,\delta^{-1}\ell(Q)$.
	
	If $\ell(Q')\geq C_5^{-1}\ell(R)$, we let $Q''=Q'$. Otherwise, we let $Q''$ be an ancestor of $Q'$ belonging to $\tree(R)$ and satisfying $C_5^{-1}\ell(Q)\leq \ell(Q'')< 2C_5^{-1}\ell(Q)$. 	The above condition $\ell(Q)=\ell(P)\leq C_5\,\ell(R)$ ensures the existence of $Q''$. Then, in any case, it  follows  that $Q'$ is $a^*$-close to $Q$, for $a^*$ big enough depending on $\delta$.
\end{proof}

In the rest of the lemmas in this subsection, we assume, without loss of generality, that $\Omega_R^+\subset U_R
	\subset\Omega$.

\begin{lemma}
	If $Q\in\tree(R)$, then $\dist(w(Q),\Omega^+_R)\leq C\,\ell(Q)$. 	Also, if $\Omega_R^-\neq\varnothing$, $\dist(w(Q),\Omega^-_R)\leq C\,\ell(Q)$.
\end{lemma}

\begin{proof} We will prove the first statement. The second one follows by the same arguments.
	It is clear that $\dist(w(Q),\Omega_R^+)\leq C\,\ell(R)$, 
	and so the statement above holds if $\ell(Q)\gtrsim\ell(R)$.

	So we may assume that $\ell(Q)\leq c_1\,\ell(R)$ for some small $c_1$ to be fixed below. 
	By construction, the parent $\wh Q$ of $Q$ satisfies $\wh Q\cap B^+(R)\cap B^-(R)\neq\varnothing$.
	Thus there exists   $z\in B^+(R)\cap B^-(R)$ such that $|z-x_Q|\lesssim \ell(Q)$.
	Clearly, it holds $\dist(z,\pom)\lesssim \ell(Q)$. 	On the other hand, we consider the point $x=(z',z_{n+1}+2\ell(Q))$, so that 
	\begin{equation}\label{eqnew77}
		\dist(x,\pom\cup\Gamma_R^+)\geq \ell(Q).
	\end{equation}
	By  definition of $d_R$ and $D_R$, $D_R(x_Q)\leq d_R(x_Q) \leq \ell(Q)$. Hence, $D_R(x) \leq D_R(x_Q) + C\,\ell(Q)\lesssim\ell(Q) \leq\delta_\Omega(x)$. 
	Assuming $\delta$ small enough, we deduce that $10\delta\,D_R(x)\leq \dist(x,\pom) \leq C\ell(Q)\leq c_1 C\ell(R)$. 	By the definition of {$\Omega_R^+$, it is easy to check that this implies that $x\in\Omega_R^+$} if $c_1$ is small enough. Indeed, since $z\in B^+(R)\cap B^-(R)$, {by \rf{eqnew77} and the last estimate, $A(x') + \delta\,D_R(x')< x_{n+1}\ll_{c_1} \ell(R)$. }
\end{proof} 

We denote $\partial\tree^{**}(R):= \{Q\in\tree^{**}(R): w(Q)\not \subset\Omega_R\}$.
 
\begin{lemma}\label{lemdtree**}
	For all $S\in\DD_\sigma$, we have
	$$\sum_{Q\in\partial\tree^{**}(R):Q\subset S} \sigma(Q) \lesssim_{a^{**}}\sigma(S).$$
\end{lemma}

\begin{proof}
	We will prove the following:
	\begin{claim*}
		For each $Q\in\partial\tree^{**}(R)$ there exists some cube $P=P(Q)\in\WW(\Omega)$ such that
		$$P\cap \pom_R\neq \varnothing,\qquad \ell(P)\approx_{a^{**},\delta} \ell(Q), \qquad \dist(P,Q) \lesssim_{a^{**},\delta}\ell(Q).$$
	\end{claim*}
	
	The lemma follows easily from this claim. Indeed, using that $\Omega_R$ is either a Lipschitz domain or a union of two Lipschitz domains, that
	\begin{equation}\label{eqclaim1}
		\HH^n( 2P \cap \pom_R)\approx \ell(P)^n\approx_{a^{**},\delta} \sigma(Q),
	\end{equation}
	the finite superposition of the cubes $2P$, and the fact that  $\#\{P\in\WW(\Omega):P=P(Q)\}\leq C(a^{**},\delta)$ for every 	$Q\in\partial\tree^{**}(R)$, we get
	\begin{multline}\nonumber
		\sum_{Q\in\partial\tree^{**}(R):Q\subset S} \sigma(Q) 
		  \lesssim_{a^{**},\delta} \!\! \sum_{\substack{P\in \WW(\Omega):\\P\subset B(x_S,C(a^{**})\ell(S))}} \HH^n(2P\cap \pom_R) 		   \lesssim_{a^{**}} \HH^n\big(\pom_R\cap B(x_S,C'(a^{**})\ell(S))\big) \lesssim
		\ell(S)^n.
	\end{multline}
	
	To prove the claim we distinguish several cases:
	
	\noindent\emph{Case 1.} If $\ell(Q)\geq c(a^{**},\delta)\ell(R)$ (with
	$c(a^{**},\delta)$ to be chosen), we let $P(Q)$ be any Whitney cube that intersects the upper half of $\partial B^+(R)$. It is immediate to check that this choice satisfies the properties described in \eqref{eqclaim1}.
	
	\noindent\emph{Case 2.} 	Suppose now that $\ell(Q)\leq c(a^{**},\delta)\ell(R)$ and that $\dist(Q,\partial (B^+(R)\cup B^-(R)))\geq C_6\,\ell(Q)$ for some
	big $C_6(a^{**})>1$ to be chosen below. Let us see that this implies that $x_Q\in C(R)$. Indeed, from the definition of $\tree^{**}(R)$ there exists some $S\in\tree(R)$ such that $Q$ and $S$ are $a^{**}$-close. Since $S\cap (B^+(R)\cap B^-(R))\neq\varnothing$,
	there exists some $\wt x_S\in S\cap B^+(R)\cap B^-(R)$. If $x_Q\not\in (B^+(R)\cup B^-(R))$, by continuity
	the segment $\overline{x_Q \,\wt x_S}$ intersects $\partial (B^+(R)\cup B^-(R))$ at some point $z$. So we have
	$$\dist(x_Q,\partial (B^+(R)\cup B^-(R)))\leq |x_Q-z| \leq |x_Q- x_{\wt S}| \leq A (\ell(Q)+\ell(S)) + \diam(Q) \leq C(a^{**})\ell(Q),$$
	which contradicts the assumption above if $C_6(a^{**})$ is big enough.
	In particular,   the conditions that $x_Q\in C(R)$ and  $\dist(Q,\partial (B^+(R)\cup B^-(R)))\geq C_6\,\ell(Q)$ 	imply that $w(Q)\subset (B^+(R)\cup B^-(R))$ if $C_6$ is taken big enough.
	
	If $w(Q)\cap\Omega_R\neq\varnothing$, then we take a Whitney cube $P$ with $\ell(P)=\ell(Q)$ contained in $w(Q)$ that intersects $\Omega_R$. Otherwise, $w(Q)\subset \Omega\setminus \Omega_R$, and from the fact that $\ell(Q)\leq c(a^{**},\delta)\ell(R)$ we infer that $w(Q)$ lies below the Lipschitz graph 
	$\Gamma_R^+$ that defines the bottom of $\pom_R$, and above the graph $\Gamma_R^-$ in case that $\Omega_R^-\neq\varnothing$.
	Then we take $x\in w(Q)$ and $x^+ =\Pi^{-1}_R(x)\cap \Gamma_R^+$, and also $x^- =\Pi_R^{-1}(x)\cap \Gamma_R^-$ 
	in case that $\Omega_R^-\neq\varnothing$.
	
	When  $\Omega_R^-=\varnothing$, we let $P$ be the Whitney that contains $x^+$. Since
	$V_R\subset\R^{n+1}\setminus \overline\Omega$, 
	there exists $y = \Pi_R^{-1}(x) \cap
	\pom \cap B^0(R)$. Then we deduce $\ell(P)\approx\dist(x^+,\pom) \approx |x^+- y|\geq \HH^1(w(Q)\cap \Pi_R^{-1}(x^+)) \geq \ell(Q)$. 	Using again that there exists some $S\in\tree(R)$ such that $Q$ and $S$ are $a^{**}$-close we get
	\begin{multline}\label{eqcase11}
		\ell(P)  \approx\dist(x^+,\pom)\lesssim D_R(x^+) = D_R(x) \leq D_R(x_Q) + C\,\ell(Q)\\
		 \leq D_R(x_S) + |x_Q - x_S| +C\,\ell(Q) \lesssim C(a^{**},\delta)\ell(Q). 
	\end{multline}
	Further,
	\begin{equation}\label{eqcase12}
		\dist(P,Q) \leq |x_Q-x^+| \leq |x_Q-x| +|x- x^+| \lesssim \ell(Q),
	\end{equation}
	and so $P$ satisfies the properties in the claim.
	
	If $\Omega_R^-\neq\varnothing$ (i.e., $V_R\subset\Omega$), we let $P$ be the largest Whitney cube that intersects $\{x^+,x^-\}$.
	From the fact that $b\beta_\sigma(k_1Q)\lesssim\ve$ and the stopping condition (c) we easily infer that there
	exists some point $y\in  \Pi_R^{-1}(x^+) \cap B^0(R)$ such that $\dist(y,\pom)\lesssim\ve\ell(Q)$.
	Then it follows $\ell(P) \gtrsim |x^+ - x^-|\geq \HH^1(w(Q)\cap \Pi_R^{-1}(x^+)) \geq \ell(Q)$. 	Also the estimates \eqref{eqcase11} and \eqref{eqcase12} are still valid, replacing $x^+$ by $x^-$ if $x^-\in P$. So $P$ satisfies the required properties.

\noindent\emph{Case 3.} 	Suppose that $\ell(Q)\leq c(a^{**},\delta)\ell(R)$ and that $\dist(Q,\partial (B^+(R)\cup B^-(R)))< C_6\,\ell(Q)$ for
	$C_6(a^{**})>1$ as above. 	So there exists   $z\in
	\partial (B^+(R)\cup B^-(R))$ such that $|x_Q-z|\lesssim_{a^{**}} \ell(Q)$. We also denote $z^+= \Pi^{-1}(z) \cap\Gamma_R^+$.
	As above, we take $S\in\tree(R)$ such that $Q$ and $S$ are $a^{**}$-close, so that since $D_R$ is $1$-Lipschitz, we have
	$$\dist(z^+, \pom)\approx D_R(z^+)=
	D_R(z) \leq D_R(x_S) + |z-x_Q| + |x_Q - x_S|\lesssim_{a^{**}} \ell(Q).$$
	From this fact we infer that 
	there exists some point $y\in\pom_R$ such that $|z-y|\approx |z_{n+1}-y_{n+1}|
	\lesssim_{a^{**}}\ell(Q)$ and $\dist(y,\pom)\approx_{a^{**}}\ell(Q)$.	This point satisfies $|x_Q -y|\leq |x_Q - z| + |z-y| \lesssim_{a^{**}} \ell(Q)$, 	and so letting $P$ be the Whitney cube that contains $y$ we are done.
\end{proof}

\subsection{The corona decomposition of $\Omega$}\label{subs:corona}

We will now perform a corona decomposition of $\Omega$ using the Lipschitz subdomains $\Omega_R$ constructed above.
We define inductively a family $\ttt\subset \DD_\sigma$ as follows. First we let $R_0\in\DD_\sigma$ be a cube {such that 
	$b\beta_\sigma(k_1R_0)\leq\ve$ having maximal side length.} Assuming $R_0,R_1,\ldots, R_i$ to be defined, we let $R_{i+1}\in\DD_\sigma$ be a cube from $\DD_\sigma \setminus \bigcup_{0\leq k\le i} \tree^{**}(R_k)$ such that $b\beta_\sigma(k_1R_{i+1})\leq \ve$ with maximal side length.
We set
$$\ttt=\{R_i\}_{i\geq0}.$$

For each $R\in\ttt$ we consider the subdomain $\Omega_R$ constructed in the previous subsection. We
split
\begin{equation}\label{eqpart93}
	\Omega = \bigcup_{R\in\ttt} \Omega_R \cup \n H, \quad\text{ where }\quad \n H=\Omega \setminus\bigcup_{R\in\ttt} \Omega_R.
\end{equation}

\begin{lemma}
	The sets $\overline\Omega_R\cap\Omega$, with $R\in\ttt$, are pairwise disjoint, assuming that the {constant $a^{**}$ is big enough (possibly depending on $a^*$). }
\end{lemma}

In particular, from this lemma it follows that the union in \eqref{eqpart93} is a partition into 
disjoint sets. The constants {$a^*$ and $a^{**}$} depend on $\delta$; however, this dependence
is harmless for our purposes.

\begin{proof}
Suppose that $R,R'\in\ttt$ satisfy {$\overline\Omega_R\cap\overline\Omega_{R'}\cap\Omega\neq\varnothing$.}
Suppose also that $R=R_i$, $R'=R_j$, with $j>i$, so that in particular $\ell(R')\leq \ell(R)$.
From Lemma \ref{lemcontingtree} we infer that there exist cubes $Q\in\tree^*(R)$ and $Q'\in\tree^*(R')$ such that $w(Q)\cap w(Q')\neq\varnothing$. Clearly, this implies that $\ell(Q)\approx \ell(Q')$, and from the definition of
$\tree^*(R)$ and $\tree^*(R')$ we deduce that there are two cubes $S\in\tree(R)$, $S'\in\tree(R')$
such that $\dist(S,S')\lesssim_a\ell(S)\approx_a\ell(S')$. Let $\wt S$ be the ancestor of $S$ with 
$\ell(\wt S)=\ell(R')$ {(or take $\wt S=S$ if $\ell(S)>\ell(R')$). Clearly, $\wt S\in\tree(R)$} and $\dist(\wt S,R')\lesssim_{a^*}\ell(R')$.
So $R'\in\tree^{**}(R)$ {if $a^{**}=a^{**}(a^*)$} is chosen big enough, which contradicts the construction of $\ttt$.
\end{proof}

\begin{lemma}\label{lempack2}
	The family $\ttt$ satisfies the packing condition
	$$\sum_{R\in\ttt:R\subset S} \sigma(R)\lesssim_{\ve,\delta} \sigma(S)\quad \mbox{ for all $S\in\DD_\sigma$}.$$
\end{lemma}

\begin{proof}
	By Lemma \ref{lempack1} we have
	\begin{multline}\label{eqlala77}
		\sum_{R\in\ttt:R\subset S} \sigma(R)  \lesssim
		\sum_{R\in\ttt:R\subset S} \sigma(Z(R)) 
		+  \sum_{R\in\ttt}\,\sum_{Q\in\sss(R)\cap\cB(\ve,\theta)} \sigma(Q)\\
		  \quad
		+ \sum_{R\in\ttt:R\subset S}\sum_{Q\in\tree(R)} \big(\beta_{\sigma,1}(k_1Q)^2+\lambda_\Omega(Q)\big)\,\sigma(Q).
	\end{multline}
	By construction, the sets $Z(R)$ are disjoint, and thus the first sum does not exceed $\sigma(S)$.
	The second term does not exceed $\sum_{Q\in\DD_\sigma(S)\cap\cB(\ve,\theta)}\sigma(Q)\lesssim_{\ve,\theta}\sigma(S)$, 	by the uniform rectifiability of $\pom$ and the   DKP condition \rf{eqDKP*}.
	Concerning the last term in \eqref{eqlala77}, the families $\tree(R)$, with $R\in\ttt$, are also disjoint by construction. Therefore, {again by the} uniform rectifiability of $\pom$,
	$$\sum_{R\in\ttt:R\subset S}\sum_{Q\in\tree(R)} \beta_{\sigma,1}(k_1Q)^2\,\sigma(Q) \leq
	\sum_{Q\subset S} \beta_{\sigma,1}(k_1Q)^2\,\sigma(Q)\lesssim_{\ve,\delta} \sigma(S).$$
	Analogously, by \rf{eqDKP*},
	$$\sum_{R\in\ttt:\subset S}\sum_{Q\in\tree(R)} \lambda_\Omega(Q)\,\sigma(Q)\leq
	\sum_{Q\subset S} \lambda_\Omega(Q)\,\sigma(Q)\lesssim \sigma(S).$$
\end{proof}

\begin{lemma}\label{lm.h}
	There is a subfamily $\HH\subset \DD_\sigma$ such that
	\begin{equation}\label{eqHH*}
		\n H\subset \bigcup_{Q\in\HH} w(Q)
	\end{equation}
	which satisfies the packing condition
	\begin{equation}\label{eqHH*2}
		\sum_{Q\in \HH:Q\subset S} \sigma(Q)\lesssim \sigma(S)\quad \mbox{ for all $S\in\DD_\sigma$},
	\end{equation}
	with the implicit constant depending on $\ve,\delta,a^{**}$.
\end{lemma}

\begin{proof}
	By construction, $\DD_\sigma \subset \cB\cup \bigcup_{R\in\ttt} \tree^{**}(R)$, 	and thus 	$$\Omega \subset \bigcup_{Q\in\cB} w(Q) \cup \bigcup_{R\in\ttt}\,\bigcup_{Q\in\partial\tree^{**}(R)}w(Q) \cup \bigcup_{R\in\ttt} \Omega_R.$$
	So \eqref{eqHH*} holds if we define
	$$\HH := \cB \cup \bigcup_{R\in\ttt} \partial\tree^{**}(R).$$
	
	It remains to prove the packing condition \eqref{eqHH*2}. From the uniform rectifiability of $\pom$,  the family $\cB$ satisfies a Carleson packing condition, and so it suffices to show that
	the same holds for
	$\bigcup_{R\in\ttt} \partial\tree^{**}(R)$. This is an immediate consequence of Lemmas \ref{lemdtree**} 
	and \ref{lempack2}. Indeed, for any $S\in\DD_\sigma$, let $T_0=\{R\in\ttt:\tree^{**}(R)\cap \DD_\sigma(S)\neq \varnothing\}$ and $T_1=\{R\in T_0:\ell(R)\leq \ell(S)\}$, $T_2=\{R\in T_0:\ell(R)> \ell(S)\}$,	so that
	\begin{equation}\nonumber
		\sum_{R\in\ttt}\,\sum_{Q\in\partial\tree^{**}(R)\cap\DD_\sigma(S)} \sigma(Q)  \leq 
		\sum_{R\in T_1}\, \sum_{Q\in\partial\tree^{**}(R)} \sigma(Q)  
		  +\sum_{R\in T_2}\,\sum_{Q\in\partial\tree^{**}(R)\cap\DD_\sigma(S)} \sigma(Q). 
	\end{equation}
	Since all the cubes from $\partial\tree^{**}(R)$ are contained in $C(a^{**}) R$, it follows that the cubes from
	$T_1$ are contained in $C'(a^{**})S$, and thus 
	$$\sum_{R\in T_1}\, \sum_{Q\in\partial\tree^{**}(R)} \sigma(Q) \lesssim_{a^{**}}\sum_{R\in T_1}\sigma(R)
	\lesssim_{\ve,\delta}\sigma(S).$$
	Also, it is immediate to check that the number of cubes from $T_2$ is uniformly bounded by some constant
	depending on $a^{**}$. Therefore, 
	$$\sum_{R\in T_2}\,\sum_{Q\in\partial\tree^{**}(R)\cap\DD_\sigma(S)} \sigma(Q)\lesssim_{a^{**}}
	\sum_{R\in T_2}\sigma(S)\lesssim_{a^{**}}\sigma(S).$$
\end{proof}

\subsection{The properties of $\Omega_R$}\label{subs:cordecOmega}

\begin{lemma}
	Assume that $\ve$, $\delta$, and $\gamma$ are small enough. Then, for each $R\in\ttt$, $\Omega_R$ is a $\tau$-Lipschitz domain or a disjoint union of two $\tau$-Lipschitz domains, with uniform Lipschitz character.
\end{lemma}

\begin{proof}
	This is an easy consequence of the definition of $\Omega_R^\pm$. Indeed, notice that in the bottom part of $B^+_R$ (in a neighborhood of the cylinder $C'(R)$, say), $\partial \Omega_R^+$
	can be written as a Lipschitz graph with slope bounded by $C\max(\delta,\gamma)$ (namely the maximum of the graph defining the bottom part of $\partial B^+(R)$ and the graph defined by $ x_{n+1}= A(x') + \delta \,D_R(x')$).
	The same happens with $\partial \Omega_R^-$.
	We leave the details for the reader.
\end{proof}

\begin{lemma}
	For each $R\in\ttt$, $\Omega_R$ is a Lipschitz domain satisfying the   $\tau$-DKP condition, assuming $\theta,\ve,\delta$ small enough
	and $M$ big enough.
\end{lemma}

Remark that, under the assumptions of the previous lemma,  $\Omega_R$ also satisfies the $C\tau$-DPR condition.

\begin{proof}
	Let $B=B(\xi,r)$ be a ball centered in $\pom_R$ with radius $r\leq\diam(\Omega_R)$.
	Denote by $\WW(\Omega_R)$ the family of Whitney cubes of $\Omega_R$. By standard arguments, it is enough to show that
	\begin{equation}\label{eqguai1}
		\sum_{Q\in\WW(\Omega_R):Q\cap B\neq\varnothing} {\rm s}_{2Q}(A)^2\,\ell(Q)^n \leq \tau\,r^n,
	\end{equation}
	where we denoted ${\rm s}_{2Q}(A) = \sup_{x\in 2Q}|\nabla A(x)|\,\ell(Q)$. 	Suppose first that $r\leq C_7^{-1}\delta_\Omega(\xi)$, for some big constant $C_7>1$ to be chosen momentarily. In this case, for any cube $Q\in\WW(\Omega_R)$ that intersects $B(\xi,r)$ (which implies that $\ell(Q)\lesssim r$), we have $C_7r \leq \dist(\xi,\pom) \leq \dist(2Q,\pom)- \dist(\xi,2Q) \leq \dist(2Q,\pom)- Cr$. 	So $\dist(2Q,\pom)\ge r\gtrsim \ell(Q),$ choosing $C_7$ big enough.
	This implies that $2Q$ is contained in a finite union of Whitney cubes $P\in\WW(\Omega)$ with $\ell(P)\approx\dist(2Q,\pom)$.
	We denote by $I_Q$ the family of such cubes $P$.
	
	Recall now that, by Lemma \ref{lemcontingtree}, $\overline\Omega_R\cap\Omega\subset \bigcup_{Q\in\tree^*(R)} w(Q)$. 	Hence for any $P\in I_Q$ there exists $S\in\tree^*(R)$, with $\ell(S)=\ell(P)$  such that $P\subset w(S)$.
	By the definition of $\tree^*(R)$ and the properties of Whitney cubes, this implies that there exists some cube $S'$ which is $a^*$-close to $S$ such that, for $M$ big enough, depending on $a^*$ and the properties of Whitney cubes,
	\begin{equation}\label{eqinclup*}
		P\subset w(S) \subset B(x_{S'},M\ell(S'))\cap\Omega_{M^{-1}\ell(S')}.
	\end{equation}
	Remark that here we have used the notation in \rf{eqomegaa} and that $M$ is the constant appearing in the definition of the coefficients $\lambda_\Omega(\cdot)$. Observe also that the parameter $a^*$ is
	independent of $M$.
	
	By construction, for any cube $S'\in\tree(R)$ we have that either $\lambda_\Omega(S') \leq \theta$, or $\lambda_\Omega(S'')\leq\theta$, being $S''$ the parent of $S'$. By choosing a larger $M$ if necessary, we can assume that the inclusion in \rf{eqinclup*} also holds replacing $S'$
	by $S''$. Then, from both inclusions and by the definition of $\lambda_\Omega(S')$ and $\lambda_\Omega(S'')$ we infer that $|\nabla A(x)|\leq \ell(S')^{-1}\theta \approx_{a^*} \ell(P)^{-1}\theta\lesssim r^{-1}\theta$, for any $x\in P$. 	Since this holds for all the cubes $P\in I_Q$, the estimate above holds for all $x\in 2Q$. Hence, ${\rm s}_{2Q}(A) \lesssim_{a^*} \theta\,\frac{\ell(Q)}r$, for any   $Q\in\WW(\Omega_R)$ such that $Q\cap B\neq\varnothing$. 	Therefore, denoting by $I_B$ the family of cubes $Q\in\WW(\Omega_R)$ such that $Q\cap B\neq\varnothing$, we have
	\begin{equation}\label{eqosc1}
		\sum_{Q\in I_B} {\rm s}_{2Q}(A)^2\,\ell(Q)^n \lesssim_{a^*} \frac{\theta}r\sum_{Q\in I_B} \ell(Q)^{n+1} \lesssim \frac{\theta}r\,m(B(\xi,Cr)) \lesssim \theta\,r^n,
	\end{equation}
	so that \rf{eqguai1} holds for $\theta$ small enough.

	We consider now the case $r> C_7^{-1}\delta_\Omega(\xi)$. For any cube $Q\in I_B$, let $\xi_Q\in\pom_R$ be such that $\dist(\xi_Q,Q)=\dist(Q,\pom_R)$. Notice that $\xi_Q\in C_8B\cap\pom_R$, for some fixed $C_8>1$.
	Let $I_B^1$ be the family of cubes $Q\in I_B$ such that $\dist(Q,\pom_R)< \frac{C_7^{-1}}5\,\dist(\xi_Q,\pom)$, 	and set $I_B^2 = I_B\setminus I_B^1.$ Denote $U= \bigcup_{y\in C_8B\cap \pom_R} B(y,C_7^{-1}\delta_\Omega(y)/5)$. 	By applying Vitali's covering theorem, there is a family of balls $B_j$, $j\in J$, centered in $y_j\in \pom_R$, with radius $r_j= C_7^{-1}\delta_\Omega(y_j)$, so that the balls $\frac15B_j$ are
	pairwise disjoint and 
	$U\subset \bigcup_{j\in J}B_j$.
	Observe that if $Q\in I_B^1$, then $Q$ intersects the ball $B(\xi_Q, C_7^{-1}\delta_\Omega(\xi_Q)/5)$, and so it also intersects at least one of the balls $B_j$, $j\in J$. In this case, we write $Q\in I_{B_j}$. 	We split
	$$\sum_{Q\in I_B} {\rm s}_{2Q}(A)^2\,\ell(Q)^n \leq \sum_{j\in J}  \sum_{Q\in I_{B_j}} {\rm s}_{2Q}(A)^2\,\ell(Q)^n + 
	\sum_{Q\in  I_B^2} {\rm s}_{2Q}(A)^2\,\ell(Q)^n =: S_1+ S_2.$$
	To deal with $S_1$ we apply the estimate \rf{eqosc1} to each ball $B_j$ and we take into account that all the balls $\frac15B_j$ are contained in $C_9B$, for some $C_9>1$, because $r_j= C_7^{-1}\delta_\Omega(y_j) \lesssim \delta_\Omega(\xi) + |\xi - y_j|\lesssim r$. Then we get
	$$S_1 \lesssim_{a^*} \theta\sum_{j\in J_B} r_j^n \approx \theta\sum_{j\in J_B}\HH^n(\tfrac15B_j\cap\pom_R)\lesssim \theta\,\HH^n(C_9B\cap \pom_R) \lesssim \theta\,r^n.$$
	
	Let us turn our attention to the sum $S_2$. We claim that  $Q\in I_B^2   \Rightarrow \dist(2Q,\pom)\approx_{a^*}\ell(Q)$. 	One inequality follows from the fact that $Q$ is a Whitney cube for $\Omega_R$: $\dist(2Q,\pom) \geq \dist(2Q,\pom_R) \gtrsim\ell(Q)$. 	To see the converse inequality, notice that by the definition of $I_B^2$, $\dist(2Q,\pom)\leq \dist(Q,\pom)\leq \dist(\xi_Q,Q) + \delta_\Omega(\xi_Q) \leq (1+5C_7)\,\dist(Q,\pom_R)\approx\ell(Q)$, 	which proves the claim.
	
	From the claim and the properties of Whitney cubes for $\Omega$, it follows that for each $Q\in I_B^2$ there is subfamily $I_Q\subset 
	\WW(\Omega)$ such that $2Q\subset \bigcup_{P\in I_Q} P$, with $\# I_Q\lesssim 1$, and $\ell(P)\approx \ell(Q)$ for each $P\in I_Q$.
	By the same arguments as above, we know that for every $P\in I_Q$
	there exists $S\in\tree^*(R)$, with $\ell(S)=\ell(P)$  such that $P\subset w(S)$.
	As in \rf{eqinclup*}, by the definition of $\tree^*(R)$, this implies that there exists some cube $S'$ which is $a^*$-close to $S$ such that, for $M$ big enough, depending on $a^*$ and the properties of Whitney cubes, $P\subset w(S) \subset B(x_{S'},M\ell(S'))\cap\Omega_{M^{-1}\ell(S')}$. Further, by choosing a larger $M$ if necessary, we can assume that $S'\in\tree(R)\setminus\sss(R)$.  
	
	Let $I_B'$ be the family of   $S'\in\tree(R)\setminus\sss(R)$ associated in this way to some $Q\in I_B^2$. Then we have
	\begin{equation}\label{eqddd2}
		\sum_{Q\in  I_B^2} {\rm s}_{2Q}(A)^2\,\ell(Q)^n \lesssim \sum_{Q\in  I_B^2} \sum_{P\in I_Q} {\rm s}_{2P}(A)^2\,\ell(P)^n 
		\leq \sum_{S'\in I_B'} \sum_{Q\in  I_B^2}   \sum_{P\in I_Q:P\subset S}{\rm s}_{2P}(A)^2\,\ell(P)^n.
	\end{equation}
	By pigeonholing, it is easy to check that each $S'\in I_B'$ is associated, at most, to a bounded number of cubes from 
	$I_B^2$ (with a bound depending on $a^*$). Therefore, by the definition of $\lambda_\Omega(S')$,
	\begin{multline}\nonumber
		\sum_{Q\in  I_B^2} {\rm s}_{2Q}(A)^2\,\ell(Q)^n  = \sum_{Q\in  I_B^2} \|\nabla A\|_{L^\infty(2Q)}^2\,\ell(Q)^{n+2} 
		\lesssim \sum_{Q\in  I_B^2} \sum_{P\in I_Q} \|\nabla A\|_{L^\infty(P)}^2\,\ell(P)^{n+2} \\
		\leq \sum_{S'\in I_B'} \sum_{Q\in  I_B^2}   \sum_{\substack{P\in I_Q:\\P\subset B(x_{S'},M\ell(S'))\cap\Omega_{M^{-1}\ell(S')}}}\|\nabla A\|_{L^\infty(P)}^2\,\ell(P)^{n+2} 
		  \lesssim_{a^*}
		\sum_{S'\in I_B'} \lambda_\Omega(S) \,\sigma(S).
	\end{multline}
	We will estimate the last sum above using the stopping condition (d) for the corona decomposition.
	
	First notice that there are constants $C,C',C'',C'''>1$, possibly depending on $a^*$, such that if $Q\in I_B^2$, $P\in I_Q$, and $S\in\tree^*(R)$ and $S'\in I_B'$ are associated with $Q$ and $P$ as above, then $S'\subset CS\subset C'P \subset C''Q\subset C'''B$. 	Since $\delta_\Omega(\xi)\leq C_7r$, we deduce that there exists some ball $B'$ centered in $\pom$, with radius $r'\approx_{a^*}r$, such
	that $C'''B\subset B'$. Thus, $\sum_{Q\in  I_B^2} {\rm s}_{2Q}(A)^2\,\ell(Q)^n \lesssim_{a^*}\sum_{S'\in \tree(R)\setminus\sss(R):S'\subset B'} \lambda_\Omega(S) \,\sigma(S)$. 	To deal with the last sum, suppose first that the cubes from $\sss(R)$ cover  $R$  up to a set of zero $\sigma$-measure.
	Then we have
	\begin{multline}\nonumber
		\sum_{\substack{S'\in \tree(R)\setminus\sss(R):\\S'\subset B'}} \lambda_\Omega(S) \,\sigma(S)  =
		\sum_{\substack{S'\in \tree(R)\setminus\sss(R):\\S'\subset B'}} \,\sum_{P\in\sss(R)}\lambda_\Omega(S) \,\sigma(P)\\
		  \leq
		\sum_{P\in\sss(R):P\subset B'} \sigma(P) \sum_{\substack{S'\in \tree(R)\setminus\sss(R):\\P\subset S'\subset R}}\lambda_\Omega(S).
	\end{multline}
	By the condition (d), taking into account that $S'\not\in\sss(R)$, we have  $\sum_{\substack{S'\in \tree(R)\setminus\sss(R):\\P\subset S'\subset R}}\lambda_\Omega(S)\leq \theta$. 	Hence
	\begin{equation}\label{eqquals}
		\sum_{\substack{S'\in \tree(R)\setminus\sss(R):\\S'\subset B'}} \lambda_\Omega(S) \,\sigma(S)\leq \theta\,\sigma(B') \approx \theta\,(r')^n
		\approx_{a^*} \theta\,r^n.
	\end{equation}
	In the case when the cubes from $\sss(R)$ do not cover  $R$, for each $k>1$ we let $\sss_k(R)$ be the subfamily of maximal cubes
	from $\sss(R)\cup\DD_k(R)$. By the same arguments as above, replacing the sum $\sum_{P\in\sss(R)}$ by $\sum_{P\in\sss_k(R)}$ we deduce that 	$$\sum_{\substack{S'\in \tree(R)\setminus\sss(R):\\S'\subset B'\!,\; \ell(S')\geq 2^{-k}\ell(R)}} \lambda_\Omega(S) \,\sigma(S)
	\leq \theta\,r^n,$$
	and letting $k\to\infty$ we infer that \rf{eqquals} holds in any case.
	Together with \rf{eqddd2} and the previous estimate for $S_1$, this shows that \rf{eqguai1} is also satisfied in the case $r> C_7^{-1}\delta_\Omega(\xi)$,  for $\theta$ small enough.
\end{proof}

\section{The almost $L$-elliptic extension}\label{sec.extension}
 
Throughout this section we assume that $\Omega\subset\R^{n+1}$, $n\geq1$, is a bounded open set with uniformly $n$-rectifiable boundary satisfying the interior corkscrew condition, and that $A$ is a real, not necessarily symmetric $(n+1)\times (n+1)$ matrix function in $\Omega$ with measurable coefficients satisfying (\ref{eq.elliptic}) and the   DKP condition in $\Omega$. Recall that we write $L=-\dv A\nabla$.

Let $f:\partial\Omega\to \R$ be a Lipschitz function, so that in particular $f\in W^{1,p}(\pom)$ for any $p\in[1,\infty)$. In this brief section we define the ``almost $L$-elliptic extension'' of $f$ to $\Omega$, following closely \cite[Section 4]{mt22} where the case $L=-\Delta$ is considered. We omit proofs of the properties of the almost $L$-elliptic extension, since these are  essentially the same as in \cite[Section 4]{mt22}.

First, we define the auxiliary extension $\wt f$ of $f$; this is done exactly the same way as in \cite[Section 4]{mt22}. Given a ball $B\subset\R^{n+1}$ centered in $\pom$ and an affine map $\m A:\R^{n+1}\to\R$, we consider the coefficient $$\gamma_{f}(B) :=\inf_{\m A}\left( |\nabla\m A| + \dashint_B \frac{|f-\m A|}{r(B)}\,d\sigma\right),$$
where the infimums are taken over all affine maps $\m A:\R^{n+1}\to\R$. We denote by $\m A_B$ an affine map that minimizes $\gamma_{f}(B)$. Next, given a $\hat c$-Whitney decomposition $\m W(\Omega)$, for each Whitney cube $I\in \WW(\Omega)$ we consider 
a $C^\infty$ bump function $\vphi_I$ supported on $1.1I$ such that the functions $\vphi_I$, $I\in\WW(\Omega)$,
form a partition of unity of $\chi_\Omega$. That is, $\sum_{I\in\WW(\Omega)}\vphi_I = \chi_\Omega$. We define the extension $\wt f:\overline \Omega\to\R$ of $f$ as follows:  
\begin{equation}\label{eq.aux}\nonumber
\wt f|_\pom =f,\qquad\quad\wt f|_\Omega = \sum_{I\in\WW(\Omega)} \vphi_I\,\m A_{2B_{\hat b(I)}}.
\end{equation}
It is clear that $\wt f$ is smooth in $\Omega$.  Here, $\hat b(I)$ is any fixed boundary cube of $I$ satisfying (\ref{eq.bi}), and	$B_{\hat b(I)}$ is the ball concentric with $\hat b(I)$ that contains $\hat b(I)$; see Sections \ref{sec.lattice} and \ref{sec.whitney}. If $f\in\Lip(\partial\Omega)$, then $\wt f\in\Lip(\overline{\Omega})$ with $\Lip(\wt f)\lesssim\Lip(f)$ \cite[Lemma 4.2]{mt22}. Let us also record the following useful fact.
\begin{lemma}[\hspace{-0.1mm}{\cite[Lemma 4.6]{mt22}}]\label{lm.gradf} For each $p\in(1,\infty)$, there exists $C>0$ so that 
\begin{equation}\label{eq.gradf}
|\nabla\wt f(x)|\lesssim m_{CB_{\hat b(I)},\sigma}(\nabla_{H,p}f),\qquad\text{for each }x\in I,~ I\in\m W(\Omega).
\end{equation}
\end{lemma}

Given the corona construction in terms of the family $\ttt$ from Section \ref{sec.corona},
for each $R\in\ttt$ we denote by $v_R$ the solution of the continuous Dirichlet problem for the operator $L$ in $\Omega_R$ with boundary data $\wt f|_{\pom_R}$. We define the function $v=v_f:\overline \Omega\to \R$ by
$$v = \left\{\begin{array}{ll} \wt f & \text{in $\overline\Omega \setminus \bigcup_{R\in\ttt}\Omega_R$,}\\& \\ v_R & \text{in each $\Omega_R$, with $R\in\ttt$,}
\end{array}
\right.
$$
and we call it the {\it almost $L$-elliptic extension of $f$}.

If $f$ is Lipschitz on $\partial\Omega$, then $v$ is continuous on $\overline{\Omega}$, and moreover, $v\in\dt{W}^{1,2}(\Omega)$ with $\Vert\nabla v\Vert_{L^2(\Omega)}\lesssim\Lip(f)m(\Omega)^{1/2}$ (these facts follow by   essentially the same argument that yields \cite[Lemma 4.3]{mt22}). Furthermore, for each $R\in\ttt$, by virtue of the properties of $\Omega_R$ in Section \ref{subs:cordecOmega} and the main result  of Dindo\v{s}, Pipher, and Rule in \cite{dpr17}, we have that $v_R$ is the solution of the Regularity problem $(R)_p^L$ with boundary data $\wt f|_{\partial\Omega}$. Let us be more precise; define
\begin{equation}\label{eq.ntmaxtr1}\nonumber
	\widetilde{\m N}_2^{\Omega_R}(u)(\zeta):=\sup_{x\in\gamma_1^{\Omega_R}(\zeta)}\Big(\dashint_{B(x,\delta_{\Omega_R}(x)/2)}|u|^2\,dm\Big)^{1/2},\qquad\zeta\in\partial\Omega_R,
\end{equation}
and let $\nabla_{t_R}$ be the tangential derivative in $\Omega_R$.

\begin{theorem}[{\cite[Theorem 2.10]{dpr17}}]\label{thm.regr} Fix $p\in(1,\infty)$. For each $R\in\ttt$ and $q\in(1,\max\{2,p\}]$, the problem $(\Reg_q^L)$ is solvable in $\Omega_R$. In particular, if $f\in\Lip(\partial\Omega)$ and $\wt f$, $v_R$ are as above, then
\begin{equation}\label{eq.regr}\nonumber
\Vert\widetilde{\m N}^{\Omega_R}_2(\nabla v_R)\Vert_{L^q(\partial\Omega_R,\m H^n|_{\partial\Omega_R})}\lesssim\Vert\nabla_{t_R}\wt f\Vert_{L^q(\partial\Omega_R,\m H^n|_{\partial\Omega_R})}.
\end{equation}
\end{theorem}

\section{Proof of Proposition \ref{prop.conormal}}\label{sec.reg}
  
In Section \ref{sec.strat}, we saw how to reduce the proof of Theorem \ref{thm.regularity} to proving Proposition \ref{prop.conormal}. With the almost $L$-elliptic extension at hand from Section \ref{sec.extension} and the corona decomposition into Lipschitz subdomains from Section \ref{sec.corona}, we are ready to plunge into the details of the proof of Proposition \ref{prop.conormal}.  Throughout this section we assume that $\Omega\subset\R^{n+1}$, $n\geq2$, is a bounded open set with uniformly $n$-rectifiable
boundary satisfying the corkscrew condition.

\subsection{A n.t.\ maximal function estimate for the almost $L$-elliptic extension}\label{sec.proofext} We will need the following fact regarding the almost $L$-elliptic extension. It tells us that the almost $L$-elliptic extension satisfies a similar estimate to (\ref{eq.regest}).
	
\begin{proposition}\label{prop.ext}  Let $p>1$  and $L=-\dv A\nabla$, where $A$ is a   DKP matrix in $\Omega$. Let $\varphi\in\Lip(\partial\Omega)$ and $v_\varphi$ the almost $L$-elliptic extension of $\varphi$ (see Section \ref{sec.extension}).  Then  we have that
\begin{equation}\label{eq.ext}
\Vert\wt{\m N}_2(\nabla v_\varphi)\Vert_{L^p(\partial\Omega)}\lesssim\Vert\nabla_{H,p}\varphi\Vert_{L^p(\partial\Omega)}.
\end{equation}
\end{proposition}

\noindent\emph{Proof.} Let $p'$ be the H\"older conjugate of $p$. We show that
\begin{equation}\label{eq.nta1}
\Big|\int_\Omega F\nabla v_\varphi\,dm\Big|\lesssim\Vert\n C_2(F)\Vert_{L^{p'}(\partial\Omega)}\Vert\nabla_{H,p}\varphi\Vert_{L^p(\partial\Omega)},\qquad\text{for any }F\in{\bf C}_{2,p'}.
\end{equation}
Once we show that the above estimate holds, then (\ref{eq.ext}) follows immediately by Proposition \ref{prop.duality}. Fix $F\in{\bf C}_{2,p'}$, and denote\footnote{The reason for the choice of the subscript $2$ in $T_2$ will be apparent in Section \ref{sec.proofconormal}.} $T_2:=\big|\int_\Omega F\nabla v_\varphi\,dm\big|$. We split $T_2$ up   as follows:
\begin{equation}\label{eq.break2}\nonumber
	T_2\leq\Big|\int_{\n H}F\nabla\wt\varphi\,dm\Big|+\sum_{R\in\ttt}\Big|\int_{\Omega_R}F\nabla v_R\,dm\Big|=T_{21}+T_{22},
\end{equation}
where $\n H:=\Omega\backslash\cup_{R\in\ttt}\Omega_R$ and $\ttt$, $\Omega_R$ are defined in the corona construction in Section \ref{sec.corona}, and $\wt\varphi$, $v_R$   are defined in Section \ref{sec.extension}.

\subsubsection{Estimate for $T_{21}$}\label{sec.t21} We estimate $T_{21}$ first:
\begin{multline}\label{eq.break3}
	T_{21}\leq\sum_{Q\in\m H}\sum_{I\cap w(Q)\neq\varnothing}\Big(\int_I|F|^2\,dm\Big)^{1/2}\Big(\int_I|\nabla\wt\varphi|^2\,dm\Big)^{1/2}\\  \lesssim \sum_{Q\in\m H}\sum_{I\cap w(Q)\neq\varnothing}\Big(\dashint_{CB_Q\cap\partial\Omega}|\nabla_{H,p}\varphi|\,d\sigma\Big)\int_Im_{2,I}(F)\,dm\\ \lesssim \sum_{Q\in\m H}\Big[\Big(\dashint_{CB_Q\cap\partial\Omega}|\nabla_{H,p}\varphi|\,d\sigma\Big)\frac1{\ell(Q)^n}\int_{w(Q)}m_{2,4B_y}(F)\,dm(y)\Big]\ell(Q)^n\\ \lesssim\int_{\partial\Omega}\sup_{Q\in\m H:\zeta\in Q}\Big\{\Big(\dashint_{CB_Q\cap\partial\Omega}|\nabla_{H,p}\varphi|\,d\sigma\Big)\frac1{\ell(Q)^n}\int_{w(Q)}m_{2,4B_y}(F)\,dm(y)\Big\}\,d\sigma(\zeta)\\ \lesssim\int_{\partial\Omega}\m M_{\sigma}(\nabla_{H,p}\varphi)(\zeta)\n C_2(F)(\zeta)\,d\sigma(\zeta)\lesssim\Vert\n C_2(F)\Vert_{L^{p'}(\partial\Omega)}\Vert\nabla_{H,p}\varphi\Vert_{L^p(\partial\Omega)},
\end{multline}
where in the second line we used Lemma \ref{lm.gradf}, and in the fourth line we used Carleson's Theorem and   that the family $\m H$ satisfies a Carleson packing condition (Lemma \ref{lm.h}).

\subsubsection{Estimate for $T_{22}$}\label{sec.t22}

We now consider $T_{22}$. For each $R\in\ttt$, let $\Ins_R$ be the family of $I\in\m W(\Omega)$ such that $I^*\subset\Omega_R$ and let $\Bdry_R$ consist of all $I\in\m W(\Omega)$ which intersect $\Omega_R$ and do not belong to $\Ins_R$. With this notation, we split $T_{22}$ further:
\begin{multline}\label{eq.break4}
	T_{22}\leq\sum_{R\in\ttt}\Big(\sum_{I\in\Ins_R}\int_{I}|F||\nabla v_R|\,dm +\sum_{I\in\Bdry_R}\int_{I\cap\Omega_R}|F||\nabla v_R|\,dm\Big)\\=\sum_{R\in\ttt}(T_{221,R}+T_{222,R}) = T_{221}+T_{222}.
\end{multline}

\subsubsection{Estimate for $T_{221}$}\label{sec.t221}  By H\"older's inequality and arguing as in (\ref{eq.break3}), it is not hard to see that
\begin{equation}\label{eq.int1}
	T_{221,R}\leq\sum_{I\in\Ins_R}\Big(\dashint_I|\nabla v_R|^2\,dm\Big)^{1/2}\int_Im_{2,4B_x}(F)\,dm.
\end{equation}
Let $\m F_R$ be the family of cubes $Q\in\m D_\sigma$ such that $Q\in b_\Omega(I)$ for some $I\in\Ins_R$. Let
$\sigma_R$ be the $n$-dimensional Hausdorff measure restricted to $\partial\Omega_R$, and define $\m D_{\sigma_R}$ as in Section \ref{sec.lattice}, with the oldest generation $j_0\in\bb Z$ verifying $2^{-j_0}\geq\frac{\diam(\partial\Omega_R)}2$. Let us show that for any $Q\in\m F_R$, there exists $P\in\m D_{\sigma_R}$ such that $\ell(P)=\ell(Q)$ and $\dist(P,Q)\leq K\ell(Q)$. Let $\zeta\in\partial\Omega_R$ verify that $\dist(\zeta,Q)=\dist(Q,\partial\Omega_R)$, and take $P\in\m D_{\sigma_R}$ to be the largest dyadic cube such that $\zeta\in P$ and $\ell(P)\leq\ell(Q)$. On the other hand, since $Q\in b_\Omega(I)$, we have that
\[
\ell(Q)=\ell(I)=\frac{\diam(I)}{\sqrt{n+1}}\leq\frac1{\sqrt{n+1}}\diam(\Omega_R)\leq\frac1{\sqrt{n+1}}\diam(\partial\Omega_R)\leq 2^{-j_0},
\]
where we have used that $\Omega_R$ is a bounded domain. Hence there are dyadic cubes of length $\ell(Q)$ in $\m D_{\sigma_R}$, and therefore we have $\ell(P)=\ell(Q)$. Now, using (\ref{eq.bi}), it is easy to see that $\dist(P,Q)\leq K\ell(Q)$ for $K$ depending only on $n$ and $\hat c$.

Given $Q\in\m F_R$, if $P\in\m D_{\sigma_R}$ satisfies that $\ell(P)=\ell(Q)$ and $\dist(P,Q)\leq K\ell(Q)$, we denote it by $P=L(Q)$ ($P$ is a \emph{lift} of $Q$). It is not hard to see that
\begin{equation}\label{eq.cardp}
	\card\{P:P=L(Q)\}\leq N_0,
\end{equation}
where $N_0$ depends only on $n$, $K$, and the $n$-Ahlfors regularity constant of $\partial\Omega_R$.

We let $\m G_R$ be the family of all $P\in\m D_{\sigma_R}$ such that $P=L(Q)$ for some $Q\in\m F_R$. Furthermore, let
\begin{equation}\label{eq.ntmaxtr}
	\widetilde{\m N}_{\beta,2}^{s,\Omega_R}(h)(\zeta):=\sup_{x\in\gamma_\beta^{\Omega_R}(\zeta)\cap B(\zeta,s)}\Big(\dashint_{B(x,\delta_{\Omega_R}(x)/2)}|h|^2\,dm\Big)^{1/2},\qquad\zeta\in\partial\Omega_R,
\end{equation}
and given $I\in\Ins_R$, let us show that
\begin{equation}\label{eq.int2}
	\Big(\dashint_I|\nabla v_R|^2\,dm\Big)^{1/2}\leq C_1\inf_{\zeta\in P}\widetilde{\m N}_{\beta_0,2}^{C\ell(P),\Omega_R}(\nabla v_R)(\zeta),\qquad\text{for each }P=L(Q), Q\in b_\Omega(I),
\end{equation} 
with $C_1$ and $\beta_0$ depending only on $n$, $\theta$, and $\hat c$,    and $C$ depending only on $\hat c$ and $n$. Since $I\in\Ins_R$, then $I^*=(1+\theta)I\subseteq\Omega_R$, whence $\theta\ell(I)\leq\dist(P,I)\lesssim\ell(I)$. Let $c=2^{-k_0}$, where $k_0\in\bb N$ is the smallest integer with $\frac{c\sqrt{n+1}}{\theta}\leq\frac12$.  Partition $I$ into subcubes $J$ of equal length $\ell(J)=c\ell(I)$. Then  
\begin{equation}\label{eq.detail1}
	\dashint_I|\nabla v_R|^2\,dm=c^{n+1}\sum_J\dashint_J|\nabla v_R|^2\,dm.
\end{equation}
Note that $\card\{J\}=2^{k_0(n+1)}$. It is easy to see that  for each $x\in J$, $J\subseteq B(x,\delta_{\Omega_R}(x)/2)$, and therefore
\begin{equation}\label{eq.detail2}
	\dashint_J|\nabla v_R|^2\,dm\lesssim_{n,\theta,\hat c}\dashint_{B(x,\delta_{\Omega_R}(x)/2)}|\nabla v_R|^2\,dm,\qquad\text{for each }x\in J.
\end{equation}
However, for each $x\in J$ and each $\zeta\in P$, we have that $x\in\gamma_{\beta_0}(\zeta)$, for $\beta_0$ large enough depending only on $n$, $\theta$ and $\hat c$. It follows that 
\begin{equation}\label{eq.detail3}
	\Big(\dashint_{B(x,\delta_{\Omega_R}(x)/2)}|\nabla v_R|^2\,dm\Big)^{1/2}\leq\inf_{\zeta\in P}\widetilde{\m N}_{\beta_0,2}^{C\ell(P),\Omega_R}(\nabla v_R)(\zeta),\qquad\text{for each }x\in J.
\end{equation}
Using (\ref{eq.detail1}), (\ref{eq.detail2}), and (\ref{eq.detail3}), the estimate (\ref{eq.int2}) follows.

Using (\ref{eq.int2}) and the properties of the family $\m G_R$, from (\ref{eq.int1}) we deduce that
\begin{multline}\label{eq.int3}
	T_{221,R}\lesssim\sum_{Q\in\m F_R}\sum_{I\in\Ins_R: Q\in b_\Omega(I)}a_Q\Big(\dashint_{P}\widetilde{\m N}_{\beta,2}^{C\ell(P),\Omega_R}(\nabla v_R)\,d\sigma_R\Big)\big(\inf_{\xi\in Q}\n C_2(F)(\xi)\big)\\ \lesssim \sum_{Q\in\m F_R}a_Q\Big(\dashint_{P}\widetilde{\m N}_{\beta,2}^{C\ell(P),\Omega_R}(\nabla v_R)\,d\sigma_R\Big)\big(\inf_{\xi\in Q}\n C_2(F)(\xi)\big)\\ \leq\Big\{\sum_{Q\in\m F_R}a_Q\Big(\dashint_{P}\widetilde{\m N}_{\beta,2}^{C\ell(P),\Omega_R}(\nabla v_R)\,d\sigma_R\Big)^p\Big\}^{\frac1p}\Big\{\sum_{Q\in\m F_R}a_Q\Big(\dashint_Q\n C_2(F)\,d\sigma\Big)^{p'}\Big\}^{\frac1{p'}}
\end{multline}
where
\begin{equation}\label{eq.aq2}\nonumber
	a_Q:=\frac1{\inf_{\zeta\in Q}\n C_2(F)(\zeta)}\int_{w(Q)}m_{2,B(z,\delta(z)/2)}(F)\,dm(z),
\end{equation}
$P$ is any cube such that $P=L(Q)$, and in the second line we used that there are at most a uniform number of Whitney cubes $I$ with $Q\in b_\Omega(I)$. It was shown in (\ref{eq.packing}) that the family $\{a_Q\}_{Q\in\m D_\sigma}$ satisfies the following Carleson packing condition:
\[
\sum_{Q\in\m D_{\sigma}:Q\subseteq S}a_Q\lesssim\sigma(S),\qquad\text{for any }S\in\m D_\sigma.
\]
Given $P\in\m G_R$,   one can easily prove that $\card\{Q\in\m F_R: P=L(Q)\}\leq C$, where $C$ depends only on $K$, $n$, and the $n$-Ahlfors regularity constant of $\partial\Omega$. For any $P\in\m G_R$, define
\[
\wt a_P:=\sum_{Q\in\m F_R: P=L(Q)}a_Q.
\]
We now claim that $\{\wt a_P\}_{P\in\m G_R}$ satisfies a Carleson packing condition. Given $S\in\m G_R$, fix a cube $Q(S)\in\m F_R$ which satisfies that $S=L(Q(S))$. Fix $P\in\m G_R$ with $P\subseteq S$, and let $Q\in\m F_R$ with $P=L(Q)$. Let us show that there exists $\wt Q\in\m D_\sigma$ with $Q\subseteq\wt Q$, $\ell(\wt Q)=\ell(Q(S))$ and $\dist(\wt Q,Q(S))\lesssim\ell(Q(S))$. Indeed, let $\wt Q\in\m D_\sigma$ be the unique ancestor of $Q$ verifying $\ell(\wt Q)=\ell(Q(S))$. By construction, we have that $Q\subseteq\wt Q$, and finally, note that
\begin{multline}\nonumber
	\dist(\wt Q,Q(S))\leq\dist(Q,Q(S))\leq\dist(Q,S)+\diam S+\dist(S,Q(S))\\ \leq\dist(Q,P)+\ell(S)+K\ell(Q(S))\leq K\ell(P)+\ell(Q(S))+K\ell(Q(S))\\ \leq(2K+1)\ell(Q(S)),
\end{multline}
as desired. Now let $\m F^S$ consist of all cubes $\wt Q\in\m D_\sigma$ such that $\ell(\wt Q)=\ell(Q(S))$ and $\dist(\wt Q,Q(S))\leq(2K+1)\ell(Q(S))$. It is easy to show that
\begin{equation}\label{eq.cards}
	\card\{\wt Q\in\m F^S\}\leq C.
\end{equation}
With these observations, we deduce that
\begin{equation}\label{eq.packing2}
	\sum_{P\in\m G_R:P\subseteq S}\wt a_P=\sum_{P\in\m G_R:P\subseteq S}\sum_{Q\in\m F_R:P=L(Q)}a_Q\lesssim\sum_{\wt Q\in\m F^S}\sum_{Q\subseteq\wt Q}a_Q\lesssim\sum_{\wt Q\in\m F^S}\sigma(\wt Q)  \lesssim\sigma(Q(S))\lesssim\sigma(S),
\end{equation}
where we used (\ref{eq.cardp}), the properties of the family $\m F^S$, the Carleson packing condition of $\{a_Q\}$, (\ref{eq.cards}), and the fact that $\partial\Omega$ is $n$-Ahlfors regular.  This completes the proof of the claim that $\{\wt a_P\}_{P\in\m G_R}$ satisfies a Carleson packing condition.

With this fact at hand, we see that
\begin{multline}\label{eq.int4}
	\sum_{Q\in\m F_R}a_Q\Big(\dashint_{P:P=L(Q)}\widetilde{\m N}_{\beta,2}^{C\ell(P),\Omega_R}(\nabla v_R)\,d\sigma_R\Big)^p \leq\sum_{P\in\m G_R}\Big(\dashint_{P}\widetilde{\m N}_{\beta,2}^{C\ell(P),\Omega_R}(\nabla v_R)\,d\sigma_R\Big)^p\sum_{Q\in\m F_R: P=L(Q)}a_Q\\ =\sum_{P\in\m G_R}\wt a_P\Big(\dashint_{P}\widetilde{\m N}_{\beta,2}^{C\ell(P),\Omega_R}(\nabla v_R)\,d\sigma_R\Big)^p \lesssim\int_{\partial\Omega_R}\Big(\sup_{P\in\m G_R:\zeta\in P}\dashint_{P}\widetilde{\m N}_{\beta,2}^{C\ell(P),\Omega_R}(\nabla v_R)\,d\sigma_R\Big)^p\,d\sigma_R(\zeta)\\ \leq\int_{\partial\Omega_R}\m M_{\sigma_R}(\widetilde{\m N}_{\beta,2}^{\Omega_R}(\nabla v_R))^p \,d\sigma_R \lesssim \Vert\widetilde{\m N}_{\beta,2}^{\Omega_R}(\nabla v_R)\Vert_{L^p(\partial\Omega_R)}^p\lesssim\Vert\nabla_{t_R}\wt\varphi\Vert_{L^p(\partial\Omega_R)}^p,
\end{multline}
where  in the third line we used (\ref{eq.packing2}) and Carleson's theorem, and in the last line we used Theorem \ref{thm.regr} for $L^*$. With (\ref{eq.int3}) and (\ref{eq.int4}), another application of H\"older's inequality for sums yields
\begin{multline}\label{eq.int5}
	\sum_{R\in\ttt}T_{221,R}\lesssim\Big\{\sum_{R\in\ttt}\Vert\nabla_{t_R}\wt\varphi\Vert_{L^p(\partial\Omega_R)}^p\Big\}^{\frac1p}\Big\{\sum_{R\in\ttt}\sum_{Q\in\m F_R}a_Q\Big(\dashint_Q\n C_2(F)\,d\sigma\Big)^{p'}\Big\}^{\frac1{p'}} \\ \lesssim\Big\{\sum_{R\in\ttt}\Vert\nabla_{t_R}\wt\varphi\Vert_{L^p(\partial\Omega_R)}^p\Big\}^{\frac1p}\Big\{\sum_{Q\in\m D_\sigma}a_Q\Big(\dashint_Q\n C_2(F)\,d\sigma\Big)^{p'}\Big\}^{\frac1{p'}}\\ \lesssim \Big\{\sum_{R\in\ttt}\Vert\nabla_{t_R}\wt\varphi\Vert_{L^p(\partial\Omega_R)}^p\Big\}^{\frac1p}\Big\{\int_{\partial\Omega}\Big(\sup_{Q:\zeta\in Q}\dashint_Q\n C_2(F)\,d\sigma\Big)^{p'}\,d\sigma(\zeta)\Big\}^{\frac1{p'}} \\ \lesssim \Big\{\sum_{R\in\ttt}\Vert\nabla_{t_R}\wt\varphi\Vert_{L^p(\partial\Omega_R)}^p\Big\}^{\frac1p}\Vert\m M_\sigma(\n C_2(F))\Vert_{L^{p'}(\partial\Omega)}  \lesssim\Big\{\sum_{R\in\ttt}\Vert\nabla_{t_R}\wt\varphi\Vert_{L^p(\partial\Omega_R)}^p\Big\}^{\frac1p}\Vert\n C_2(F)\Vert_{L^{p'}(\partial\Omega)},
\end{multline}
where in the second line we used that a given cube $Q\in\m D_\sigma$ belongs to at most a uniformly bounded number of the families $\m F_R$, and in the third line we used Carleson's theorem and the Carleson packing condition of $\{a_Q\}$. Finally, since the sets $\partial\Omega_R\cap\partial\Omega$  are pairwise disjoint, and following the argument of \cite[Lemma 4.7]{mt22}, we note that
\begin{multline}\label{eq.int6}
	\sum_{R\in\ttt}\Vert\nabla_{t_R}\wt\varphi\Vert_{L^p(\partial\Omega_R)}^p\lesssim\Vert\nabla_{H,p}\varphi\Vert_{L^p(\partial\Omega)}^p+\sum_{R\in\ttt}\Vert\nabla_{t_R}\wt\varphi\Vert_{L^p(\partial\Omega_R\backslash\partial\Omega)}^p\\ \lesssim\Vert\nabla_{H,p}\varphi\Vert_{L^p(\partial\Omega)}^p+\sum_{Q\in\m H}\ell(Q)^n\Big(\dashint_{CB_Q\cap\partial\Omega}|\nabla_{H,p}\varphi|\,d\sigma\Big)^p\lesssim\Vert\nabla_{H,p}\varphi\Vert_{L^p(\partial\Omega)}^p,
\end{multline}
where once again we used Carleson's theorem. From (\ref{eq.int5}) and (\ref{eq.int6}), we conclude that
\begin{equation}\label{eq.t321}
	T_{221}\lesssim \Vert\n C_2(F)\Vert_{L^{p'}(\partial\Omega)}\Vert\nabla_{H,p}\varphi\Vert_{L^p(\partial\Omega)}.
\end{equation} 

\subsubsection{Estimate for $T_{222}$}\label{sec.t222}   Fix $I\in\Bdry_R$, and by H\"older's inequality, note that
\begin{equation}\label{eq.holder1}
	\int_{I\cap\Omega_R}|F||\nabla v_R|\,dm\leq\Big(\int_I|F|^2\,dm\Big)^{1/2}\Big(\int_{I\cap\Omega_R}|\nabla v_R|^2\,dm\Big)^{1/2}.
\end{equation}
Let $\m W(\Omega_R)$ be a $4$-Whitney decomposition of $\Omega_R$ (see Section \ref{sec.whitney}), and observe that
\begin{multline}\label{eq.break5}
	\int_{I\cap\Omega_R}|\nabla v_R|^2\,dm  \leq 
	\sum_{J\in\m W(\Omega_R): J\cap I\neq\varnothing}\int_J|\nabla v_R|^2\,dm    \lesssim 
	\sum_{J\in\m W(\Omega_R): J\cap I\neq\varnothing}\ell(J)\int_{CJ\cap\partial\Omega_R}\widetilde{\m N}_{\beta,2}^{C\ell(J),\Omega_R}(\nabla v_R)^2\,d\sigma_R\\ \lesssim\sum_{k\in\bb N}\sum_{J\in\m W(\Omega_R): J\cap I\neq\varnothing, \ell(J)=2^{-k+k_0}\ell(I)}\frac{\ell(I)}{2^k}\int_{CJ\cap\partial\Omega_R}\widetilde{\m N}_{\beta,2}^{C\ell(J),\Omega_R}(\nabla v_R)^2\,d\sigma_R \lesssim\ell(I)\int_{\wt CI\cap\partial\Omega_R}\widetilde{\m N}_{\beta,2}^{\wt C\ell(I),\Omega_R}(\nabla v_R)^2\,d\sigma_R 
\end{multline}
where $C$, $\wt C$, $k_0$, and $\beta$ are  uniform constants depending only on the parameters of $\m D_\sigma$, and $\widetilde{\m N}^{\Omega_R}$ was defined in (\ref{eq.ntmaxtr}). In the last estimate of (\ref{eq.break5}), we used that the uniform dilations of Whitney cubes $J\in\m W(\Omega_R)$ of a given generation which intersect $I$ have uniformly bounded overlap. Now, we take $\hat c$ to be small enough that $\wt C\hat c\leq1/2$, and this guarantees that
\begin{equation}\label{eq.far}
	\frac12\dist(I,\partial\Omega)\leq\dist(16\wt CI,\partial\Omega)\leq\dist(I,\partial\Omega).
\end{equation}
Using the localization result for the regularity problem, Theorem \ref{thm.loc}, we obtain that
\begin{equation}\label{eq.locuse}
	\int_{\wt CI\cap\partial\Omega_R}\widetilde{\m N}_{\beta,2}^{\wt C\ell(I),\Omega_R}(\nabla v_R)^2\,d\sigma_R\lesssim\int_{16\wt CI\cap\partial\Omega_R}|\nabla_{t_R}\wt\varphi|^2\,d\sigma_R+\ell(I)^n\Big(\dashint_{4\wt CI\cap\Omega_R}|\nabla v_R|\,dm\Big)^2.
\end{equation}
Next, let $Q\in b_\Omega(I)$, and by (\ref{eq.far}), Lemma \ref{lm.gradf}, and standard  computations, it is easy to show that
\[
|\nabla\wt\varphi(x)|\leq C_1 m_{C_2B_Q,\sigma}(\nabla_{H,p}\varphi),\qquad\text{for each }x\in16\wt CI,
\]
with the constants $C_1,C_2$ depending on $\hat c$. It follows that
\begin{equation}\label{eq.loc3}
	\int_{16\wt CI\cap\partial\Omega_R}|\nabla_{t_R}\wt\varphi|^2\,d\sigma_R\lesssim\ell(I)^n\big(m_{C_2B_Q,\sigma}(\nabla_{H,p}\varphi)\big)^2.
\end{equation}
On the other hand, by essentially the same argument as in (\ref{eq.break5}), we have that
\begin{equation}\label{eq.nt2}
	\dashint_{4\wt CI\cap\Omega_R}|\nabla v_R|\,dm\lesssim\dashint_{\wt C_2I\cap\partial\Omega_R}\widetilde{\m N}_{\beta,1}^{\wt C_2\ell(I),\Omega_R}(\nabla v_R)\,d\sigma_R.
\end{equation}
Putting together (\ref{eq.holder1}), (\ref{eq.break5}), (\ref{eq.locuse}), (\ref{eq.loc3}), and (\ref{eq.nt2}), we see that
\begin{equation}\label{eq.break5.5}
	\int_{I\cap\Omega_R}|F||\nabla v_R|\,dm\lesssim m_{2,I}(F)\ell(I)^{n+1} \Big[m_{C_2B_Q,\sigma}(\nabla_{H,p}\varphi)~+~\dashint_{\wt C_2I\cap\partial\Omega_R}\widetilde{\m N}_{\beta,1}^{\wt C_2\ell(I),\Omega_R}(\nabla v_R)\,d\sigma_R\Big].
\end{equation}
Therefore,
\begin{multline}\label{eq.break6}
	\sum_{I\in\Bdry_R}\int_{I\cap\Omega_R}|F||\nabla v_R|\,dm \lesssim\sum_{I\in\Bdry_R}m_{C_2B_{\hat b_\Omega(I)},\sigma}(\nabla_{H,p}\varphi)\int_Im_{2,4B_x}(F)\,dm(x) \\ + \sum_{I\in\Bdry_R}\Big\{\ell(I)^{\frac n{p'}}\frac1{\ell(I)^n}\int_Im_{2,4B_x}(F)\,dm(x)\Big\}\Big\{\ell(I)^{\frac np}\Big(\dashint_{\wt C_2I\cap\partial\Omega_R}\widetilde{\m N}_{\beta,2}^{\Omega_R}(\nabla v_R)^p\,d\sigma_R\Big)^{\frac1p}\Big\} \\ =: T_{2221,R}+T_{2222,R},
\end{multline}
where $\hat b_\Omega(I)$ is any boundary cube of $I$ (see Section \ref{sec.whitney}). 

Note that the term $\sum_{R\in\ttt}T_{2221,R}$ may be handled exactly the same way as how we handled $T_{21}$ before; as such, we have that
\begin{equation}\label{eq.breaking1}
	\sum_{R\in\ttt}T_{2221,R}\lesssim\Vert\n C_2(F)\Vert_{L^{p'}(\partial\Omega)}\Vert\nabla_{H,p}\varphi\Vert_{L^p(\partial\Omega)}.
\end{equation}
We study now the term $T_{2222,R}$. Note that if $I\in\Bdry_R$ and $Q\in b_{\Omega}(I)$, then $Q\in\m H$. Let $\m H_R$ be the family of $Q\in\m H$ such that $Q\in b_\Omega(I)$ for some $I\in\Bdry_R$.  By H\"older's inequality, we obtain
\begin{multline}\label{eq.breaking2}
	T_{2222,R}\lesssim\Big(\sum_{I\in\Bdry_R}\ell(I)^n\Big[\frac1{\ell(I)^n}\int_Im_{2,4B_x}(F)\,dm(x)\Big]^{p'}\Big)^{1/{p'}}\Big(\int_{\partial\Omega_R}\widetilde{\m N}_{\beta,2}^{\Omega_R}(\nabla v_R)^p\,d\sigma_R\Big)^{1/p}\\ \lesssim \Big(\sum_{Q\in\m H_R}\ell(Q)^n\Big[\frac1{\ell(Q)^n}\int_{w(Q)}m_{2,4B_x}(F)\,dm(x)\Big]^{p'}\Big)^{1/{p'}}\Vert\nabla_{t_R}\wt\varphi\Vert_{L^p(\partial\Omega_R)}
\end{multline}
where in the first line we used that the sets $\wt C_2I$ have uniformly bounded overlap, and in the last line we used the solvability of $\operatorname{(R)}_p$ in the Lipschitz domain $\Omega_R$. Using (\ref{eq.breaking2}) and H\"older's inequality  once again, we see that
\begin{multline}\label{eq.break7}
	\sum_{R\in\ttt}T_{2222,R} \lesssim\Big(\sum_{Q\in\m H}\ell(Q)^n\Big[\frac1{\ell(Q)^n}\int_{w(Q)}m_{2,4B_x}(F)\,dm(x)\Big]^{p'}\Big)^{\frac1{p'}}\Big(\sum_{R\in\ttt}\Vert\nabla_{t_R}\wt\varphi\Vert_{L^p(\partial\Omega_R)}^p\Big)^{\frac1p}\\ \lesssim\Big(\int_{\partial\Omega}\sup_{Q\in\m H:\xi\in Q}\Big[\frac1{\ell(Q)^n}\int_{B(\xi,C\ell(Q))\cap\Omega}m_{2,4B_x}(F)\,dm\Big]^{p'}\,d\sigma(\xi)\Big)^{\frac1{p'}}\Big(\sum_{R\in\ttt}\Vert\nabla_{t_R}\wt\varphi\Vert_{L^p(\partial\Omega_R)}^p\Big)^{\frac1p}\\  \lesssim\Big(\int_{\partial\Omega}\n C_2(F)^{p'}\,d\sigma\Big)^{1/{p'}}\Vert\nabla_{H,p}\varphi\Vert_{L^p(\partial\Omega)},
\end{multline}
where in the second line we used Carleson's theorem and the fact that $\m H$ satisfies a Carleson packing condition, and in the last line we used the definition of $\n C_2$, and (\ref{eq.int6}).  From (\ref{eq.break6}), (\ref{eq.breaking1}), and (\ref{eq.break7}), we conclude that
\begin{equation}\label{eq.ivcalc}
	T_{222}\lesssim\Vert\n C_2(F)\Vert_{L^{p'}(\partial\Omega)}\Vert\nabla_{H,p}\varphi\Vert_{L^p(\partial\Omega)}.
\end{equation}
Putting (\ref{eq.break3}), (\ref{eq.t321}), and (\ref{eq.ivcalc}) together, we have shown the desired estimate (\ref{eq.nta1}).\hfill{$\square$}

\subsection{Proof of Proposition \ref{prop.conormal}}\label{sec.proofconormal}   Let $v_{\varphi}$ be the almost $L^*$-elliptic extension of $\varphi$ defined in Section \ref{sec.extension}. Since $w$ solves $Lw=-\dv F$, it follows that $B[w,\Phi]=B[w,v_\varphi]$.  Hence
\begin{equation}\label{eq.breaki}
|\ell_w(\varphi)|\leq\Big|\int_\Omega A\nabla w\nabla v_{\varphi}\,dm\Big|+\Big|\int_\Omega F\nabla v_\varphi\,dm\Big|=T_1+T_2. 
\end{equation}
From Proposition \ref{prop.duality} and Proposition \ref{prop.ext}, we see that
\begin{equation}\label{eq.t3}
T_2\lesssim\Vert\n C_2(F)\Vert_{L^{p'}(\partial\Omega)}\Vert\nabla_{H,p}\varphi\Vert_{L^p(\partial\Omega)}.
\end{equation}
Hence we need only control $T_1$. Fortunately, several parts of the argument  to control $T_1$ may be reduced to terms considered  in the proof of Proposition \ref{prop.ext}. We split up $T_1$ as follows:
\begin{equation}\label{eq.t1}
	T_1\leq\Big|\int_{\n H}A\nabla w\nabla\wt\varphi\,dm\Big|+\sum_{R\in\ttt}\Big|\int_{\Omega_R}A\nabla w\nabla v_R\,dm\Big|=T_{11}+\sum_{R\in\ttt}T_{12,R}=T_{11}+T_{12},
\end{equation}

\subsubsection{Estimate for $T_{11}$}\label{sec.t11} By the Caccioppoli inequality (Lemma \ref{lm.cacc}), we have that
\begin{multline}\label{eq.t11}
	T_{11}\lesssim\sum_{Q\in\m H}\sum_{I\cap w(Q)\neq\varnothing}\Big(\int_I|\nabla w|^2\,dm\Big)^{\frac12}\Big(\int_I|\nabla\wt\varphi|^2\,dm\Big)^{\frac12}\\ \lesssim\sum_{Q\in\m H}\sum_{I\cap w(Q)\neq\varnothing}\ell(I)^{n}\Big(\dashint_{I^*}|w|^2\,dm\Big)^{\frac12}\Big(\dashint_I|\nabla\wt\varphi|^2\,dm\Big)^{\frac12}+\sum_{Q\in\m H}\sum_{I\cap w(Q)\neq\varnothing}\Big(\int_{I^*}|F|^2\,dm\Big)^{\frac12}\Big(\int_I|\nabla\wt\varphi|^2\,dm\Big)^{\frac12} \\ =T_{111}+T_{112}.
\end{multline}
The term $T_{112}$ is estimated in exactly the same way as the term $T_{21}$, which was controlled in Section \ref{sec.t21}. Thus we obtain that
\begin{equation}\label{eq.t112}
T_{112}\lesssim\Vert\n C_2(F)\Vert_{L^{p'}(\partial\Omega)}\Vert\nabla_{H,p}\varphi\Vert_{L^p(\partial\Omega)}.
\end{equation}
We consider the term $T_{111}$ now. Recall that the truncated non-tangential maximal function was defined in (\ref{eq.truncated}). It is not hard to show that  for any $Q\in\m D_{\sigma}$,
\begin{equation}\label{eq.t111.1}
\Big(\dashint_{I^*}|w|^2\,dm\Big)^{\frac12}\leq C_1\inf_{\xi\in Q}\widetilde{\m N}_{\beta_1,2}^{C_2\ell(Q)}(w)(\xi),\qquad\text{for each }I\in\m W(\Omega) \text{ s.t. }I\cap w(Q)\neq\varnothing,
\end{equation}
where $C_1$ and $C_2$ depend only on $n$, and $\hat c$, and we may take $\beta_1=17$.  Now, using (\ref{eq.t111.1}), Lemma \ref{lm.gradf}, and the fact that $\card\{I:I\cap w(Q)\neq\varnothing\}\leq C$, we obtain that
\begin{multline}\label{eq.t111.2}
T_{111}\lesssim\sum_{Q\in\m H}\Big\{\ell(Q)^{\frac{n}{p'}}\inf_{\xi\in Q}\widetilde{\m N}_{2}^{C_2\ell(Q)}(w)(\xi)\Big\}\Big\{\ell(Q)^{\frac np}\Big(\dashint_{CB_Q\cap\partial\Omega}|\nabla_{H,p}\varphi|\,d\sigma\Big)\Big\}\\ \lesssim\Big(\sum_{Q\in\m H}\ell(Q)^n\Big(\dashint_Q\widetilde{\m N}_2^{C_2\ell(Q)}(w)\,d\sigma\Big)^{p'}\Big)^{\frac1{p'}}\Big(\sum_{Q\in\m H}\ell(Q)^n\Big(\dashint_{CB_Q\cap\partial\Omega}|\nabla_{H,p}\varphi|\,d\sigma\Big)^p\Big)^{\frac1p}\\ \lesssim\Big(\int_{\partial\Omega}\m M_\sigma(\widetilde{\m N}_2(w))^{p'}\,d\sigma\Big)^{\frac1{p'}}\Big(\int_{\partial\Omega}\m M_\sigma(\nabla_{H,p}\varphi)^{p}\,d\sigma\Big)^{\frac1{p}} \lesssim\Vert\widetilde{\m N}_2(w)\Vert_{L^{p'}(\partial\Omega)}\Vert\nabla_{H,p}\varphi\Vert_{L^p(\partial\Omega)}\\ \lesssim\big[\Vert g\Vert_{L^{p'}(\partial\Omega)}+\Vert\n C_2(F)\Vert_{L^{p'}(\partial\Omega)}\big]\Vert\nabla_{H,p}\varphi\Vert_{L^p(\partial\Omega)},
\end{multline}
where in the second line we used H\"older's inequality, in the third line we used Carleson's theorem and the fact that the family $\m H$ satisfies a Carleson packing condition, and in the last line we used (\ref{eq.ntmaxn}). From (\ref{eq.t112}) and (\ref{eq.t111.2}), we conclude that
\begin{equation}\label{eq.t11done}
T_{11}\lesssim\big[\Vert g\Vert_{L^{p'}(\partial\Omega)}+\Vert\n C_2(F)\Vert_{L^{p'}(\partial\Omega)}\big]\Vert\nabla_{H,p}\varphi\Vert_{L^p(\partial\Omega)}.
\end{equation}

\subsubsection{Estimate for $T_{12}$}\label{sec.t12} Write $w=w_g+w_F$, where $w_g$ solves the homogeneous problem with Dirichlet data $g$ on $\partial\Omega$, and $w_F$ solves the Poisson problem with source data $-\dv F$ and $0$ Dirichlet data on the boundary.  With these definitions, note that
\begin{equation}\label{eq.trick}
T_{12,R}\leq\Big|\int_{\Omega_R}A\nabla w_g\nabla v_R\,dm\Big|+\Big|\int_{\Omega_R}A\nabla w_F\nabla v_R\,dm\Big|=T_{121,R}+T_{122,R}.
\end{equation}

\subsubsection{Estimate for $T_{121}:=\sum_{R\in\ttt}T_{121,R}$}\label{sec.t121} Let us see how to control $T_{121}$. Let $\vec{N}$ be the unit outer normal on $\partial\Omega_R$, which is well defined $\sigma_R$-almost everywhere. We will argue similarly as in the proof of \cite[Lemma 5.1]{mt22}, with some small technical differences. For each $k\in\bb N$, let $G_k=G_{R,k}$ be the family of dyadic cubes $I\subset\bb R^{n+1}$ with side length $2^{-k}$  such that $3I\cap\partial\Omega_R\neq\varnothing$, and let
\[
\Omega_k=\Omega_{R,k}:=\Omega_R\backslash\bigcup_{I\in G_k}\overline{I}.
\]
It is easy to see that $\partial\Omega_k\subset\Omega_R$, that $\partial\Omega_k$ is $n$-Ahlfors regular uniformly on $k$ and $R$, and that $\dist(x,\partial\Omega_R)\approx 2^{-k}$ for each $x\in\partial\Omega_k$. Since $\int_{\Omega_R}|A^T\nabla v_R\nabla w_g\,dm|<+\infty$, then by the Lebesgue Dominated Convergence Theorem and the facts that $A^T\in\Lip_{\loc}(\Omega)$ and $L^*v_R=0$ pointwise a.e.\ in $\Omega_k$, we have that 
\begin{equation}\label{eq.limit4}
\int_{\Omega_R}A^T\nabla v_R\nabla w_g\,dm=\lim_{k\ra\infty}\int_{\Omega_k}A^T\nabla v_R\nabla w_g\,dm=\lim_{k\ra\infty}\int_{\Omega_k}\dv((A^T\nabla v_R)w_g)\,dm.
\end{equation}
Seeing as the vector field $(A^T\nabla v_R)w_g$ is continuous in $\overline{\Omega}_k$, we use the divergence theorem  to obtain  
\begin{equation}\label{eq.div}
\int_{\Omega_k}\dv((A^T\nabla v_R)w_g)dm=\int_{\partial\Omega_k}(A^T\nabla v_R)\cdot\vec{N}w_g\,d\m H^n,\qquad\text{for each }k.
\end{equation}

From   (\ref{eq.limit4}), (\ref{eq.div}), and H\"older's inequality, we obtain that
\begin{equation}\nonumber
T_{121,R}\lesssim\Big\{\limsup_{k\ra\infty}\Big(\int_{\partial\Omega_k}|\nabla v_R|^p\,d\m H^n\Big)^{\frac1p}\Big\}\Big\{\limsup_{k\ra\infty}\Big(\int_{\partial\Omega_k}|w_g|^{p'}\,d\m H^n\Big)^{\frac1{p'}}\Big\}=T_{1211,R}\times T_{1212,R},
\end{equation}
and so
\begin{equation}\label{eq.limit6}
\sum_{R\in\ttt}T_{121,R}\lesssim\Big(\sum_{R\in\ttt}T_{1211,R}^p\Big)^{\frac1p}\Big(\sum_{R\in\ttt}T_{1212,R}^{p'}\Big)^{\frac1{p'}}=T_{1211}\times T_{1212}.
\end{equation}
Let us control the factor $T_{1212}$. Since $w_g\in C(\overline{\Omega_R})$ and $\Omega_R$ is bounded, we have that
\begin{equation}\label{eq.tricky4}
T_{1212,R}^{p'}=\int_{\partial\Omega_R}|w_g|^{p'}\,d\m H^n =\Vert g\Vert_{L^{p'}(\partial\Omega_R\cap\partial\Omega)}^{p'}+ \int_{\partial\Omega_R\backslash\partial\Omega}|w_g|^{p'}\,d\m H^n.
\end{equation}
Note that
\begin{multline}\label{eq.tricky5}
\int_{\partial\Omega_R\backslash\partial\Omega}|w_g|^{p'}\,d\m H^n\leq\sum_{I\in\Bdry_R}\int_{I\cap\partial\Omega_R}|w_g|^{p'}\,d\m H^n\\ \lesssim\sum_{I\in\Bdry_R}\ell(I)^n\Vert w_g\Vert_{L^{\infty}(I)}^{p'} \lesssim\sum_{I\in\Bdry_R}\ell(I)^{n}\Big(\dashint_{I^*}|w_g|^2\,dm\Big)^{\frac{p'}2} \lesssim \sum_{Q\in\m H_R}\ell(Q)^n\Big(\dashint_{Q}\widetilde{\m N}_2(w_g)\,d\sigma\Big)^{p'},
\end{multline}
where we have used the interior Moser estimate for $w_g$ and (\ref{eq.t111.1}). From (\ref{eq.tricky4}) and (\ref{eq.tricky5}) and the fact that the sets $\partial\Omega_R\cap\partial\Omega$ are disjoint, we see that
\begin{multline}\label{eq.tricky6}
T_{1212}\lesssim\Vert g\Vert_{L^{p'}(\partial\Omega)}+\Big(\sum_{Q\in\m H}\ell(Q)^n\Big(\dashint_{Q}\widetilde{\m N}_2(w_g)\,d\sigma\Big)^{p'}\Big)^{\frac1{p'}}\\ \lesssim\Vert g\Vert_{L^{p'}(\partial\Omega)}+\Big(\int_{\partial\Omega}\m M_\sigma(\widetilde{\m N}_2(w_g))^{p'}\,d\sigma\Big)^{\frac1{p'}} \lesssim\Vert g\Vert_{L^{p'}(\partial\Omega)}+\Vert\widetilde{\m N}_2(w_g)\Vert_{L^{p'}(\partial\Omega)} \lesssim\Vert g\Vert_{L^{p'}(\partial\Omega)},
\end{multline}
where we  used that  $\m H$ satisfies a Carleson packing condition, Carleson's theorem, and (\ref{eq.ntmaxn}).

We turn to the factor $T_{1211}$ in (\ref{eq.limit6}). Note that $|\nabla A^T|\leq C/\delta_\Omega$ in $\Omega$, and since for each $x\in\Omega_R$ we have that $\delta_{\Omega_R}(x)\leq\delta_\Omega(x)$, then it follows that $|\nabla A^T|\leq C/\delta_{\Omega_R}$ in $\Omega_R$.  For each $J\in\m W(\Omega_R)$, there exists a boundary cube $\hat b_{\Omega_R}(J)\in\m D_{\sigma_R}$ such that $\ell(J)=\ell(\hat b_{\Omega_R}(J))$ and $\dist(J,\hat b_{\Omega_R}(J))\approx\ell(J)$. Let $\m G_1=\m G_{1,R}$ be the family of cubes $P\in\m D_{\sigma_R}$ such that there exists some $J\in\m W(\Omega_R)$ with $P=\hat b_{\Omega_R}(J)$ and $J\cap\partial\Omega_k\neq\varnothing$. Consider the estimate
\begin{multline}\label{eq.limit7}
T_{1211,R}^p\leq\sum_{J\in\m W(\Omega_R):J\cap\partial\Omega_k\neq\varnothing}\int_{J\cap\partial\Omega_k}|\nabla v_R|^p\,d\m H^n\leq\sum_{J\in\m W(\Omega_R):J\cap\partial\Omega_k\neq\varnothing}\Vert\nabla v_R\Vert_{L^{\infty}(J)}^p\ell(J)^n\\ \lesssim\sum_{J\in\m W(\Omega_R):J\cap\partial\Omega_k\neq\varnothing}\Big(\dashint_{J^*}|\nabla v_R|^2\,dm\Big)^{\frac p2}\ell(J)^n \lesssim\sum_{J\in\m W(\Omega_R):J\cap\partial\Omega_k\neq\varnothing}\big(\inf_{\zeta\in\hat b_{\Omega_R}(J)}(\widetilde{\m N}_2^{\Omega_R}(\nabla v_R)(\zeta)\big)^p\ell(J)^n\\ \lesssim\sum_{P\in\m G_1}\big(\inf_{\zeta\in P}\widetilde{\m N}_2^{\Omega_R}(\nabla v_R)(\zeta)\big)^p\ell(P)^n,
\end{multline}
where we used Lemma \ref{lm.gradient}  and a similar estimate to (\ref{eq.int2}). In the proof of \cite[Lemma 5.1]{mt22}, it is shown that the family $\m G_1$ satisfies the Carleson packing condition
\begin{equation}\label{eq.carlesonpacking}
\sum_{P\in\m G_1: P\subset S}\ell(P)^n\lesssim\ell(S)^n.
\end{equation}
Hence, from (\ref{eq.limit7}), (\ref{eq.carlesonpacking}), and Carleson's theorem, we see that
\begin{equation}\label{eq.limit8}
T_{1211,R}^p\lesssim\int_{\partial\Omega_R}\m M_\sigma(\widetilde{\m N}_2^{\Omega_R}(\nabla v_R))^p\,d\m H^n\lesssim\Vert\widetilde{\m N}_2^{\Omega_R}(\nabla v_R)\Vert_{L^p(\partial\Omega_R)}^p\lesssim\Vert\nabla_{t_R}\wt{\varphi}\Vert_{L^p(\partial\Omega_R)},
\end{equation}
where in the last estimate we used Theorem \ref{thm.regr} for $L^*$. Therefore, from (\ref{eq.limit8}) and (\ref{eq.int6}), 
\begin{equation}\label{eq.t1211}
	T_{1211}\lesssim\Big(\sum_{R\in\ttt}\Vert\nabla_{t_R}\wt\varphi\Vert_{L^p(\partial\Omega_R)}^p\Big)^{\frac1p}\lesssim\Vert\nabla_{H,p}\varphi\Vert_{L^p(\partial\Omega)}.
\end{equation}
Combining (\ref{eq.limit6}), (\ref{eq.tricky6}), and (\ref{eq.t1211}), we deduce that 
\begin{equation}\label{eq.t121}
	\sum_{R\in\ttt}T_{121,R}\lesssim\Vert g\Vert_{L^{p'}(\partial\Omega)}\Vert\nabla_{H,p}\varphi\Vert_{L^p(\partial\Omega)}.
\end{equation}

\subsubsection{Estimate for $T_{122}:=\sum_{R\in\ttt}T_{122,R}$}\label{sec.t122} Recall that $T_{122,R}$ has been defined in (\ref{eq.trick}); it remains to bound $T_{122}$. Recall that for each $R\in\ttt$,  we denote by $\Ins_R$ the family of Whitney cubes $I\in\m W(\Omega)$ such that $I^*=(1+\theta)I\subset\Omega_R$, and we denote by $\Bdry_R$ all the cubes $J\in\m W(\Omega)\backslash\Ins_R$ which intersect $\Omega_R$. For each $I\in\Ins_R$, let $\eta_{0,I}\in C_c^{\infty}(I^*)$ be such that $\eta_{0,I}\equiv1$ on $I$ and $|\nabla\eta_{0,I}|\lesssim\frac1{\theta\ell(I)}$. Let $\eta_0:=\sum_{I\in\Ins_R}\eta_{0,I}$, and note that $\eta_0>0$ on $\cup_{I\in\Ins_R}I^*$. For each $I\in\Ins_R$, let $\eta_I:=\frac{\eta_{0,I}}{\eta_0}$, and define $\eta_R:=\sum_{I\in\Ins_R}\eta_I$. Then $\eta_R\in C^\infty(\Omega_R)$, $0\leq\eta_R\leq1$,  $\eta_R\equiv1$ on $\bigcup_{I\in\Ins_R}I$, and  
\begin{equation}\label{eq.etar}
\supp((1-\eta_R){\1}_{\Omega_R})\subset\bigcup_{J\in\Bdry_R}J,\qquad\qquad|\nabla\eta_R|\lesssim\frac1{\theta\ell(J)},~\text{on }J, \text{ for each }J\in\Bdry_R.
\end{equation}
Moreover, $\eta_R\equiv0$ on $\partial\Omega_R\backslash\partial\Omega$. Now note that
\begin{multline}\label{eq.tricky}
\int_{\Omega_R}A\nabla w_F\nabla v_R\,dm=\int_{\Omega_R}A\nabla(w_F\eta_{R})\nabla v_R\,dm+\int_{\Omega_R}A\nabla(w_F(1-\eta_{R}))\nabla v_R\,dm\\ =\int_{\Omega_R}A^T\nabla v_R\nabla(w_F\eta_R)\,dm+\int_{\Omega_R}A\nabla(w_F(1-\eta_R))\nabla v_R\,dm=\int_{\Omega_R}A\nabla(w_F(1-\eta_R))\nabla v_R\,dm,
\end{multline}
where we used that $L^*v_R=0$ in $\Omega_R$ and that $w_F\eta_R\in W_0^{1,2}(\Omega_R)$. Next, we see that
\begin{multline}\label{eq.morebreak}
T_{122,R}\lesssim\sum_{I\in\Bdry_R}\int_{I\cap\Omega_R}|\nabla(w_F(1-\eta_R))||\nabla v_R|\,dm\\ \lesssim\sum_{I\in\Bdry_R}\Big(\int_{I\cap\Omega_R}\big(w_F^2|\nabla\eta_R|^2+|\nabla w_F|^2\big)\,dm\Big)^{\frac12}\Big(\int_{I\cap\Omega_R}|\nabla v_R|^2\,dm\Big)^{\frac12}\\ \lesssim\sum_{I\in\Bdry_R}\frac1{\ell(I)}\Big(\int_{I^*}w_F^2\,dm\Big)^{\frac12}\Big(\int_{I\cap\Omega_R}|\nabla v_R|^2\,dm\Big)^{\frac12}\\+\sum_{I\in\Bdry_R} \Big(\int_{I^*}|F|^2\,dm\Big)^{\frac12}\Big(\int_{I\cap\Omega_R}|\nabla v_R|^2\,dm\Big)^{\frac12}= T_{1221,R}+T_{1222,R},
\end{multline}
where we have used (\ref{eq.etar}), Lemma \ref{lm.cacc}, Lemma \ref{lm.cacc}. Note that the term $\sum_{R\in\ttt}T_{1222,R}$ may be controlled in exactly the same way which we controlled the term $T_{222}$ in Section \ref{sec.t222}. Hence, we have that
\begin{equation}\label{eq.t1222}
\sum_{R\in\ttt}T_{1222,R}\lesssim\Vert\n C_2(F)\Vert_{L^{p'}(\partial\Omega)}\Vert\nabla_{H,p}\varphi\Vert_{L^p(\partial\Omega)}.
\end{equation}
We turn our attention to $T_{1221}:=\sum_{R\in\ttt}T_{1221,R}$. For fixed $R\in\ttt$, we see that
\begin{multline}\label{eq.t1221}
T_{1221,R}\lesssim\sum_{I\in\Bdry_R}\ell(I)^n\big(\inf_{\xi\in \hat b_\Omega(I)}\widetilde{\m N}_2(w_F)(\xi)\big)\Big(\dashint_{\wt CI\cap\partial\Omega_R}\widetilde{\m N}_{\beta,2}^{\wt C\ell(I),\Omega_R}(\nabla v_R)^2\,d\sigma_R\Big)^{\frac12}\\ \lesssim\sum_{I\in\Bdry_R}\ell(I)^n\big(\inf_{\xi\in \hat b_\Omega(I)}\widetilde{\m N}_2(w_F)(\xi)\big)m_{C_2B_{\hat b_\Omega(I)},\sigma}(\nabla_{H,p}\varphi)\\ + \sum_{I\in\Bdry_R}\ell(I)^n\big(\inf_{\xi\in \hat b_\Omega(I)}\widetilde{\m N}_2(w_F)(\xi)\big) \dashint_{\wt C_2I\cap\partial\Omega_R}\widetilde{\m N}_{\beta,1}^{\wt C_2\ell(I),\Omega_R}(\nabla v_R)\,d\sigma_R=T_{12211,R}+T_{12212,R},
\end{multline}
where in the first estimate we used (\ref{eq.t111.1}) and (\ref{eq.break5}), and in the second we used (\ref{eq.locuse}), (\ref{eq.loc3}), and (\ref{eq.nt2}).

We control $T_{12211}:=\sum_{R\in\ttt}T_{12211,R}$ now.  Let $\m H_R$ be the family of $Q\in\m H$ such that $Q\in b_\Omega(I)$ for some $I\in\Bdry_R$. Note that 
\begin{multline}\label{eq.t12211}
\sum_{R\in\ttt}T_{12211,R}\lesssim\sum_{R\in\ttt}\sum_{Q\in\m H_R}\Big\{\ell(Q)^{\frac n{p'}}\inf_{\xi\in Q}\widetilde{\m N}_2(w_F)(\xi)\Big\}\Big\{\ell(Q)^{\frac np}m_{C_2B_Q,\sigma}(\nabla_{H,p}\varphi)\Big\}\\ \lesssim\sum_{Q\in\m H}\Big\{\ell(Q)^{\frac n{p'}}\inf_{\xi\in Q}\widetilde{\m N}_2(w_F)(\xi)\Big\}\Big\{\ell(Q)^{\frac np}m_{C_2B_Q,\sigma}(\nabla_{H,p}\varphi)\Big\}\\ \lesssim\Big(\sum_{Q\in\m H}\ell(Q)^n\Big(\dashint_Q\widetilde{\m N}_2(w_F)\,d\sigma\Big)^{p'}\Big)^{\frac1{p'}}\Big(\sum_{Q\in\m H}\ell(Q)^n\big(m_{C_2B_Q,\sigma}(\nabla_{H,p}\varphi)\big)^{p}\Big)^{\frac1{p}}\\ \lesssim\Big(\int_{\partial\Omega}\m M_\sigma(\widetilde{\m N}_2(w_F))^{p'}\,d\sigma\Big)^{\frac1{p'}}\Big(\int_{\partial\Omega}\m M_\sigma(\nabla_{H,p}\varphi)^{p}\,d\sigma\Big)^{\frac1{p}}\\ \lesssim\Vert\widetilde{\m N}_2(w_F)\Vert_{L^{p'}(\partial\Omega)}\Vert\nabla_{H,p}\varphi\Vert_{L^p(\partial\Omega)}\lesssim\Vert\n C_2(F)\Vert_{L^{p'}(\partial\Omega)}\Vert\nabla_{H,p}\varphi\Vert_{L^p(\partial\Omega)},
\end{multline}
where in the second line we used that any given $Q$ may only belong to a uniformly bounded number of the families $\m H_R$, in the third line we used H\"older's inequality, in the fourth line we used Carleson's theorem and the fact that the family $\m H$ satisfies a Carleson packing condition, and in the last line we used   (\ref{eq.ntmaxn}).

One last term remains to bound, $T_{12212}:=\sum_{R\in\ttt}T_{12212,R}$. Actually, this term is similar to term $T_{2222}$ which was defined in (\ref{eq.break6}). Indeed, note that
\begin{equation}\label{eq.t12212}
T_{12212,R}\leq\sum_{I\in\Bdry_R}\Big\{\ell(I)^{\frac n{p'}}\inf_{\xi\in\hat b_\Omega(I)}\widetilde{\m N}_2(w_F)(\xi)\Big\}\times\Big\{\ell(I)^{\frac np}\dashint_{\wt C_2I\cap\partial\Omega_R}\widetilde{\m N}_{\beta,1}^{\wt C_2\ell(I),\Omega_R}(\nabla v_R)\,d\sigma_R\Big\}.
\end{equation}
We control the first factor of (\ref{eq.t12212}) similarly as (\ref{eq.t12211}), while we control the second factor of (\ref{eq.t12212}) similarly as in  the proof of (\ref{eq.ivcalc}). Putting the arguments together, one arrives at the estimate
\begin{equation}\label{eq.t12212done}
\sum_{R\in\ttt}T_{12212,R}\lesssim\Vert\n C_2(F)\Vert_{L^{p'}(\partial\Omega)}\Vert\nabla_{H,p}\varphi\Vert_{L^p(\partial\Omega)}.
\end{equation}
From the splits (\ref{eq.trick}), (\ref{eq.morebreak}), (\ref{eq.t1221}), and the estimates (\ref{eq.t121}),   (\ref{eq.t1222}),   (\ref{eq.t12211}), and (\ref{eq.t12212done}), it follows 
\begin{equation}\label{eq.t12}
T_{12}\lesssim\big[\Vert g\Vert_{L^{p'}(\partial\Omega)}+\Vert\n C_2(F)\Vert_{L^{p'}(\partial\Omega)}\big]\Vert\nabla_{H,p}\varphi\Vert_{L^p(\partial\Omega)}.
\end{equation}
From (\ref{eq.t1}), (\ref{eq.t11done}), and (\ref{eq.t12}), we conclude that
\begin{equation}\label{eq.t1done}
T_1\lesssim\big[\Vert g\Vert_{L^{p'}(\partial\Omega)}+\Vert\n C_2(F)\Vert_{L^{p'}(\partial\Omega)}\big]\Vert\nabla_{H,p}\varphi\Vert_{L^p(\partial\Omega)}.
\end{equation}

Finally, from (\ref{eq.breaki}), (\ref{eq.t3}) and (\ref{eq.t1done}), the desired estimate (\ref{eq.boundconormal}) follows.\hfill{$\square$}

\appendix

\section{Proofs of auxiliary results}\label{sec.app}

\subsection{Proof of Lemma \ref{lm.changeavg}}\label{sec.proofchangeavg} Assume without loss of generality that $\hat c_1<\hat c_2$, and then in this case, the direction $\lesssim$ in (\ref{eq.changeavg2}) is straightforward. Thus it suffices to prove the direction $\gtrsim$ in (\ref{eq.changeavg2}). Let $B$ be the unit ball centered at the origin in $\bb R^{n+1}$, and let $\{B_j\}_{j=1}^{N}$ be a family of balls of radius $\frac{\hat c_1}{2\hat c_2}<1$ centered at points $e_j\in B$, such that $B\subset\cup_jB_j$ and the $B_j$'s have uniformly bounded overlap (depending only on $n$); we have  that $N$ depends only on $n$ and $\hat c_1/(2\hat c_2)$. For each $x\in\Omega$ and $j=1,\ldots, N$, let $B_j^x$ be the ball of radius $\hat c_1\delta(x)/2$ centered at $x_j:=x+\hat c_2\delta(x)e_j$; clearly we have that $\{B_j^x\}_{j=1}^N$ is an open cover of $B(x,\hat c_2\delta(x))$. Then we see that
\begin{multline}\label{eq.cq}
	\n C_{\hat c_2,q}(H)(\xi)\leq\sup_{r>0}\frac1{r^n}\int_{\Omega\cap B(\xi,r)}\frac1{m(B(x,\hat c_2\delta(x)))^\frac1q}\sum_{j=1}^N\Big(\int_{B_j^x}|H|^q\,dm\Big)^{\frac1q}\,dm(x)\\ \lesssim\sum_{j=1}^N\sup_{r>0}\frac1{r^n}\int_{\Omega\cap B(\xi,r)}\Big(\dashint_{B(x_j,\hat c_1\delta(x_j))}|H|^q\,dm\Big)^{\frac1q}\,dm(x)
\end{multline}
where in the second estimate we used the Minkowski inequality. To continue, we claim that for each $j=1,\ldots,n$ and each $r>0$, we have that 
\begin{equation}\label{eq.claim2}
	\int_{\Omega\cap B(\xi,r)}\Big(\dashint_{B(x_j,\hat c_1\delta(x_j))}|H|^q\,dm\Big)^{\frac1q}\,dm(x)\leq2\int_{\Omega\cap B(\xi,2r)}\Big(\dashint_{B(y,\hat c_1\delta(y))}|H|^q\,dm\Big)^{\frac1q}\,dm(y).
\end{equation}
Using (\ref{eq.claim2}) in (\ref{eq.cq}) immediately yields the estimate $\n C_{\hat c_2,q}(H)(\xi)\lesssim\n C_{\hat c_1,q}(H)(\xi)$, from which the desired bound follows.  Thus it remains only to prove (\ref{eq.claim2}), which we do via a change of variables. First,   define the map $F_j:\bb R^{n+1}\ra\bb R^{n+1}$ as
\[
F_j(x):=\left\{\begin{array}{ll}x+\hat c_2\delta_\Omega(x)e_j,&\text{if }x\in\Omega,\\x,&\text{if }x\in\bb R^{n+1}\backslash\Omega,\end{array}\right.
\]
and the non-negative function $Q:\bb R^{n+1}\ra\bb R$ as
\[
Q(x):=\left\{\begin{array}{ll}m_{q,B(F_j(x),\hat c_1\delta(F_j(x)))}(H),&\text{if }x\in\Omega\cap B(\xi,r),\\0,&\text{if }x\in\bb R^{n+1}\backslash(\Omega\cap B(\xi,r)).\end{array}\right.
\]
Note that $x_j=F_j(x)$ whenever $x\in\Omega$. It is not hard to see that $F_j$ is Lipschitz continuous in $\bb R^{n+1}$, with Lipschitz constant $1+\hat c_2\leq3/2$. Denote by $JF_j$ the Jacobian of the map $F_j$ (see \cite[Definition 3.6]{eg1}); then by the matrix determinant lemma, we have that 
\[
JF_j(x)=1+\hat c_2\nabla\delta_\Omega(x)e_j>\frac12,\qquad\text{for }~m_n\text{-a.e.\ }x\in\Omega,
\]
since $|\nabla\delta_\Omega|\leq1$, $|e_j|<1$, and $\hat c_2\leq1/2$. Next we claim that $F_j$ is injective: Indeed, say that $x_1,x_2\in\bb R^{n+1}$ satisfy that $F_j(x_1)=F_j(x_2)$. If both $x_1,x_2\in\bb R^{n+1}\backslash\Omega$, then it trivially follows that $x_1=x_2$. Suppose now that both $x_1,x_2\in\Omega$. Then $F_j(x_1)=F_j(x_2)$ implies that $|x_1-x_2|=\hat c_2|e_j||\delta_\Omega(x_1)-\delta_\Omega(x_2)|\leq|x_1-x_2|/2$, which implies that $x_1=x_2$. Finally, if $x_1\in\bb R^{n+1}\backslash\Omega$ and $x_2\in\Omega$, then $|x_1-x_2|=\hat c_2|e_j|\delta_\Omega(x_2)\leq|x_1-x_2|/2$, which is a contradiction. Thus $F_j$ is injective.

Let us define $F_j^{-1}:F_j(\Omega\cap B(\xi,r))\ra\Omega\cap B(\xi,r)$ the inverse map in $\Omega\cap B(\xi,r)$, so that $F_j(F_j^{-1}(y))=y$ for each $y\in F_j(\Omega\cap B(\xi,r))$.  We observe that
\begin{multline}\nonumber
	\int_{\Omega\cap B(\xi,r)}\Big(\dashint_{B(x_j,\hat c_1\delta(x_j))}|H|^q\,dm\Big)^{\frac1q}\,dm(x)=\int_{\Omega}Q(x)\,dm(x)<2\int_\Omega Q(x)JF_j(x)\,dm(x)\\=2\int_{\bb R^{n+1}} Q(x)JF_j(x)\,dm(x)  =2\int_{\bb R^{n+1}}\Big[\sum_{x\in F_j^{-1}\{y\}}Q(x)\Big]\,dm(y)=2\int_{F_j(\Omega\cap B(\xi,r))}Q(F_j^{-1}(y))\,dm(y)\\ \leq2\int_{\Omega\cap B(\xi,2r)}\Big(\dashint_{B(y,\hat c_1\delta(y))}|H|^q\,dm\Big)^{\frac1q}\,dm(y),
\end{multline}
where in the third line we used the change of variables for Lipschitz maps \cite[Theorem 3.9]{eg1}, and in the last line we used the definition of $Q$ and the fact that $F_j(\Omega\cap B(\xi,r))\subseteq\Omega\cap B(\xi,2r)$. From these estimates, the claim (\ref{eq.claim2}) follows.\hfill{$\square$}

\subsection{Proof of Proposition \ref{prop.duality}}\label{sec.proofduality} 
To prove the proposition, we will show that the operators $\wt{\m N}$ and $\n C$, as defined in Section \ref{sec.main}, can be   discretized. After this, we will   appeal to the discrete duality result of Hyt\"onen and Ros\'en \cite[Theorem 2.2]{hr13} to finish the proof. 

Before we get started though, let us record a technical lemma, which is Lemma 3.5 in \cite{hr13} when $\Omega=\bb R^{n+1}_+$. This lemma is what essentially allows us to obtain ``change of aperture'' results, and it is needed in the proofs of Lemmas \ref{lm.ns} and \ref{lm.nd}.  Its proof in our case is very similar to that of \cite[Lemma 3.5]{hr13}, and thus we omit it.
\begin{lemma}\label{lm.lc} Let $f,g$ be two non-negative functions on $\partial\Omega$. Assume that there exist constants $c_1,C_2\in(0,\infty)$ such that for every $\xi\in\partial\Omega$, the bound $f(\xi)>\tau$ implies that there exists a Borel set $E=E(\xi)\subset\partial\Omega$ such that $g>c_1\tau$ on $E$, and
	\begin{equation}\label{eq.lc1}
		0<\big(\sup_{\zeta\in E}|\zeta-\xi|\big)^n\leq C_2\sigma(E).
	\end{equation}
	Then there exists $C_3\in(0,\infty)$ such that for any $p\in[1,\infty]$, $\Vert f\Vert_{L^p(\partial\Omega)}\leq C_3\Vert g\Vert_{L^p(\partial\Omega)}$. 
\end{lemma}

\subsubsection{The discrete model of Hyt\"onen and Ros\'en} Let us quickly describe the particular version of Theorem 2.2 of \cite{hr13} which we shall use. Let $\m W$ be a $1/2$-Whitney decomposition of $\Omega$  (see Section \ref{sec.whitney}), and define $w^*(Q)$ as in (\ref{eq.wq2}) for $\theta$ small enough. For fixed $q\in[1,\infty]$, let $(u)=(u_Q)_{Q\in\m D_\sigma}$ be a sequence of functions $u_Q\in L^q(w^*(Q))$, and define the discrete non-tangential maximal function (acting on sequences $(u)$)  by
\[
\m N^s_q(u)(\xi):=\sup_{Q\in\m D_\sigma:\xi\in Q}m(w^*(Q))^{-1/q}\Vert u_Q\Vert_{L^q(w^*(Q))},\qquad\xi\in\partial\Omega.
\]
For a sequence $(H)=(H_Q)_{Q\in\m D_\sigma}$ where $H_Q\in L^{q'}(w^*(Q))$, we define the discrete Carleson functional (acting on sequences $(H)$) by
\[
\n C_{q'}^s(H)(\xi):=\sup_{Q\in\m D_\sigma:\xi\in Q}\frac1{\sigma(Q)}\sum_{R\in\m D_\sigma: R\subseteq Q}m(w^*(R))^{1-\frac1{q'}}\Vert H_R\Vert_{L^{q'}(w^*(R))},\qquad\xi\in\partial\Omega.
\]
Let $p\in[1,\infty)$ and $p'$ its H\"older conjugate. Then (\cite[Theorem 2.2]{hr13})
\begin{multline}\label{eq.hr1}
\sum_{Q\in\m D_\sigma}\Big|\int_{w^*(Q)}u_QH_Q\,dm\Big|\lesssim\Vert\m N_q^s(u)\Vert_{L^p(\partial\Omega)}\Vert\n C_{q'}^s(H)\Vert_{L^{p'}(\partial\Omega)},\quad \text{ for any } u_Q\in L^q(w^*(Q)), H_Q\in L^{q'}(w^*(Q)).
\end{multline}
\begin{equation}\label{eq.hr2}
\Vert\m N_q^s(u)\Vert_{L^p(\partial\Omega)}\lesssim\sup_{\Vert\n C_{q'}^s(H)\Vert_{L^{p'}(\partial\Omega)}=1}\sum_{Q\in\m D_\sigma}\Big|\int_{w^*(Q)}u_QH_Q\,dm\Big|,\qquad u_Q\in L^q(w^*(Q)),
\end{equation}
\begin{equation}\label{eq.hr3}
\Vert\n C_{q'}^s(H)\Vert_{L^{p'}(\partial\Omega)}\lesssim\sup_{\Vert\m N_q^s(u)\Vert_{L^{p}(\partial\Omega)}=1}\sum_{Q\in\m D_\sigma}\Big|\int_{w^*(Q)}u_QH_Q\,dm\Big|,\qquad H_Q\in L^{q'}(w^*(Q)).
\end{equation}

\subsubsection{Passing from functions to sequences and viceversa}\label{sec.passing} In order to use the estimates (\ref{eq.hr1})-(\ref{eq.hr3}) to prove Proposition \ref{prop.duality}, we need to pass from the discrete model to our setting. As a first step, we need a natural way to identify sequences $(u_Q)_{Q\in\m D_\sigma}$ with functions $u\in L^q_{\loc}(\Omega)$.

Given $u\in L^q_{\loc}(\Omega)$, we define a sequence $S(u)=(u_Q)_{Q\in\m D_\sigma}$ as follows: for each $Q\in\m D_\sigma$, let $u_Q:=u|_{w^*(Q)}\in L^q(w^*(Q))$. Given a sequence $(u_Q)_{Q\in\m D_\sigma}$ with $u_Q\in L^q(w^*(Q))$, we construct a function $F(u)\in L^q_{\loc}(\Omega)$ as follows: 
\[
F(u)(x):=\frac1{\sum_{Q\in\m D_\sigma}\1_{w^*(Q)}(x)}\sum_{Q\in\m D_\sigma}u_Q(x)\1_{w^*(Q)}(x),\qquad x\in\Omega.
\]
Note that $1\leq\sum_{Q\in\m D_\sigma}\1_{w^*(Q)}(x)\leq C$ for every $x\in\Omega$. Moreover, for any function $u\in L^q_{\loc}(\Omega)$, we have that $F(S(u))=u$ pointwise in $\Omega$, but   $S\circ F$ is not the identity on sequences.

Let us now see that with these maps, we are able to extend the estimates (\ref{eq.hr1})-(\ref{eq.hr3}) to dyadic versions of $\wt{\m N}$ and $\n C$ acting on functions. Let
\[
\m N_q^du(\xi):=\sup_{Q\in\m D_\sigma:\xi\in Q}m(w^*(Q))^{-1/q}\Vert u\Vert_{L^q(w^*(Q))},\qquad\xi\in\partial\Omega, u\in L^q_{\loc}(\Omega),
\]
\[
\n C_{q'}^dH(\xi):=\sup_{Q\in\m D_\sigma:\xi\in Q}\frac1{\sigma(Q)}\sum_{R\in\m D_\sigma: R\subseteq Q}m(w^*(R))^{1-\frac1{q'}}\Vert H\Vert_{L^{q'}(w^*(R))},\quad~\xi\in\partial\Omega, H\in L^{q'}_{\loc}(\Omega).
\]
It is easy to see then that
\begin{equation}\label{eq.nsg}
\m N_q^sS(u)=\m N_q^du,\qquad\n C_{q'}^sS(H)=\n C_{q'}^dH,
\end{equation}
pointwise on $\partial\Omega$, for any $u\in L^q_{\loc}(\Omega), H\in L^{q'}_{\loc}(\Omega)$. The behavior of $F(u)$ is a bit more subtle in our setting; we record it in the following lemma.

\begin{lemma}\label{lm.ns} Let $p\in[1,\infty)$, $q\in[1,\infty]$, and $p'$,$q'$ the H\"older conjugates. Suppose that either $\Omega$ is bounded, or $\partial\Omega$ is unbounded. Then
\begin{equation}\label{eq.ns}
\Vert\m N_q^dF(u)\Vert_{L^p(\partial\Omega)}\approx\Vert\m N^s_{q}u\Vert_{L^p(\partial\Omega)},\qquad\text{for each }(u)=(u_Q)_{Q\in\m D_\sigma},
\end{equation}
and
\begin{equation}\label{eq.cs}
\Vert\n C_{q'}^dF(H)\Vert_{L^{p'}(\partial\Omega)}\approx\Vert\n C_{q'}^sH\Vert_{L^{p'}(\partial\Omega)},\qquad\text{for each }(H_Q)_{Q\in\m D_\sigma},
\end{equation}
\end{lemma}

\noindent\emph{Proof.} {\bf Proof of (\ref{eq.ns}).} Let us prove the direction $\gtrsim$ first; in fact, we show that
\begin{equation}\label{eq.nsp}
\m N_q^su\lesssim\m N_q^dF(u)\quad\text{pointwise on }\partial\Omega.
\end{equation}
Without loss of generality, we may assume that $u_Q\geq0$ for each $Q\in\m D_\sigma$, and then, by the definition of $F(u)$,
\begin{equation}\label{eq.fu1}
	u_Q(x)\lesssim F(u)(x),\qquad\text{for all }x\in\Omega.
\end{equation} 
Fix $\xi\in\partial\Omega$ and $Q\in\m D_\sigma$ such that $\xi\in Q$. Then
\begin{equation}\label{eq.ns1}
\Big(\int_{w^*(Q)}|u_Q|^q\,dm\Big)^{\frac1q}\lesssim\Big(\int_{w^*(Q)}|F(u)|^q\,dm\Big)^{\frac1q},
\end{equation}
so that (\ref{eq.nsp}) follows. 

We now consider the direction $\lesssim$ in (\ref{eq.ns}). We will show that the hypotheses of Lemma \ref{lm.lc} are verified. Without loss of generality, we may again assume that $u_Q\geq0$ for each $Q\in\m D_\sigma$. Fix $\xi\in\partial\Omega$ and $\tau>0$, and assume that $\m N_q^dF(u)(\xi)>\tau$. Then there exists $Q\in\m D_\sigma$ with $\xi\in Q$ so that
\begin{equation}\label{eq.ns2}
\tau<\Big(\dashint_{w^*(Q)}|F(u)|^q\,dm\Big)^{\frac1q}\lesssim m(w^*(Q))^{-\frac1q}\sum_{\substack{S\in\m D_\sigma\\w^*(S)\cap w^*(Q)\neq\varnothing}}\Big(\int_{w^*(S)}|u_S|^q\,dm\Big)^{\frac1q},
\end{equation}
where we used the Minkowski inequality and the definition of $F$. Let $\m F_Q$ be the family of all $S\in\m D_\sigma$ such that $w^*(S)\cap w^*(Q)\neq\varnothing$. If $S\in\m F_Q$, then by standard geometric arguments, we may see that $\ell(S)\approx\ell(Q)$ and $\dist(S,Q)\lesssim\ell(Q)$; these facts imply that $\card\m F_Q\approx1$ and that $m(w^*(S))\approx m(w^*(Q))$ for all $S\in\m F_Q$. With these observations in mind, from (\ref{eq.ns2}) we deduce that there exists $\wt S\in\m F_Q$ so that
\begin{equation}\label{eq.ns3}\nonumber
\tau\lesssim\Big(\dashint_{w^*(\wt S)}|u_{\wt S}|^q\,dm\Big)^{\frac1q}\leq\inf_{\zeta\in\wt S}\m N_q^su(\zeta).
\end{equation}
In other words, there exists $c_1>0$ so that $\m N_q^su\geq c_1\tau$ on $\wt S$. On the other hand, since
\[
\sup_{\zeta\in\wt S}|\zeta-\xi|\leq\diam S+\dist(S,Q)+\diam Q\lesssim\ell(S)+\ell(Q)\approx\ell(S),
\]
it follows that the hypotheses of Lemma \ref{lm.lc} are satisfied, and we immediately conclude the desired estimate.

{\bf Proof of (\ref{eq.cs}).} The direction $\gtrsim$ is straightforward by applying the estimate (\ref{eq.ns1}). We show the direction $\lesssim$; once again, our strategy is to show that the hypotheses of Lemma \ref{lm.lc} are verified. We assume without loss of generality that $H_Q\geq0$ for each $Q\in\m D_\sigma$. Fix $\xi\in\partial\Omega$ and $\tau>0$, and assume that $\n C_{q'}^dF(H)(\xi)>\tau$. Then there exists $Q\in\m D_\sigma$ with $\xi\in Q$ so that
\begin{multline}\label{eq.cs1}
\tau<\frac1{\sigma(Q)}\sum_{\substack{R\in\m D_\sigma\\R\subseteq Q}}m(w^*(R))^{1-\frac1{q'}}\Big(\int_{w^*(R)}|F(H)|^{q'}\,dm\Big)^{\frac1{q'}}\\ \lesssim\frac1{\sigma(Q)}\sum_{\substack{R\in\m D_\sigma\\R\subseteq Q}}\sum_{S\in\m F_R}m(w^*(S))^{1-\frac1{q'}}\Big(\int_{w^*(S)}|H_S|^{q'}\,dm\Big)^{\frac1{q'}},
\end{multline}
where we used a similar estimate as in (\ref{eq.ns2}), and $\m F_R$ is the family of all $S\in\m D_\sigma$ such that $w^*(S)\cap w^*(R)\neq\varnothing$. If $S\in\m F_R$, then $\ell(S)\leq2\ell(R)\leq2\ell(Q)$, and
\[
\dist(S,Q)\leq\dist(S,R)\lesssim\ell(R)\leq\ell(Q).
\]
For any $S\in\m F_R$ and $R\in\m D_\sigma$ with $R\subseteq Q$, let $Q_S\in\m D_\sigma$ be the unique cube satisfying that $S\subseteq Q_S$ and $\ell(Q_S)=2\ell(Q)$. By the previous remarks,  there is a uniformly bounded number of distinct $Q_S$; more precisely $\cup_R\{Q_S\}_{S\in\m F_R}=\{Q_j\}_{j=1}^N$. Moreover, for any $j=1,\ldots,N$ and $S\in\m D_\sigma$ with $S\subset Q_j$, we have that $\card\{R\in\m D_\sigma:R\subseteq Q, S\in\m F_R\}\lesssim1$. With these facts in mind, from (\ref{eq.cs1}) we obtain that
\begin{multline}
\tau\lesssim\frac1{\sigma(Q)}\sum_{j=1}^N\sum_{\substack{S\in\m D_\sigma\\S\subseteq Q_j}}m(w^*(S))^{1-\frac1{q'}}\Big(\int_{w^*(S)}|H_S|^{q'}\,dm\Big)^{\frac1{q'}}\\ \lesssim\frac1{\sigma(Q_k)}\sum_{\substack{R\in\m D_\sigma\\R\subseteq Q_k}}m(w^*(R))^{1-\frac1{q'}}\Big(\int_{w^*(R)}|H_R|^{q'}\,dm\Big)^{\frac1{q'}}\leq\inf_{\zeta\in Q_k}\n C_{q'}^sH(\zeta),
\end{multline}
for some $k\in\{1,\ldots,N\}$. It is also not hard to see that $\sup_{\zeta\in Q_k}|\zeta-\xi|\lesssim\ell(Q_k)$, so that the hypotheses of Lemma \ref{lm.lc} are satisfied. The desired estimate follows.\hfill{$\square$}

With Lemma \ref{lm.ns} at hand, we are ready to interpret the estimates (\ref{eq.hr1})-(\ref{eq.hr3}) for the dyadic operators acting on functions. The following lemma is an analogue of \cite[Corollary 2.3]{hr13}.

\begin{lemma}\label{lm.dn} Let $p\in[1,\infty)$, $q\in[1,\infty]$ and  $\frac1p+\frac1{p'}=\frac1q+\frac1{q'}=1$. Assume that either $\Omega$ is bounded, or $\partial\Omega$ is unbounded. Then 
\begin{equation}\label{eq.hr4}
\Vert uH\Vert_{L^1(\Omega)}\lesssim\Vert\m N_q^du\Vert_{L^p(\partial\Omega)}\Vert\n C_{q'}^dH\Vert_{L^{p'}(\partial\Omega)},\qquad u\in L^q_{\loc}(\Omega), H\in L^{q'}_{\loc}(\Omega),
\end{equation}
\begin{equation}\label{eq.hr5}
\Vert\m N_q^du\Vert_{L^p(\partial\Omega)}\lesssim\sup_{\Vert\n C_{q'}^dH\Vert_{L^{p'}(\partial\Omega)}=1}\Vert uH\Vert_{L^1(\Omega)},\qquad u\in L^q_{\loc}(\Omega),
\end{equation}
\begin{equation}\label{eq.hr6}
\Vert\n C_{q'}^dH\Vert_{L^{p'}(\partial\Omega)}\lesssim\sup_{\Vert\m N_q^du\Vert_{L^p(\partial\Omega)}=1}\Vert uH\Vert_{L^1(\Omega)},\qquad H\in L^{q'}_{\loc}(\Omega).
\end{equation}
\end{lemma}

\noindent\emph{Proof.} {\bf Proof of (\ref{eq.hr4}).} We may assume that $u\geq0$ and $H\geq0$. Then we see that
\begin{equation}\nonumber
\int_\Omega uH\,dm\approx\sum_{Q\in\m D_\sigma}\int_{w^*(Q)}uH\,dm=\sum_{Q\in\m D_\sigma}\int_{w^*(Q)}u_QH_Q\,dm \lesssim\Vert\m N_q^sS(u)\Vert_{L^p(\partial\Omega)}\Vert\n C_{q'}^sS(H)\Vert_{L^{p'}(\partial\Omega)},
\end{equation}
where  we used (\ref{eq.hr1}). Then (\ref{eq.hr4}) follows from the above bound and (\ref{eq.nsg}).

{\bf Proof of (\ref{eq.hr5}).} Assume without loss of generality that $u\geq0$. We have that
\begin{multline}\nonumber
\Vert\m N_q^du\Vert_{L^p(\partial\Omega)}=\Vert\m N_q^sS(u)\Vert_{L^p(\partial\Omega)} \lesssim\sup_{(H):\Vert\n C_{q'}^s(H)\Vert_{L^{p'}(\partial\Omega)}<\infty}\frac1{\Vert\n C_{q'}^s(H)\Vert_{L^{p'}(\partial\Omega)}}\sum_{Q\in\m D_\sigma}\int_{w^*(Q)}u_QH_Q\,dm\\ \lesssim \sup_{(H):\Vert\n C_{q'}^s(H)\Vert_{L^{p'}(\partial\Omega)}<\infty}\frac1{\Vert\n C_{q'}^dF(H)\Vert_{L^{p'}(\partial\Omega)}}\sum_{Q\in\m D_\sigma}\int_{w^*(Q)}uF(H)\,dm\\ \lesssim\sup_{H:\Vert\n C_{q'}^dH\Vert_{L^{p'}(\partial\Omega)}<\infty}\frac1{\Vert\n C_{q'}^dH\Vert_{L^{p'}(\partial\Omega)}}\sum_{Q\in\m D_\sigma}\int_{w^*(Q)}uH\,dm \lesssim\sup_{\Vert\n C_{q'}^dH\Vert_{L^{p'}(\partial\Omega)}=1}\int_\Omega uH\,dm,
\end{multline}
where first we used (\ref{eq.nsg}), in the second line we used (\ref{eq.hr2}) and the definition of $S(u)$, in the third line we used (\ref{eq.cs}) and (\ref{eq.fu1}) (we have that $H_Q\geq0$ since $u_Q\geq0$), and in the last estimate we used the uniformly bounded overlap of the Whitney regions $w^*(Q)$.

{\bf Proof of (\ref{eq.hr6}).} This estimate is obtained similarly as (\ref{eq.hr5}), but using (\ref{eq.hr3}) and (\ref{eq.ns}) instead of (\ref{eq.hr2}) and (\ref{eq.cs}), respectively. We omit further details.\hfill{$\square$}

\subsubsection{Discretization of $\wt{\m N}$ and $\n C$}  
The following lemma shows that the  dyadic versions of $\wt{\m N}$ and $\n C$ are compatible with the non-dyadic versions that we use throughout this manuscript. In the case that $\Omega$ is the half-space, the lemma is shown in Propositions 3.6 and 3.7 in \cite{hr13}. Our proof has some similarities to that of \cite{hr13}; let us give the details for the benefit of the reader.
\begin{lemma}\label{lm.nd} Let $p\in[1,\infty)$, $q\in[1,\infty]$, and $p'$,$q'$ the H\"older conjugates. Suppose that either $\Omega$ is bounded, or $\partial\Omega$ is unbounded. Then
\begin{equation}\label{eq.nd}
\Vert\wt{\m N}_qu\Vert_{L^p(\partial\Omega)}\approx\Vert\m N^d_{q}u\Vert_{L^p(\partial\Omega)},\qquad\text{for each }u\in L^q_{\loc}(\Omega),
\end{equation}
and
\begin{equation}\label{eq.cd}
\Vert\n C_{q'}H\Vert_{L^{p'}(\partial\Omega)}\approx\Vert\n C_{q'}^dH\Vert_{L^{p'}(\partial\Omega)},\qquad\text{for each }H\in L^{q'}_{\loc}(\Omega).
\end{equation}
\end{lemma}

\noindent\emph{Proof.} {\bf Proof of (\ref{eq.nd}).} Fix $u\in L^q_{\loc}(\Omega)$, and let us show the direction $\gtrsim$ first. In fact, we show  that
\begin{equation}\label{eq.ndp}
\m N_q^du\lesssim\wt{\m N}_{\alpha,\frac12,q}u\qquad\text{pointwise on }\partial\Omega \text{ for large enough }\alpha.
\end{equation}
Fix $\xi\in\partial\Omega$, and let $Q\in\m D_\sigma$ be any cube with $\xi\in Q$. Note that
\begin{equation}\label{eq.nd1}
\Big(\dashint_{w^*(Q)}|u|^q\,dm\Big)^{\frac1q}\approx\Big(\sum_{I\in\m W:Q\in b(I)}\dashint_{I^*}|u|^q\,dm\Big)^{\frac1q}\approx\max_{I\in\m W:Q\in b(I)}~\Big(\dashint_{I^*}|u|^q\,dm\Big)^{\frac1q},
\end{equation} 
where we used the definition of $w^*(Q)$, the fact  that $m(w^*(Q))\approx m(I^*)$ for any $I\in\m W$ with $I\cap w(Q)\neq\varnothing$, that the $I^*$ have uniformly bounded overlap depending only on $n$, and that there exists $N_0\in\bb N$ so that for any $Q\in\m D_\sigma$, $\card\{I\in\m W:Q\in b(I)\}\leq N_0$. Now fix $I\in\m W$ such that $Q\in b(I)$, and let $x\in I$. Then it is not hard to see that for some $\alpha>0$ large enough depending only on $n$, we have that $x\in\gamma_\alpha(\xi)$. Since $I^*\subset B(x,\delta(x)/2)$ and $\ell(I^*)\approx\delta(x)$, it follows that
\begin{equation}\label{eq.nde}
\Big(\dashint_{I^*}|u|^q\,dm\Big)^{\frac1q}\lesssim\Big(\dashint_{B(x,\delta(x)/2)}|u|^q\,dm\Big)^{\frac1q}\leq\sup_{y\in\gamma_\alpha(\xi)}\Big(\dashint_{B(y,\delta(y)/2)}|u|^q\,dm\Big)^{\frac1q}
\end{equation}
and so from this estimate and (\ref{eq.nd1}), the estimate (\ref{eq.ndp}) follows. 

We turn to the direction $\lesssim$ in (\ref{eq.nd}); actually we show that
\begin{equation}\label{eq.ndp2}
\wt{\m N}_{\alpha,c_\theta,q}u\lesssim\m N_q^du\qquad\text{pointwise on }\partial\Omega \text{ for small enough }\alpha,c_\theta.
\end{equation}
Fix $\xi\in\partial\Omega$ and $x\in\gamma_\alpha(\xi)$ for $\alpha\in(0,\frac{255}{257})$. Let $I\in\m W$ be the unique Whitney cube such that $x\in I$, and $Q\in\m D_\sigma$ the unique boundary cube such that $\ell(Q)=\ell(I)$ and $\xi\in Q$. We claim that   $Q\in b(I)$. To see this, from (\ref{eq.bi}) we only need to show that $\dist(I,Q)\leq2^9\diam I$, and indeed,
\[
\dist(I,Q)\leq|x-\xi|\leq(1+\alpha)\delta(x)\leq(1+\alpha)257\diam I<2^9\diam I.
\]
Thus $Q\in b(I)$, and so $B(x,c_\theta\delta(x))\subset I^*\subset w^*(Q)$ for some $c_\theta$ small depending only on $n$ and $\theta$. Since we also have that $m(B(x,c_\theta\delta(x)))\approx_\theta m(I^*)\approx m(w^*(Q))$, we deduce that
\begin{equation}\label{eq.ndp3}
\Big(\dashint_{B(x,c_\theta\delta(x))}|u|^q\,dm\Big)^{\frac1q}\lesssim\Big(\dashint_{w^*(Q)}|u|^q\,dm\Big)^{\frac1q},
\end{equation}
from which (\ref{eq.ndp2}) follows.

Finally, (\ref{eq.nd}) follows from (\ref{eq.ndp}),  (\ref{eq.ndp2}), and Lemma \ref{lm.ntchange}.

{\bf Proof of (\ref{eq.cd}).} Fix $H\in L^{q'}_{\loc}(\Omega)$, and we show the direction $\gtrsim$. We prove that
\begin{equation}\label{eq.cdp}
\n C_{q'}^dH\lesssim\n C_{\frac12,q'}H,\qquad\text{pointwise on }\partial\Omega.
\end{equation}
Fix $\xi\in\partial\Omega$ and $Q\in\m D_\sigma$ with $\xi\in Q$. Given any $R\in\m D_\sigma$ with $R\subseteq Q$, let $I_R$ be any one Whitney cube such that $I_R\cap w(R)\neq\varnothing$. Consider that
\begin{multline}\nonumber
\sum_{R\in\m D_\sigma: R\subseteq Q}m(w^*(R))\Big(\dashint_{w^*(R)}|H|^{q'}\,dm\Big)^{\frac1{q'}}\lesssim\sum_{R\in\m D_\sigma: R\subseteq Q}m(I_R)\Big(\dashint_{w^*(R)}|H|^{q'}\,dm\Big)^{\frac1{q'}}\\ \lesssim\sum_{R\in\m D_\sigma: R\subseteq Q}\int_{I_R}\Big(\dashint_{B(x,\delta(x)/2)}|H|^{q'}\,dm\Big)^{\frac1{q'}}\,dm(x) \lesssim\int_{B(\xi,C\ell(Q))\cap\Omega}\Big(\dashint_{B(x,\delta(x)/2)}|H|^{q'}\,dm\Big)^{\frac1{q'}}\,dm(x),
\end{multline}
where we used that $m(w^*(R))\approx m(I_R)$, and in the second line we used (\ref{eq.nd1}) and (\ref{eq.nde}). From the above bound and the fact that $\sigma(Q)\approx\ell(Q)^n$, the estimate (\ref{eq.cdp}) follows.

We now consider the direction $\lesssim$ in (\ref{eq.cd}). Here, it is not possible to prove a pointwise estimate; instead, we will verify that the hypotheses of Lemma \ref{lm.lc}  are satisfied. Fix $\xi\in\partial\Omega$ and assume that $\n C_{c_\theta,q'}H(\xi)>\tau$, where $c_\theta\in(0,1/2]$ is small.  By definition of $\n C_{c_\theta,q'}H$, there exists $r\in(0,\diam\Omega)$ so that
\begin{multline}\label{eq.cd1}
\tau r^n<\int_{B(\xi,r)\cap\Omega}\Big(\dashint_{B(x,c_\theta\delta(x))}|H|^{q'}\,dm\Big)^{\frac1{q'}}\,dm(x) \leq\sum_{I\in\m W:I\cap B(\xi,r)\neq\varnothing}\int_{I}\Big(\dashint_{B(x,c_\theta\delta(x))}|H|^{q'}\,dm\Big)^{\frac1{q'}}\,dm(x)\\ \lesssim \sum_{R\in\m D_\sigma: w(R)\cap B(\xi,r)\neq\varnothing}\int_{w(R)}\Big(\dashint_{w^*(R)}|H|^{q'}\,dm\Big)^{\frac1{q'}}\,dm(x) \lesssim \sum_{R\in\m D_\sigma: w(R)\cap B(\xi,r)\neq\varnothing}m(w^*(R))^{1-\frac1{q'}}\Vert H\Vert_{L^{q'}(w^*(R))},
\end{multline}
where in the third estimate we used (\ref{eq.ndp3}), the uniformly bounded overlap of the ``Whitney regions'' $w(R)$, and the assumption that either $\Omega$ is bounded or $\partial\Omega$ is unbounded. Let $\m F$ be the family of $R\in\m D_\sigma$ such that $w(R)\cap B(\xi,r)\neq\varnothing$. Then it is easy to see that $\ell(R)\leq r/(32\sqrt{n+1})$ and that $R\subseteq B(\xi,Cr)\cap\partial\Omega$ for any $R\in\m F$. For any $R\in\m F$, let $Q_R\in\m D_\sigma$ be the unique cube such that $R\subseteq Q_R$ and $\frac1{32\sqrt{n+1}}r\leq\ell(Q_R)<\frac1{16\sqrt{n+1}}r$. By the previous observations, we have that $\{Q_R\}_{R\in\m F}=\{Q_j\}_{j=1}^N$ for a universal constant $N\in\bb N$, and $Q_j\subset B(\xi,\tilde Cr)$.  Hence, from (\ref{eq.cd1}) we see that
\begin{multline}\nonumber
\tau\lesssim\frac1{r^n}\sum_{j=1}^N\sum_{R\in\m D_\sigma:R\subseteq Q_j}m(w^*(R))^{1-\frac1{q'}}\Vert H\Vert_{L^{q'}(w^*(R))}\\ \lesssim\frac1{\sigma(Q_k)}\sum_{R\in\m D_\sigma:R\subseteq Q_k}m(w^*(R))^{1-\frac1{q'}}\Vert H\Vert_{L^{q'}(w^*(R))} \leq\inf_{\zeta\in Q_k}(\n C_{q'}^dH)(\zeta),
\end{multline}
for some $k\in\{1,\ldots,N\}$. In other words, we have that $\n C_{q'}^dH\geq c_1\tau$ on $Q_k$. By the definition of $Q_k$, it is straightforward to see that $(\sup_{\zeta\in Q_k}|\zeta-\xi|)^n\leq C_2\sigma(Q_k)$. Consequently, the hypotheses of Lemma \ref{lm.lc} are verified, and therefore it follows that $\Vert\n C_{c_\theta,q'}H\Vert_{L^p(\partial\Omega)}\lesssim\Vert\n C_{q'}^dH\Vert_{L^p(\partial\Omega)}$, as desired.\hfill{$\square$}

\subsubsection{Conclusion} The estimates in Proposition \ref{prop.duality}  follow from the estimates in Lemmas \ref{lm.dn} and \ref{lm.nd}. Finally, the fact that ${\bf N}_{q,p}=({\bf C}_{q',p'})^*$ is proven similarly as in \cite[Theorem 2.4(iii)]{hr13}, using the density of compactly supported functions in ${\bf C}_{q',p'}$ (Lemma \ref{lm.lipschitz} (i)). We omit further details. This finishes the proof of Proposition \ref{prop.duality}.\hfill{$\square$}

\subsection{Proof of Lemma \ref{lm.lipschitz}}\label{sec.prooflipschitz}  

We show (i) first. Fix $H\in{\bf C}_{q,p}$. Let $S$ and $F$ be the maps defined in Section \ref{sec.passing} which pass between sequences and functions. By \cite[Lemma 2.5]{hr13}, we may find a sequence $\{h^k\}_{k\in\bb N}$ with $h^k=(h_Q^k)_{Q\in\m D_\sigma}$, $\Vert\n C_q^sh^k\Vert_{L^p(\partial\Omega)}<\infty$, so that $\Vert\n C_q^s(S(H)-h^k)\Vert_{L^p(\partial\Omega)}\ra0$ as $k\ra\infty$. For each $k\in\bb N$, it is straightforward that $F(h^k)$ is compactly supported in $\Omega$ and that $\{F(h^k)\}_k\subset{\bf C}_{q,p}$.  Since $H=F(S(H))$, by Lemmas \ref{lm.ns} and \ref{lm.nd} we have that
\begin{equation}\nonumber
\Vert\n C_q(H-F(h^k))\Vert_{L^p(\partial\Omega)}\approx\Vert\n C_q^d(F\big(S(H)-h^k\big))\Vert_{L^p(\partial\Omega)} \approx\Vert\n C_q^s(S(H)-h^k)\Vert_{L^p(\partial\Omega)}\longrightarrow0,
\end{equation}
as $k\ra\infty$. Hence $F(h^k)\ra H$ strongly in ${\bf C}_{q,p}$, as desired.

We turn to (ii).  Fix $H\in{\bf C}_{q,p}$ such that $H$ is compactly supported in $\Omega$, and fix a bump function $\phi\in C_c^{\infty}(B(0,1))$ satisfying that $0\leq\phi\leq1$. For each $\tau>0$, let $\phi_\tau(x):=\frac1{\tau^{n+1}}\phi\big(\frac{x}{\tau}\big)$ and then given $\hat c\in(0,1/6]$, for each $\ep\in(0,1)$  define
\[
H_\ep(x):=(H*\phi_{\hat c\delta(x)\ep})(x)=\frac1{(\hat c\delta(x)\ep)^{n+1}}\int_{\bb R^{n+1}}H(x-y)\phi\Big(\frac y{\hat c\delta(x)\ep}\Big)\,dm(y),\qquad x\in\Omega.
\]
For fixed $\ep\in(0,1)$, since the function $\phi_{\hat c\delta(\cdot)\ep}$ is locally Lipschitz in $\Omega$ and $H$ is compactly supported in $\Omega$,   it follows that $H_\ep$ is Lipschitz continuous and compactly supported in $\Omega$. We will show that $H_\ep\ra H$ in ${\bf C}_{q,p}$ as $\ep\ra0$. First, let us prove that 
\begin{equation}\label{eq.lip1}
	\sup_{\ep\in(0,1)}\Big(\dashint_{B(x,\hat c\delta(x))}|H_\ep|^q\,dm\Big)^{\frac1q}\lesssim\Big(\dashint_{B(x,3\hat c\delta(x))}|H|^q\,dm\Big)^{\frac1q},\qquad\text{for each }x\in\Omega.
\end{equation}
Fix $x\in\Omega$, and note that
\begin{multline}
	\dashint_{B(x,\hat c\delta(x))}|H_\ep(y)|^q\,dm(y)=|B(0,1)|^q\dashint_{B(x,\hat c\delta(x))}\Big|\dashint_{B(y,\hat c\delta(y)\ep)}H(z)\phi\Big(\frac{y-z}{\hat c\delta(y)\ep}\Big)\,dm(z)\Big|^q\,dm(y) \\ \leq|B(0,1)|^q\dashint_{B(x,\hat c\delta(x))} \dashint_{B(y,\hat c\delta(y)\ep)}|H(z)|^q\,dm(z)\,dm(y) \leq2^{n+1}3^{(n+1)q}|B(0,1)|^q\dashint_{B(x,3\hat c\delta(x))}|H(z)|^q\,dm(z),
\end{multline}
where we used Fubini's theorem and the fact that $\delta(x)\leq 2\delta(y)$ whenever $y\in B(x,\hat c\delta(x))$.  

We now show that $\n C_{\hat c,q}(H_\ep-H)\ra0$ pointwise $\sigma$-almost everywhere in $\partial\Omega$.    Since $H\in{\bf C}_{q,p}$, we have that $\n C_q(H)\in L^p(\partial\Omega,\sigma)$, whence Lemma \ref{lm.changeavg} implies that $\n C_{3\hat c,q}(H)$ is finite on a set $S\subseteq\partial\Omega$ with $\sigma(\partial\Omega\backslash S)=0$. Now fix $\xi\in S$, and let us prove that $\n C_{\hat c,q}(H_\ep-H)(\xi)\ra0$ as $\ep\ra0$.

Since $H\in L^q(\Omega)$, then by the properties of convolution, we have that for any bounded open set $K$ compactly contained in $\Omega$, $\Vert H_\ep-H\Vert_{L^q(K)}\ra0$ as $\ep\ra0$. In particular, for any $x\in\Omega$, we have that
\[
\Big(\dashint_{B(x,\hat c\delta(x))}|H_\ep-H|^q\,dm\Big)^{\frac1q}\ra0,\qquad\text{as }\ep\ra0.
\]
In other words, the sequence $\{m_{q,B(\cdot,\hat c\delta(\cdot))}(H_\ep-H)\}_{\ep>0}$ converges pointwise to $0$ in $\Omega$. On the other hand, by (\ref{eq.lip1}) we have that 
\[
m_{q,B(\cdot,\hat c\delta(\cdot))}(H_\ep-H)\lesssim m_{q,B(\cdot,3\hat c\delta(\cdot))}(H)\in L^1(\Omega\cap B(\xi,r)),\qquad\text{for all }r>0,
\]
since $\xi\in S$. By the Lebesgue Dominated Convergence Theorem, it follows that 
\[
Q_r(H_\ep-H):=\int_{\Omega\cap B(\xi,r)}\Big(\dashint_{B(x,\hat c\delta(x))}|H_\ep-H|^q\,dm\Big)^{\frac1q}\,dm(x)\longrightarrow0,\qquad \text{ for any }r>0.
\]
Next, since $H$ is compactly supported in $\Omega$, there exists $s>0$ so that   $\n C_{\hat c,q}(H_\ep-H)(\xi)=\sup_{r\in[s,1/s]}\frac1{r^n}Q_r(H_\ep-H)$, for any $\ep>0$. We thus see that  $\n C_{\hat c,q}(H_\ep-H)(\xi)\leq\frac1{s^n}Q_{\frac1s}(H_\ep-H)\longrightarrow0$, as  $\ep\ra0$.  Since $\xi\in S$ was arbitrary, it follows that $\n C_{\hat c,q}(H_\ep-H)\ra0$ pointwise $\sigma$-almost everywhere in $\partial\Omega$. Furthermore, by (\ref{eq.lip1}) we have that $\n C_{\hat c,q}(H_\ep-H)\lesssim\n C_{\hat c,q}(H)\in L^p(\partial\Omega)$, so that by the Lebesgue Dominated Convergence Theorem we conclude that $\n C_{\hat c,q}(H_\ep-H)\ra0$ in $L^p(\partial\Omega)$. This finishes the proof of (ii).\hfill{$\square$}

\subsection{Proof of Proposition \ref{prop.wrhp}}\label{sec.proofwrhp}

The direction (b)$\implies$(a) is shown in \cite{hl2018}, while the argument for  (b)$\iff$(c) is the same as that for \cite[Theorem 9.2 (b)$\iff$(c)]{mt22}. Thus we only need to prove   (a)$\implies$(b). Fix $\xi\in\partial\Omega$, $r\in(0,\diam\partial\Omega)$, and $x\in\Omega\backslash2B$. Write $B=B(\xi,r)$. First, let us show that for any $f\in C_c(B\cap\partial\Omega)$ with $f\geq0$, we have that
\begin{equation}\label{eq.wrhp1}
	\Big|\int_Bf\,d\omega^x\Big|\lesssim\frac1{\sigma(B)^{1/p'}}\Vert f\Vert_{L^{p'}(B)}.
\end{equation}
Let $u(y):=\int_Bf\,d\omega^y$ for each $y\in\Omega$, so that $u$ solves the continuous Dirichlet problem for $L$ in $\Omega$ with data $f$ on $\partial\Omega$. Since $f\geq0$, then $u\geq0$ in $\Omega$ by the maximum principle. Let $\m W$ be a $1/2$-Whitney decomposition of $\Omega$ (see Section \ref{sec.whitney}), and let $I\in\m W$ be the unique Whitney cube such that $x\in I$. Let $Q\in b_\Omega(I)$, and note that $\ell(Q)=\ell(I)\approx\delta(x)$.

Fix a large constant $M\geq1$ to be specified later, and we distinguish two cases: either $\delta(x)\geq r/M$ or $\delta(x)<r/M$. If $\delta(x)\geq r/M$, then note that
\begin{equation}\label{eq.wrhp2}
	u(x)\leq\inf_{\zeta\in Q}\m N_\alpha(u)(\zeta)\leq\Big(\dashint_Q|\m N_\alpha(u)|^{p'}\,d\sigma\Big)^{\frac1{p'}}\lesssim\frac1{\sigma(Q)^{1/p'}}\Vert f\Vert_{L^{p'}(B)}\lesssim_M\frac1{\sigma(B)^{1/p'}}\Vert f\Vert_{L^{p'}(B)},
\end{equation}
where we have used an aperture $\alpha$ large enough, H\"older's inequality, the fact that $(\Di_{p'}^L)$ is solvable in $\Omega$, the $n$-Ahlfors regularity of $\partial\Omega$, and the fact that $\ell(Q)\gtrsim\delta(x)\geq r/M$.

If instead we have that $\delta(x)<r/M$, let $\hat Q$ be the unique ancestor of $Q$ such that $r/M\leq\ell(\hat Q)<2r/M$, and $K=2^{10}\sqrt{n+1}$. Then $x\in KB_{\hat Q}$, and if $M\geq8K$, then we have that $KB_{\hat Q}\cap B=\varnothing$, since $x\in\Omega\backslash2B$. Using the boundary H\"older continuity, Lemma \ref{lm.boundaryholder}, on the ball $KB_{\hat Q}$ for the solution $u$, we see that
\begin{multline}\label{eq.wrhp3}
	u(x)\lesssim\Big(\frac{\delta(x)}{K\ell(\hat Q)}\Big)^{\eta}\frac1{|2KB_{\hat Q}|}\int_{2KB_{\hat Q}\cap\Omega}u\,dm\lesssim\dashint_{\wt KB_{\hat Q}}\m N(u)\,d\sigma\\ \lesssim\Big(\dashint_{\wt KB_{\hat Q}}|\m N(u)|^{p'}\,d\sigma\Big)^{\frac1{p'}}   \lesssim\frac1{\sigma(\wt KB_{\hat Q})^{1/p'}}\Vert f\Vert_{L^{p'}(B)}\approx\frac1{\sigma(B)^{1/p'}}\Vert f\Vert_{L^{p'}(B)},
\end{multline}
where $\wt K\approx K$. From (\ref{eq.wrhp2}) and (\ref{eq.wrhp3}), the estimate (\ref{eq.wrhp1}) follows for any $x\in\Omega\backslash2B$. 

Although we obtained (\ref{eq.wrhp1}) for $f\geq0$, we can obtain the same estimate for general $f\in C_c(B)$ by simply splitting $f$ into the non-negative and negative parts. Let us see how to obtain (\ref{eq.wrhp}) from (\ref{eq.wrhp1}). First, for any $x\in\Omega\cap\partial(2B)$, by Lemma \ref{lm.bourgain} we have that $\omega^x(8B)\approx1$, and hence (\ref{eq.wrhp1}) implies that
\begin{equation}\label{eq.wrhp5}
	|u(x)|\leq C\frac{1}{\sigma(B)^{1/p'}}\Vert f\Vert_{L^{p'}(B)}\omega^x(8B),\qquad\text{for all }x\in\Omega\cap\partial(2B).
\end{equation}
Let
\[
v(y):=C\frac{1}{\sigma(B)^{1/p'}}\Vert f\Vert_{L^{p'}(B)}\omega^y(8B)-u(y),\qquad\text{for any }y\in\Omega\backslash2B,
\]
and note that $v\geq0$ on $\partial(\Omega\backslash2B)$ due to (\ref{eq.wrhp5}), the fact that $u\equiv0$ on $\partial\Omega\backslash B$, and the fact that $\omega^y(8B)\geq0$ on $\overline\Omega$. Moreover, by linearity we have that $Lv=0$ in $\Omega\backslash2B$. Then, by the maximum principle (see, for example, \cite[Chapter 6]{hkm93}), we have that $v\geq0$ on $\Omega\backslash2B$. In other words, we have the bound
\begin{equation}\label{eq.wrhp6}
	\Big|\int_Bf\,d\omega^x\Big|\lesssim\frac{\omega^x(8B)}{\sigma(B)^{1/p'}}\Vert f\Vert_{L^{p'}(B)},\qquad\text{for each }x\in\Omega\backslash2B~\text{ and }f\in C_c(B).
\end{equation}

From (\ref{eq.wrhp6}), since $C_c(B)$ is dense in $L^{p'}(B)$, it follows that the   functional $\ell_{\omega^x}(f):=\int_Bf\,d\omega^x$ is bounded on $L^{p'}(B)$; hence there exists $g\in L^p(B)$ so that $\ell_{\omega^x}(f)=\int_Bfg\,d\sigma$ for every $f\in L^{p'}(B)$. It follows immediately that $\omega^x\ll\sigma$ for any $x\in\Omega$, that $d\omega^x/d\sigma=g$ $\sigma$-a.e.\ in $B$, and the estimate (\ref{eq.wrhp}) is obtained directly from (\ref{eq.wrhp6}).\hfill{$\square$}

\bibliographystyle{alpha-sort-max} 
\bibliography{refs} 

\begin{thebibliography}{AHMMMTV16}

\bibitem[AHMT23]{ahmt}
M.~Akman, S.~Hofmann, J.~M. Martell, and T.~Toro.
\newblock Perturbation of elliptic operators in 1-sided {NTA} domains
  satisfying the capacity density condition.
\newblock {\em Forum Math.}, 35(1):245--295, 2023.

\bibitem[ADFJM19]{adfjm}
D.~N. Arnold, G.~David, M.~Filoche, D.~Jerison, and S.~Mayboroda.
\newblock Localization of eigenfunctions via an effective potential.
\newblock {\em Comm. Partial Differential Equations}, 44(11):1186--1216, 2019.

\bibitem[AA11]{aa11}
P.~Auscher and A.~Axelsson.
\newblock Weighted maximal regularity estimates and solvability of non-smooth
  elliptic systems {I}.
\newblock {\em Invent. Math.}, 184(1):47--115, 2011.

\bibitem[AHLMT02]{ahlmct}
P.~Auscher, S.~Hofmann, M.~Lacey, A.~McIntosh, and P.~Tchamitchian.
\newblock The solution of the {K}ato square root problem for second order
  elliptic operators on {${\mathbb R}^n$}.
\newblock {\em Ann. of Math. (2)}, 156(2):633--654, 2002.

\bibitem[AM14]{am14}
P.~Auscher and M.~Mourgoglou.
\newblock Boundary layers, {R}ellich estimates and extrapolation of solvability
  for elliptic systems.
\newblock {\em Proc. Lond. Math. Soc. (3)}, 109(2):446--482, 2014.

\bibitem[AR12]{ar12}
P.~Auscher and A.~Ros\'{e}n.
\newblock Weighted maximal regularity estimates and solvability of nonsmooth
  elliptic systems, {II}.
\newblock {\em Anal. PDE}, 5(5):983--1061, 2012.

\bibitem[Azz]{azz19}
J.~Azzam.
\newblock {H}armonic {M}easure and the {A}nalyst's {T}raveling {S}alesman
  {T}heorem.
\newblock Preprint. November 2019.

\bibitem[AGMT23]{agmt22}
J.~Azzam, J.~Garnett, M.~Mourgoglou, and X.~Tolsa.
\newblock Uniform rectifiability, elliptic measure, square functions, and
  {$\varepsilon$}-approximability via an {ACF} monotonicity formula.
\newblock {\em Int. Math. Res. Not. IMRN}, (13):10837--10941, 2023.

\bibitem[AHMMMTV16]{ahmmmtv}
J.~Azzam, S.~Hofmann, J.~M. Martell, S.~Mayboroda, M.~Mourgoglou, X.~Tolsa, and
  A.~Volberg.
\newblock Rectifiability of harmonic measure.
\newblock {\em Geom. Funct. Anal.}, 26(3):703--728, 2016.

\bibitem[AHMMT20]{ahmmt}
J.~Azzam, S.~Hofmann, J.~M. Martell, M.~Mourgoglou, and X.~Tolsa.
\newblock Harmonic measure and quantitative connectivity: geometric
  characterization of the {$L^p$}-solvability of the {D}irichlet problem.
\newblock {\em Invent. Math.}, 222(3):881--993, 2020.

\bibitem[AHMNT17]{ahmnt}
J.~Azzam, S.~Hofmann, J.~M. Martell, K.~Nystr\"{o}m, and T.~Toro.
\newblock A new characterization of chord-arc domains.
\newblock {\em J. Eur. Math. Soc. (JEMS)}, 19(4):967--981, 2017.

\bibitem[BC94]{bc94}
R.~{Ba\~{n}uelos} and T.~Carroll.
\newblock Brownian motion and the fundamental frequency of a drum.
\newblock {\em Duke Math. J.}, 75(3):575--602, 1994.

\bibitem[BP99]{bp99}
R.~{Ba\~{n}uelos} and M.~M.~H. Pang.
\newblock Lower bound gradient estimates for solutions of {S}chr\"{o}dinger
  equations and heat kernels.
\newblock {\em Comm. Partial Differential Equations}, 24(3-4):499--543, 1999.

\bibitem[Bar21]{barton21}
A.~Barton.
\newblock The {$\dot W^{-1,p}$} {N}eumann problem for higher order elliptic
  equations.
\newblock {\em Comm. Partial Differential Equations}, 46(7):1195--1245, 2021.

\bibitem[BM16]{bm16}
A.~Barton and S.~Mayboroda.
\newblock Layer potentials and boundary-value problems for second order
  elliptic operators with data in {B}esov spaces.
\newblock {\em Mem. Amer. Math. Soc.}, 243(1149):v+110, 2016.

\bibitem[BHLGMP22]{bhlmp}
S.~Bortz, S.~Hofmann, J.~L. Luna~Garc\'ia, S.~Mayboroda, and B.~Poggi.
\newblock Critical perturbations for second-order elliptic operators, {I}:
  square function bounds for layer potentials.
\newblock {\em Anal. PDE}, 15(5):1215--1286, 2022.

\bibitem[CFMS81]{cfms}
L.~Caffarelli, E.~Fabes, S.~Mortola, and S.~Salsa.
\newblock Boundary behavior of nonnegative solutions of elliptic operators in
  divergence form.
\newblock {\em Indiana Univ. Math. J.}, 30(4):621--640, 1981.

\bibitem[CFK81]{cfk}
L.~A. Caffarelli, E.~B. Fabes, and C.~E. Kenig.
\newblock Completely singular elliptic-harmonic measures.
\newblock {\em Indiana Univ. Math. J.}, 30(6):917--924, 1981.

\bibitem[CHPM24]{chm22}
M.~Cao, P.~Hidalgo-Palencia, and J.~M. Martell.
\newblock Carleson measure estimates, corona decompositions, and perturbation
  of elliptic operators without connectivity.
\newblock {\em Math. Ann.}, 390(1):95--156, 2024.

\bibitem[Car62]{carleson62}
L.~Carleson.
\newblock Interpolations by bounded analytic functions and the corona problem.
\newblock {\em Ann. of Math. (2)}, 76:547--559, 1962.

\bibitem[CHM19]{chm}
J.~Cavero, S.~Hofmann, and J.~M. Martell.
\newblock Perturbations of elliptic operators in 1-sided chord-arc domains.
  {P}art {I}: {S}mall and large perturbation for symmetric operators.
\newblock {\em Trans. Amer. Math. Soc.}, 371(4):2797--2835, 2019.

\bibitem[CHMT20]{chmt}
J.~Cavero, S.~Hofmann, J.~M. Martell, and T.~Toro.
\newblock Perturbations of elliptic operators in 1-sided chord-arc domains.
  {P}art {II}: {N}on-symmetric operators and {C}arleson measure estimates.
\newblock {\em Trans. Amer. Math. Soc.}, 373(11):7901--7935, 2020.

\bibitem[CMS85]{cms}
R.~R. Coifman, Y.~Meyer, and E.~M. Stein.
\newblock Some new function spaces and their applications to harmonic analysis.
\newblock {\em J. Funct. Anal.}, 62(2):304--335, 1985.

\bibitem[Dah77]{dah1}
B.~E.~J. Dahlberg.
\newblock Estimates of harmonic measure.
\newblock {\em Arch. Rational Mech. Anal.}, 65(3):275--288, 1977.

\bibitem[Dah79]{dah2}
B.~E.~J. Dahlberg.
\newblock On the {P}oisson integral for {L}ipschitz and {$C^{1}$}-domains.
\newblock {\em Studia Math.}, 66(1):13--24, 1979.

\bibitem[Dah86]{dah4}
B.~E.~J. Dahlberg.
\newblock Poisson semigroups and singular integrals.
\newblock {\em Proc. Amer. Math. Soc.}, 97(1):41--48, 1986.

\bibitem[DFM23]{dfm22}
Z.~Dai, J.~Feneuil, and S.~Mayboroda.
\newblock Carleson perturbations for the regularity problem.
\newblock {\em Rev. Mat. Iberoam.}, 39(6):2119--2170, 2023.

\bibitem[DJ90]{dj}
G.~David and D.~Jerison.
\newblock Lipschitz approximation to hypersurfaces, harmonic measure, and
  singular integrals.
\newblock {\em Indiana Univ. Math. J.}, 39(3):831--845, 1990.

\bibitem[DS91]{ds1}
G.~David and S.~Semmes.
\newblock Singular integrals and rectifiable sets in {${\bf R}^n$}: {B}eyond
  {L}ipschitz graphs.
\newblock {\em Ast\'{e}risque}, 193:152, 1991.

\bibitem[DLM22]{dlm22}
G.~David, L.~Li, and S.~Mayboroda.
\newblock Carleson measure estimates for the {G}reen function.
\newblock {\em Arch. Ration. Mech. Anal.}, 243(3):1525--1563, 2022.

\bibitem[DM21]{dm2}
G.~David and S.~Mayboroda.
\newblock Good elliptic operators on {C}antor sets.
\newblock {\em Adv. Math.}, 383:Paper No. 107687, 21, 2021.

\bibitem[DM22]{dm22}
G.~David and S.~Mayboroda.
\newblock Approximation of {G}reen functions and domains with uniformly
  rectifiable boundaries of all dimensions.
\newblock {\em Adv. Math.}, 410:Paper No. 108717, 2022.

\bibitem[DS93]{ds2}
G.~David and S.~Semmes.
\newblock Analysis of and on uniformly rectifiable sets, volume~38 of {\em
  Mathematical Surveys and Monographs}.
\newblock American Mathematical Society, Providence, RI, 1993.

\bibitem[DHP23]{dhp22}
M.~Dindo{\v s}, S.~Hofmann, and J.~Pipher.
\newblock Regularity and {N}eumann problems for operators with real
  coefficients satisfying {C}arleson conditions.
\newblock {\em J. Funct. Anal.}, 285(6):Paper No. 110024, 32, 2023.

\bibitem[Din08]{dindos08}
M.~Dindo\v{s}.
\newblock Hardy spaces and potential theory on {$C^1$} domains in {R}iemannian
  manifolds.
\newblock {\em Mem. Amer. Math. Soc.}, 191(894):vi+78, 2008.

\bibitem[DK12]{dk12}
M.~Dindo\v{s} and J.~Kirsch.
\newblock The regularity problem for elliptic operators with boundary data in
  {H}ardy-{S}obolev space {$HS^1$}.
\newblock {\em Math. Res. Lett.}, 19(3):699--717, 2012.

\bibitem[DPP07]{dpp07}
M.~Dindo\v{s}, S.~Petermichl, and J.~Pipher.
\newblock The {$L^p$} {D}irichlet problem for second order elliptic operators
  and a {$p$}-adapted square function.
\newblock {\em J. Funct. Anal.}, 249(2):372--392, 2007.

\bibitem[DPR17]{dpr17}
M.~Dindo\v{s}, J.~Pipher, and D.~Rule.
\newblock Boundary value problems for second-order elliptic operators
  satisfying a {C}arleson condition.
\newblock {\em Comm. Pure Appl. Math.}, 70(7):1316--1365, 2017.

\bibitem[EG92]{eg1}
L.~C. Evans and R.~F. Gariepy.
\newblock Measure theory and fine properties of functions.
\newblock Studies in Advanced Mathematics. CRC Press, Boca Raton, FL, 1
  edition, 1992.

\bibitem[FMM98]{fmm98}
E.~Fabes, O.~Mendez, and M.~Mitrea.
\newblock Boundary layers on {S}obolev-{B}esov spaces and {P}oisson's equation
  for the {L}aplacian in {L}ipschitz domains.
\newblock {\em J. Funct. Anal.}, 159(2):323--368, 1998.

\bibitem[FJK84]{fjk}
E.~B. Fabes, D.~S. Jerison, and C.~E. Kenig.
\newblock Necessary and sufficient conditions for absolute continuity of
  elliptic-harmonic measure.
\newblock {\em Ann. of Math. (2)}, 119(1):121--141, 1984.

\bibitem[FKP91]{fkp}
R.~A. Fefferman, C.~E. Kenig, and J.~Pipher.
\newblock The theory of weights and the {D}irichlet problem for elliptic
  equations.
\newblock {\em Ann. of Math. (2)}, 134(1):65--124, 1991.

\bibitem[FLM24]{flm22}
J.~Feneuil, L.~Li, and S.~Mayboroda.
\newblock Green functions and smooth distances.
\newblock {\em Math. Ann.}, 389(3):2637--2727, 2024.

\bibitem[FP22]{fp}
J.~Feneuil and B.~Poggi.
\newblock Generalized {C}arleson perturbations of elliptic operators and
  applications.
\newblock {\em Trans. Amer. Math. Soc.}, 375(11):7553--7599, 2022.

\bibitem[FM12]{fm}
M.~Filoche and S.~Mayboroda.
\newblock Universal mechanism for {A}nderson and weak localization.
\newblock {\em Proc. Natl. Acad. Sci. USA}, 109(37):14761--14766, 2012.

\bibitem[GM12]{gm12}
M.~Giaquinta and L.~Martinazzi.
\newblock An introduction to the regularity theory for elliptic systems,
  harmonic maps and minimal graphs, volume~11 of {\em Appunti. Scuola Normale
  Superiore di Pisa (Nuova Serie) [Lecture Notes. Scuola Normale Superiore di
  Pisa (New Series)]}.
\newblock Edizioni della Normale, Pisa, second edition, 2012.

\bibitem[GN13]{gn13}
D.~S. Grebenkov and B.-T. Nguyen.
\newblock Geometrical structure of {L}aplacian eigenfunctions.
\newblock {\em SIAM Rev.}, 55(4):601--667, 2013.

\bibitem[GW82]{gw}
M.~Gr\"uter and K.-O. Widman.
\newblock The {G}reen function for uniformly elliptic equations.
\newblock {\em Manuscripta Math.}, 37(3):303--342, 1982.

\bibitem[Ha96]{haj96}
P.~Haj\l~asz.
\newblock Sobolev spaces on an arbitrary metric space.
\newblock {\em Potential Anal.}, 5(4):403--415, 1996.

\bibitem[HT02]{ht02}
A.~Hassell and T.~Tao.
\newblock Upper and lower bounds for normal derivatives of {D}irichlet
  eigenfunctions.
\newblock {\em Math. Res. Lett.}, 9(2-3):289--305, 2002.

\bibitem[HKM93]{hkm93}
J.~Heinonen, T.~Kilpel\"{a}inen, and O.~Martio.
\newblock Nonlinear potential theory of degenerate elliptic equations.
\newblock Oxford Mathematical Monographs. The Clarendon Press, Oxford
  University Press, New York, 1993.
\newblock Oxford Science Publications.

\bibitem[HKMP15a]{hkmpr}
S.~Hofmann, C.~Kenig, S.~Mayboroda, and J.~Pipher.
\newblock The regularity problem for second order elliptic operators with
  complex-valued bounded measurable coefficients.
\newblock {\em Math. Ann.}, 361(3-4):863--907, 2015.

\bibitem[HKMP15b]{hkmps}
S.~Hofmann, C.~Kenig, S.~Mayboroda, and J.~Pipher.
\newblock Square function/non-tangential maximal function estimates and the
  {D}irichlet problem for non-symmetric elliptic operators.
\newblock {\em J. Amer. Math. Soc.}, 28(2):483--529, 2015.

\bibitem[HK07]{hk}
S.~Hofmann and S.~Kim.
\newblock The {G}reen function estimates for strongly elliptic systems of
  second order.
\newblock {\em Manuscripta Math.}, 124(2):139--172, 2007.

\bibitem[HLM02]{hlm02}
S.~Hofmann, M.~Lacey, and A.~McIntosh.
\newblock The solution of the {K}ato problem for divergence form elliptic
  operators with {G}aussian heat kernel bounds.
\newblock {\em Ann. of Math. (2)}, 156(2):623--631, 2002.

\bibitem[HL18]{hl2018}
S.~Hofmann and P.~Le.
\newblock B{MO} solvability and absolute continuity of harmonic measure.
\newblock {\em J. Geom. Anal.}, 28(4):3278--3299, 2018.

\bibitem[HMM15]{hmm}
S.~Hofmann, S.~Mayboroda, and M.~Mourgoglou.
\newblock Layer potentials and boundary value problems for elliptic equations
  with complex {$L^\infty$} coefficients satisfying the small {C}arleson
  measure norm condition.
\newblock {\em Adv. Math.}, 270:480--564, 2015.

\bibitem[HR13]{hr13}
T.~Hyt\"{o}nen and A.~Ros\'{e}n.
\newblock On the {C}arleson duality.
\newblock {\em Ark. Mat.}, 51(2):293--313, 2013.

\bibitem[JK95]{jk4}
D.~Jerison and C.~E. Kenig.
\newblock The inhomogeneous {D}irichlet problem in {L}ipschitz domains.
\newblock {\em J. Funct. Anal.}, 130(1):161--219, 1995.

\bibitem[JK81]{jk1}
D.~S. Jerison and C.~E. Kenig.
\newblock The {D}irichlet problem in nonsmooth domains.
\newblock {\em Ann. of Math. (2)}, 113(2):367--382, 1981.

\bibitem[JK82a]{jk3}
D.~S. Jerison and C.~E. Kenig.
\newblock Boundary behavior of harmonic functions in nontangentially accessible
  domains.
\newblock {\em Adv. in Math.}, 46(1):80--147, 1982.

\bibitem[JK82b]{jk5}
D.~S. Jerison and C.~E. Kenig.
\newblock Hardy spaces, {$A_{\infty }$}, and singular integrals on chord-arc
  domains.
\newblock {\em Math. Scand.}, 50(2):221--247, 1982.

\bibitem[KKPT00]{kkpt}
C.~Kenig, H.~Koch, J.~Pipher, and T.~Toro.
\newblock A new approach to absolute continuity of elliptic measure, with
  applications to non-symmetric equations.
\newblock {\em Adv. Math.}, 153(2):231--298, 2000.

\bibitem[KLS13]{kls13}
C.~Kenig, F.~Lin, and Z.~Shen.
\newblock Estimates of eigenvalues and eigenfunctions in periodic
  homogenization.
\newblock {\em J. Eur. Math. Soc. (JEMS)}, 15(5):1901--1925, 2013.

\bibitem[Ken94]{kbook}
C.~E. Kenig.
\newblock Harmonic analysis techniques for second order elliptic boundary value
  problems, volume~83 of {\em CBMS Regional Conference Series in Mathematics}.
\newblock Published for the Conference Board of the Mathematical Sciences,
  Washington, DC; by the American Mathematical Society, Providence, RI, 1994.

\bibitem[KP93]{kp}
C.~E. Kenig and J.~Pipher.
\newblock The {N}eumann problem for elliptic equations with nonsmooth
  coefficients.
\newblock {\em Invent. Math.}, 113(3):447--509, 1993.

\bibitem[KP95]{kp2}
C.~E. Kenig and J.~Pipher.
\newblock The {N}eumann problem for elliptic equations with nonsmooth
  coefficients. {II}.
\newblock {\em Duke Math. J.}, 81(1):227--250 (1996), 1995.
\newblock A celebration of John F. Nash, Jr.

\bibitem[KP01]{kp3}
C.~E. Kenig and J.~Pipher.
\newblock The {D}irichlet problem for elliptic equations with drift terms.
\newblock {\em Publ. Mat.}, 45(1):199--217, 2001.

\bibitem[KR09]{kr}
C.~E. Kenig and D.~J. Rule.
\newblock The regularity and {N}eumann problem for non-symmetric elliptic
  operators.
\newblock {\em Trans. Amer. Math. Soc.}, 361(1):125--160, 2009.

\bibitem[LP95]{lp95}
M.~L. Lapidus and M.~M.~H. Pang.
\newblock Eigenfunctions of the {K}och snowflake domain.
\newblock {\em Comm. Math. Phys.}, 172(2):359--376, 1995.

\bibitem[LSW63]{lsw}
W.~Littman, G.~Stampacchia, and H.~F. Weinberger.
\newblock Regular points for elliptic equations with discontinuous
  coefficients.
\newblock {\em Ann. Scuola Norm. Sup. Pisa Cl. Sci. (3)}, 17:43--77, 1963.

\bibitem[MP21]{mp20}
S.~Mayboroda and B.~Poggi.
\newblock Carleson perturbations of elliptic operators on domains with low
  dimensional boundaries.
\newblock {\em J. Funct. Anal.}, 280(8):108930, 91, 2021.

\bibitem[May10]{may10}
S.~Mayboroda.
\newblock The connections between {D}irichlet, regularity and {N}eumann
  problems for second order elliptic operators with complex bounded measurable
  coefficients.
\newblock {\em Adv. Math.}, 225(4):1786--1819, 2010.

\bibitem[MPT14]{mpt14}
E.~Milakis, J.~Pipher, and T.~Toro.
\newblock Perturbations of elliptic operators in chord arc domains.
\newblock In {\em Harmonic analysis and partial differential equations}, volume
  612 of {\em Contemp. Math.}, pages 143--161. Amer. Math. Soc., Providence,
  RI, 2014.

\bibitem[MM11]{mm11}
D.~Mitrea and I.~Mitrea.
\newblock On the regularity of {G}reen functions in {L}ipschitz domains.
\newblock {\em Comm. Partial Differential Equations}, 36(2):304--327, 2011.

\bibitem[MM07]{mm07}
I.~Mitrea and M.~Mitrea.
\newblock The {P}oisson problem with mixed boundary conditions in {S}obolev and
  {B}esov spaces in non-smooth domains.
\newblock {\em Trans. Amer. Math. Soc.}, 359(9):4143--4182, 2007.

\bibitem[MM81]{mm}
L.~Modica and S.~Mortola.
\newblock Construction of a singular elliptic-harmonic measure.
\newblock {\em Manuscripta Math.}, 33(1):81--98, 1980/81.

\bibitem[Mou23]{mourg19}
M.~Mourgoglou.
\newblock Regularity theory and {G}reen's function for elliptic equations with
  lower order terms in unbounded domains.
\newblock {\em Calc. Var. Partial Differential Equations}, 62(9):Paper No. 266,
  69, 2023.

\bibitem[MPT]{mpt22}
M.~Mourgoglou, B.~Poggi, and X.~Tolsa.
\newblock ${L}^p$-solvability of the {P}oisson-{D}irichlet problem and its
  applications to the regularity problem.
\newblock Preprint. July 2022 (v1). arXiv: 2207.10554.

\bibitem[MT24]{mt22}
M.~Mourgoglou and X.~Tolsa.
\newblock The regularity problem for the {L}aplace equation in rough domains.
\newblock {\em Duke Math. J.}, 173(9):1731--1837, 2024.

\bibitem[Pog24]{pog21}
B.~Poggi.
\newblock Applications of the landscape function for {S}chr\"odinger operators
  with singular potentials and irregular magnetic fields.
\newblock {\em Adv. Math.}, 445:Paper No. 109665, 54, 2024.

\bibitem[She07]{shen07}
Z.~Shen.
\newblock A relationship between the {D}irichlet and regularity problems for
  elliptic equations.
\newblock {\em Math. Res. Lett.}, 14(2):205--213, 2007.

\bibitem[SX13]{sx13}
Y.~Shi and B.~Xu.
\newblock Gradient estimate of a {D}irichlet eigenfunction on a compact
  manifold with boundary.
\newblock {\em Forum Math.}, 25(2):229--240, 2013.

\bibitem[TT24]{tt22}
O.~Tapiola and X.~Tolsa.
\newblock Connectivity conditions and boundary {P}oincar\'e{} inequalities.
\newblock {\em Anal. PDE}, 17(5):1831--1870, 2024.

\bibitem[Tol14]{tolsa14}
X.~Tolsa.
\newblock Analytic capacity, the {C}auchy transform, and non-homogeneous
  {C}alder\'{o}n-{Z}ygmund theory, volume 307 of {\em Progress in Mathematics}.
\newblock Birkh\"{a}user/Springer, Cham, 2014.

\bibitem[vdBB99]{vdbb99}
M.~van~den Berg and E.~Bolthausen.
\newblock Estimates for {D}irichlet eigenfunctions.
\newblock {\em J. London Math. Soc. (2)}, 59(2):607--619, 1999.

\end{thebibliography}

\end{document}